\documentclass[11pt]{amsart}

\usepackage{amssymb}
\usepackage{amsmath}
\usepackage{amsthm}
\usepackage{mathrsfs}
\usepackage{dsfont}
\usepackage{enumitem}
\usepackage[left=2.7cm,right=2.7cm,top=2.2cm,bottom=2.2cm]{geometry}
\usepackage{mathtools}
\usepackage{hyperref}
\hypersetup{
	colorlinks=true,
	linkcolor=red
}
\usepackage{xcolor}

\numberwithin{equation}{section}

\newtheorem{Theorem}[equation]{Theorem}
\newtheorem{Lemma}[equation]{Lemma}
\newtheorem{Corollary}[equation]{Corollary}
\newtheorem{Proposition}[equation]{Proposition}

\newtheorem{Remark}[equation]{Remark}
\theoremstyle{definition}
\newtheorem{Definition}[equation]{Definition}

\setlist{nosep}

\title{Pseudolocality of \texorpdfstring{$G_2$}{G2}-Laplacian Flow}

\author{Weiyong He}
\address{Department of Mathematics, University of Oregon,
Eugene, OR 97403, USA}
\email{whe@uoregon.edu}

\author{Longhuan Yan}
\address{School of Mathematical Sciences,
University of Science and Technology of China,
No. 96, Jinzhai Road, Hefei, Anhui, China}
\email{yanlonghuan@mail.ustc.edu.cn}
\date{\today}

\begin{document}

\begin{abstract}
	In this paper, we give a new proof of Perelman's pseudolocality theorem for Ricci flow, which does not rely on Perelman's Harnack-type inequality.
    The key step is an estimate relating initial local $\mu$-entropy to pointed Nash entropy at later spacetime points.
	The argument also works for other reasonable geometric flows, like the $G_2$-Laplacian flow, at a scale depending on its deviation from Ricci flow and its scalar curvature lower bound.
	As a consequence, we can also prove the quantitative rigidity results concerning pointed Nash entropy due to Hein--Naber and Bamler for such geometric flows.
    
    In the appendix, we also show that the other results in \cite{bamler2020entropy}, when suitably modified, remain valid for Laplacian flow with bounded torsion. These include log-Sobolev/Poincar\'{e} inequality for conjugate heat kernel measure, heat kernel estimates, and dependence of the pointed Nash entropy on its basepoint.
\end{abstract}


\maketitle
\tableofcontents
\bibliographystyle{amsalpha}

\newcommand{\C}{{\mathbb C}}
\newcommand{\Q}{{\mathbb Q}}
\newcommand{\R}{{\mathbb R}}
\newcommand{\Z}{{\mathbb Z}}
\newcommand{\N}{{\mathbb N}}
\newcommand*\norm[1]{\left\lVert#1\right\rVert}
\newcommand*\normm[1]{\lVert#1\rVert}
\newcommand*\abs[1]{\left\lvert#1\right\rvert}
\newcommand*\abss[1]{\lvert#1\rvert}
\newcommand*\inn[2]{\langle#1,#2\rangle}
\renewcommand{\geq}{\geqslant}
\renewcommand{\leq}{\leqslant}

\newcommand{\tr}{{\mathrm{tr}}}
\newcommand{\Ric}{{\mathrm{Ric}}}
\newcommand{\Rm}{{\mathrm{Rm}}}

\section{Introduction}
The primary goal of this paper is to study the $G_2$ Laplacian flow with bounded torsion. We refer to the foundational paper for the $G_2$ Laplacian flow Lotay--Wei \cite{lotay2017laplacian} for the background and references. In particular, Lotay--Wei proved that the flow extends once the Ricci curvature is uniformly bounded. 
In the same paper the authors raised the question  on compact manifolds, whether the flow extends once the torsion, or equivalently the scalar curvature remains uniformly bounded. The first step toward this is given by Chen in \cite{chen2018shi}, where he proved a no local collapsing theorem by using Perelman's $\mathcal{W}$-functional. To some extend, \cite{chen2018shi} shed light on the possibility of carrying over the entropy approach for Ricci flow to our setting of Laplacian flow with bounded torsion. In this paper, we will pursue this idea and prove pseudolocality for a variant of Ricci flow, which contains the Laplacian flow with bounded torsion as a special case.

\subsection{Pseudolocality}
After fixing the diffeomorphism invariance, the Ricci flow is a nonlinear heat-type equation.
As for the heat equation, perturbations propagate with infinite speed, and therefore the evolution of the metric on a region is not determined solely by the initial geometry of that region.
Nevertheless, the effect of distant geometry should be weak over short time intervals, and pseudolocality is a quantitative substitute for locality in this setting.
Roughly speaking, it asserts that a region which is sufficiently regular at the initial time cannot immediately develop large curvature in a slightly smaller region, even when no
curvature bound is imposed outside the original region.

This principle was introduced by Perelman in Section~10 of \cite{perelman2002entropy}.
It has since been extended and reformulated in several directions.
Chau--Tam--Yu extended the argument to complete noncompact Ricci flows in \cite{chau2011pseudolocality}.
A particularly useful development is Bing Wang's theory of local entropy.
In \cite{wang2018local,wang2020local}, Wang localized Perelman's entropy functionals and obtained a more flexible pseudolocality theorem.
Besides, Bamler's pointed Nash entropy and heat-kernel estimates have provided a closely related analytic framework for local regularity and $\epsilon$-regularity \cite{bamler2020entropy}.

In this paper, we will combine the ideas of \cite{wang2018local,wang2020local} and \cite{bamler2020entropy} to generalize the pseudolocality theorem to reasonable geometric flows (by which we mean those admitting Shi type estimates and compactness theorem) over a scale depending only on $\sup|\partial_tg+2\mathrm{Ric}_g|$, $\inf R_g$ and the dimension.
Besides, we will show that the results in \cite{bamler2020entropy}, when suitably modified, remain valid in this setting.
These include log-Sobolev/Poincar\'{e} inequality for conjugate heat kernel measure, heat kernel estimates, dependence of the pointed Nash entropy on its basepoint, and the $\epsilon$-regularity theorem.
The philosophy is that flows satisfying $\partial_tg+2\Ric_g\geq-Ag$ for some $A\geq 0$ (resp. $|\partial_tg+2\Ric_g|\leq Cg$ for some $C\geq 0$) behave like super Ricci flow (resp. Ricci flow) over a small but definite spacetime scale.

\subsection{Laplacian flow}\label{sec:introduction-Laplacian-flow}
Let $M$ be a $7$-manifold. We call $\varphi\in\Omega^3(M)$ positive if the bilinear form $B_\varphi$ defined by
\begin{equation*}
	B_\varphi(u, v)
	=\frac{1}{6}(u\lrcorner\varphi)\wedge(v\lrcorner\varphi)\wedge \varphi,
	\qquad\forall\,x\in M,\ u,v\in T_xM
\end{equation*}
is positive definite. By definition, a $G_2$-structure on $M$ is a positive $3$-form. Any $G_2$-structure $\varphi$ will induce an orientation and a unique Riemannian metric $g_\varphi$ on $M$. Let $*_\varphi$ be the induced Hodge star operator, $\psi:=*_\varphi\varphi$, and $\nabla$ be the Levi-Civita connection associated to $g_\varphi$. One can show that there exists a unique $2$-tensor $T_\varphi$ satisfying $\nabla_X\varphi=T_\varphi(X)\lrcorner\psi$ for all vector field $X$. This $T_\varphi$ is called the torsion tensor of $\varphi$. When $T_\varphi=0$, the $G_2$-structure $\varphi$ is called torsion-free.

For any torsion-free $G_2$-structure $\varphi$, $\mathrm{Hol}(g_\varphi)\subset G_2$ and hence $g_\varphi$ is Ricci-flat. If in addition $M$ is compact, then $\mathrm{Hol}(g_\varphi)=G_2$ if and only if $\pi_1(M)$ is finite. Thus finding torsion-free $G_2$ structures is essential for constructing compact manifolds with holonomy $G_2$. For this purpose, Bryant introduced in \cite{bryant2003some} the Laplacian flow for closed $G_2$-structure (over a compact $7$-manifold)
\begin{equation*}
	\begin{cases*}
		\frac{\partial}{\partial t}\varphi
		=\Delta_\varphi\varphi,\\
		d\varphi=0,\\
		\varphi(0)=\varphi_0.
	\end{cases*}
\end{equation*}
Here $\Delta_\varphi=dd^*_\varphi+d^*_\varphi d$ is the Hodge Laplacian operator. The short-time existence and uniqueness is proved in \cite{bryant2011laplacian}. Later, Lotay and Wei developed the foundational theory for the Laplacian flow in \cite{lotay2017laplacian}. To mention a few of their results, they established long-time existence, Shi-type derivative estimates for $\Rm(g_\varphi)$ and a compactness theorem for the flow. Moreover, they proved $|\Ric|(g_\varphi)$ must blow up at a finite-time singularity, paralleling \v{S}e\v{s}um's result for Ricci flow in \cite{vsevsum2005curvature}. Motivated by the extension problem in Ricci flow, Lotay and Wei asked whether the scalar curvature (or equivalently, the norm of torsion $|T_\varphi|$) must blow up at a finite-time singularity.

\subsection{Outline}
In the appendices, we recall some basic definitions and results needed for our discussion of pseudolocality and entropy rigidity.
The difficulty in proving pseudolocality for Laplacian flow is explained in \hyperref[sec:preliminaries-difficulty]{Section~\ref*{sec:preliminaries-difficulty}}.
We generalize the results in \cite{bamler2020entropy} (except the rigidity results of pointed Nash entropy) from (super) Ricci flow to smooth family of Riemannian metrics satisfying $\partial_tg+2\Ric_g\geq-Ag$ or $|\partial_tg+2\Ric_g|\leq C$, depending on the very situation.
The proofs are minor modifications of those in \cite{bamler2020entropy}, but we repeat most of them in detail for completeness and for the convenience of the readers.

In \hyperref[sec:pseudolocality]{Section~\ref*{sec:pseudolocality}}, we prove our pseudolocality theorem for Laplacian flow.
The method is closely related to the strategy in \cite{wang2018local,wang2020local}.
But there is one crucial difference between our case and Ricci flow, namely that we lack an effective monotonicity formula for the local $\mu$-functional.
To remedy this, we prove a new bound on pointed Nash entropy by local $\mu$-entropy (under a suitable Ricci curvature bound), and then reduce $\mu$-entropy rigidity problems to those for pointed Nash entropy.
This argument will in particular give a new proof of the pseudolocality theorem in \cite{wang2020local}, assuming a scalar curvature lower bound which is natural for Ricci flow.

In \hyperref[sec:rigidity-of-pointed-Nash-entropy]{Section~\ref*{sec:rigidity-of-pointed-Nash-entropy}}, we prove that the $\epsilon$-regularity theorem in \cite{bamler2020entropy} remains valid for Laplacian flow with bounded torsion, provided we restrict ourselves to a sufficiently small (but definite) spacetime scale.

\subsection{Notation and convention}
In this paper $I$ will always denote an interval on which the flow $(g(t))$ of Riemannian metrics is defined, and
\begin{equation}\label{eq:definition-of-error-term}
	E(g):=\partial_tg+2\Ric_g
\end{equation}
is its deviation from Ricci flow.
We write $|E|_{\mathrm{max}}$ for $\sup_{M\times I}|E|$.
Whenever it appears, we are implicitly assuming it is finite, i.e., $\sup_{M\times I}|E|<\infty$.
Similar conventions are made for other quantities like $R_{\mathrm{min}}$, etc.
In the presence of $|E|_{\mathrm{max}}$, we have $E\geq-Ag$ for $A=|E|_{\mathrm{max}}$.

We adopt the sign convention $\delta=-\mathrm{tr}_{12}\nabla$ for the divergence operator. For a set $X$, we write $\mathds{1}_X$ for the indicator function of $X$, i.e., $\mathds{1}_X(x)=1$ if $x\in X$ and $\mathds{1}_X(x)=0$ otherwise.

Here are the definition of some constants. We set
\begin{equation*}
	H_n:=\frac{(n-1)\pi^2}{2}+4,
	\qquad
	\alpha_n:=\frac{1}{100n}.
\end{equation*}
Let $\omega_n$ denote the volume of the unit ball in $\R^n$.

When passing to a subsequence, we will use the same symbol to label it.

We use the notation $C(a\mid b)$ to denote a generic function that satisfies $\lim_{a\to 0}C(a\mid b)=0$ for fixed $b$.

\subsection{Acknowledgments}
Both authors thank Professor Bing Wang for helpful discussions and valuable suggestions.
The second author is grateful to him for his guidance and constant support.

\subsection{Declaration}
This paper was completed in December 2024.
No AI tools were used in its preparation.
A related backward pseudolocality result is obtain with the help of AI, and it is separated into a companion paper \cite{he2026backward}.

\section{Pseudolocality}\label{sec:pseudolocality}
In this section we prove our generalized pseudolocality theorem.
Though stated only for $G_2$-Laplacian flow, the result actually holds for any geometric flow admitting Shi type estimates and compactness theorem, at the scale
\begin{equation*}
	\left(  \max\Big\{  \sup_{M\times I}|E(g)|,-\inf_{M\times I}R_g  \Big\}  \right)^{-1/2}.
\end{equation*}
The basic idea is very simple: smallness (by small we mean almost non-negative) of initial local $\mu$-entropy should imply smallness of pointed Nash entropy in the near future, and the rigidity of pointed Nash entropy then gives us the desired curvature bound.

\subsection{Bounding pointed Nash entropy by local \texorpdfstring{$\mu$}{mu}-entropy}
We shall start from the following lemma that relates pointed Nash entropy and local $\mu$-entropy.
One can compare this with \cite[Theorem 5.1]{wang2018local}.

\begin{Lemma}\label{lem:mu-and-N}
	Let $M^n$ be a compact manifold and $(g(t))_{t\in[0,T]}$ be a smooth family of Riemannian metrics over $M$.
	Assume $R\geq R_{\mathrm{min}}$ and $|E|\leq |E|_{\mathrm{max}}$ over $M\times[0,T]$. Given a point $x_0\in M$, define $u:=K(x_0,T;\cdot,\cdot)$ and $d\nu:=u\,dV_g$.
	Let $\Omega_0'\subset\Omega_0$ be two domains in $M$ and $h:M\to[0,1]$ be a cut-off function supported in $\Omega_0$ with $h\equiv 1$ over $\Omega_0'$.
	Set $S=\int_{\Omega_0}hu\,dV_g(0)$.
	Then
	\begin{align*}
		\mu(\Omega_0,g(0),T)
		&\leq
		\frac{\nu(\Omega_0')}{S}\mathcal{N}_{x_0,T}(T)
		+\frac{1}{S}(C_1+C_3\nu(M\setminus\Omega_0'))\\
		&\quad+\frac{1}{S}
		\bigg(4T\sup_{\Omega_0\setminus\Omega_0'}   \big|\nabla\sqrt{h}\big|^2+e^{-1}\bigg)
		\nu(\Omega_0\setminus\Omega_0')\\
		&\quad+\frac{1}{S}
		C_4\cdot
		\bigg(\sqrt{T}\sup_{\Omega_0\setminus\Omega_0'}|\nabla h|\bigg)
		\nu(\Omega_0\setminus\Omega_0')^{1/2}.
	\end{align*}
	Here $C_1=C_1(n,|E|_{\mathrm{max}}T)$ and $C_3,C_4$ depend only on $n,R_{\mathrm{min}}T,|E|_{\mathrm{max}}T$.
\end{Lemma}
\begin{proof}
	Define $\tilde{u}:=hu/S$, $\tau(t):=T-t$ and $u=:(4\pi\tau)^{-n/2}e^{-f}$. Then
	\begin{align*}
		\mu(\Omega_0,g(0),\tau(0))
		\leq   \mathcal{W}(\Omega_0,g(0),\sqrt{\tilde{u}},\tau(0))
		=&\,\log S
		+\frac{1}{S}\int_{\Omega_0}   \{\tau(|\nabla f|^2+R)+f-n\}hu\,dV_g  \\
		&+\frac{1}{S}\int_{\Omega_0}   \{\tau|\nabla\log h|^2-2\tau\nabla f\cdot\nabla\log h-\log h\}hu\,dV_g,
	\end{align*}
	where the last expression is understood to be computed at time $t=0$. We estimate this term by term as follows. First we note $S\leq\int_{\Omega_0}u\,dV_g(0)=1$ and therefore $\log S\leq 0$. By \hyperref[cor:N-and-W]{Corollary~\ref*{cor:N-and-W}},
	\begin{equation*}
		\mathcal{W}(M,g(0),f(0),\tau(0))
		\leq
		\mathcal{N}_{x_0,T}(T)+C_1(|E|_{\mathrm{max}}T\mid n).
	\end{equation*}
	By \hyperref[thm:heat-kernel-upper-bound]{Theorem~\ref*{thm:heat-kernel-upper-bound}},
	\begin{equation*}
		f(\cdot,0)\geq\mathcal{N}_{x_0,T}(T)-C_2(n,R_{\mathrm{min}}T,|E|_{\mathrm{max}}T).
	\end{equation*}
	These two inequalities allow us to estimate the second term as
	\begin{align*}
		&\frac{1}{S}\int_{\Omega_0}   \{\tau(|\nabla f|^2+R)+f-n\}hu\,dV_g\\
		=\,&
		\frac{1}{S}\int_{M}   \{\tau(|\nabla f|^2+R)+f-n\}u\,dV_g
		-
		\frac{1}{S}\int_{M}   \{\tau(|\nabla f|^2+R)+f-n\}(1-h)u\,dV_g\\
		=\,&
		\frac{1}{S}\mathcal{W}(M,g(0),f(0),\tau(0))
		-
		\frac{1}{S}\int_{M}   \{\tau(|\nabla f|^2+R)+f-n\}(1-h)u\,dV_g\\
		\leq\,&
		\frac{1}{S}(\mathcal{N}_{x_0,T}(T)+C_1)
		+\frac{1}{S}   \left(
		-\min\{R_{\mathrm{min}}T,0\}
		+C_2-\mathcal{N}_{x_0,T}(T)
		+n
		\right)   \nu(M\setminus\Omega_0')\\
		=\,&
		\frac{1}{S}(\mathcal{N}_{x_0,T}(T)+C_1)
		+\frac{1}{S}
		\left(   C_3-\mathcal{N}_{x_0,T}(T)   \right) \nu(M\setminus\Omega_0')\\
		=\,&
		\frac{1}{S}(1-\nu(M\setminus\Omega_0'))\mathcal{N}_{x_0,T}(T)
		+\frac{1}{S}(C_1+C_3\nu(M\setminus\Omega_0')),
	\end{align*}
	where $C_3:=C_2-\min\{R_{\mathrm{min}}T,0\}+n$. Since
	$1-\nu(M\setminus\Omega_0')=\nu(\Omega_0')$, we get
	\begin{equation*}
		\frac{1}{S}\int_{\Omega_0}   \{\tau(|\nabla f|^2+R)+f-n\}hu\,dV_g
		\leq
		\frac{\nu(\Omega_0')}{S}\mathcal{N}_{x_0,T}(T)
		+\frac{1}{S}(C_1+C_3\nu(M\setminus\Omega_0')).
	\end{equation*}
	Next we deal with the last term.
	Note
	\begin{equation*}
		\frac{1}{S}\int_{\Omega_0}   \tau|\nabla\log h|^2   hu\,dV_g
		=\frac{1}{S}\int_{\Omega_0}   4\tau|\nabla\sqrt{h}|^2   d\nu\\
		\leq\frac{4T}{S}\bigg(  \sup_{\Omega_0\setminus\Omega_0'}|\nabla\sqrt{h}|^2  \bigg)
		\nu(\Omega_0\setminus\Omega_0').
	\end{equation*}
	By \hyperref[prop:energy-estimate]{Proposition~\ref*{prop:energy-estimate}},
	\begin{align*}
		\frac{1}{S}\int_{\Omega_0}   -2\tau(\nabla f\cdot\nabla\log h)   hu\,dV_g
		&=
		\frac{1}{S}\int_{\Omega_0}   -2\tau(\nabla f\cdot\nabla h)   d\nu\\
		&\leq
		\frac{2T}{S}\bigg(\int_{M}|\nabla f|^2d\nu\bigg)^{1/2}
		\bigg(\int_{\Omega_0\setminus\Omega_0'}|\nabla h|^2d\nu\bigg)^{1/2}\\
		&\leq
		\frac{\sqrt{T}}{S}
		C_4(n,R_{\mathrm{min}}T,|E|_{\mathrm{max}}T)
		\bigg(\sup_{\Omega_0\setminus\Omega_0'}|\nabla h|\bigg)
		\nu(\Omega_0\setminus\Omega_0')^{1/2}.
	\end{align*}
	Finally,
	\begin{equation*}
		\frac{1}{S}\int_{\Omega_0}   -(\log h)   hu\,dV_g
		=
		\frac{1}{S}\int_{\Omega_0}   -h\log h   \,d\nu
		\leq
		\frac{e^{-1}}{S}   \nu(\Omega_0\setminus\Omega_0').
	\end{equation*}
	The result follows by combining all these estimates together.
\end{proof}

When $\nu_{x_0,T;0}$ is concentrated near $x_0$ at time $t=0$ and the time interval $[0,T]$ is not too long, the above lemma will give an effective estimate on $\mathcal{N}_{x_0,T}(T)$ in term of $\mu(\Omega_0,g(0),T)$ for some neighborhood $\Omega_0$ of $x_0$.

\begin{Theorem}\label{thm:mu-and-N}
	Let $M^n$ be a compact manifold and $(g(t))_{t\in[0,T]}$ be a smooth family of Riemannian metrics over $M$ with $E(g)\geq-Ag$ for some $A\geq 0$.
	Assume $R\geq R_{\mathrm{min}}$ and $|E|\leq |E|_{\mathrm{max}}$ over $M\times[0,T]$.
	Suppose at $x_0\in M$ we have
	\begin{equation*}
		\Ric(x,t)\leq\frac{A'}{\beta^2\bar{\sigma}_{A}(t)},
		\quad
		\forall\,x\in B\big(x_0,t,\beta\sqrt{\bar{\sigma}_{A}(t)}\big),
		\ t\in(0,T]
	\end{equation*}
	for some $A'>0$ and $\beta\in(0,1]$.
	Set $A'':=(\frac{2}{3}A'+n-1)\beta^{-1}$ and $\Omega_0:=B\big(x_0,0,20A''\sqrt{\sigma_A(T)}\big)$.
	Then
	\begin{equation*}
		\mathcal{N}_{x_0,T}(T)
		\geq e^{2/5}\min\{\mu(\Omega_0,g(0),T),0\}
		-C_6(n,R_{\mathrm{min}}T,|E|_{\mathrm{max}}T)(A'')^{-1}
		-C_7(|E|_{\mathrm{max}}T\mid n).
	\end{equation*}
\end{Theorem}
\begin{proof}
	Clearly, $A''\geq 1$.
	Let $h$ be the cut-off function constructed in the proof of \hyperref[thm:concentration-under-Ricci-bound]{Theorem~\ref*{thm:concentration-under-Ricci-bound}}.
	By \hyperref[rmk:gradient-bound-on-cut-off-function]{Remark~\ref*{rmk:gradient-bound-on-cut-off-function}} and the fact $T\leq\sigma_A(T)$, we have
	\begin{equation*}
		4T\sup_{\Omega_0\setminus\Omega_0'}   \big|\nabla\sqrt{h}\big|^2+e^{-1}
		\leq\frac{T}{10(A'')^2\sigma_A(T)}+e^{-1}
		\leq 1,
	\end{equation*}
	\begin{equation*}
		\sqrt{T}\sup_{\Omega_0\setminus\Omega_0'}|\nabla h|
		\leq
		\frac{\sqrt{T}}{A''\sqrt{\sigma_A(T)}}
		\leq 1.
	\end{equation*}
	Then \hyperref[lem:mu-and-N]{Lemma~\ref*{lem:mu-and-N}} implies
	\begin{align*}
		\mu(\Omega_0,g(0),T)
		&\leq
		\frac{\nu(\Omega_0')}{S}\mathcal{N}_{x_0,T}(T)
		+\frac{1}{S}(C_1+C_3\nu(M\setminus\Omega_0'))
		+\frac{1}{S}\nu(\Omega_0\setminus\Omega_0')
		+\frac{1}{S}C_4\nu(\Omega_0\setminus\Omega_0')^{1/2}\\
		&\leq
		\frac{\nu(\Omega_0')}{S}\mathcal{N}_{x_0,T}(T)
		+\frac{1}{S}C_1
		+\frac{1}{S}C_5\nu(M\setminus\Omega_0')^{1/2},
	\end{align*}
	where $C_1=C_1(|E|_{\mathrm{max}}T\mid n)$, and $C_5:=1+C_3+C_4$ depends only on $n,R_{\mathrm{min}}T,|E|_{\mathrm{max}}T$.
	Now \hyperref[rmk:concentration]{Remark~\ref*{rmk:concentration}} gives
	\begin{equation*}
		1\geq S  \geq\nu(\Omega_0')  \geq\exp\left(-\frac{2}{5(A'')^2}\right)  \geq e^{-2/5}.
	\end{equation*}
	Finally, using $1-e^{-x}\leq x$, we obtain
	\begin{equation*}
		\nu(M\setminus\Omega_0')
		=1-\nu(\Omega_0')
		\leq\frac{2}{5(A'')^2}.
	\end{equation*}
\end{proof}

\subsection{Stability and rigidity of entropy}
We record two elementary stability and rigidity facts that will be used in the blow-up arguments below.
The first is a local consequence of the scale invariance of the $\mu$-functional.
\begin{Lemma}[Stability of local $\mu$-entropy near Euclidean space]\label{lem:local-mu-euclidean-stability}
	Suppose $(M_i,g_i,x_i)$ converges smoothly to $(\mathbb{R}^n,g_{\mathrm{E}},0)$.
	Then, for every $L,T>0$,
	\begin{equation*}
		\liminf_{i\to\infty}\inf_{0<\tau\leq T}
		\mu\big(B_{g_i}(p_i,L\sqrt{\tau}),g_i,\tau\big)\geq 0.
	\end{equation*}
\end{Lemma}
\begin{proof}
	If the assertion were false, then after passing to a subsequence there would be $\tau_i\in(0,T]$ such that
	\begin{equation*}
		\mu\big(B_{g_i}(p_i,L\sqrt{\tau}),g_i,\tau\big)<-\epsilon
	\end{equation*}
	for some $\epsilon>0$.
	If $\tau_i$ has a positive limit, the assertion follows from smooth convergence and continuity of the $\mu$-functional \cite[Corollary 2.5]{wang2018local}.
	If $\tau_i\to0$, rescale $g_i$ by $\tau_i^{-1}$.
	On the ball of radius $L$ the rescaled metrics converge smoothly to the Euclidean metric.
	Scale invariance and the same continuity give
	\begin{equation*}
		\lim_{i\to\infty}
		\mu\big(B_{g_i}(p_i,L\sqrt{\tau_i}),g_i,\tau_i\big)
		=\mu\big(B_{g_{\mathrm E}}(0,L),g_{\mathrm E},1\big)\geq0,
	\end{equation*}
	where the last inequality follows by domain monotonicity from
	$\mu(\mathbb R^n,g_{\mathrm E},1)=0$.
\end{proof}

The next lemma isolates the entropy-rigidity part of the proof of \cite[Theorem 10.2]{bamler2020entropy}.
\begin{Lemma}[Pointed entropy rigidity]\label{lem:pointed-entropy-rigidity}
	Let $(M_i^7,\varphi_i(t),g_i(t),x_i)_{t\in[0,1]}$ be a sequence of compact Laplacian flows satisfying
	\begin{equation*}
		\sup_{M_i\times[0,1]}|T_{\varphi_i}|^2\leq\lambda_i
		\qquad\text{and}\qquad
		|\Rm|(g_i)\leq\Lambda
		\text{ over }B_{g_i(t)}(x_i,\rho_i),\forall\,t\in[0,1]
	\end{equation*}
	for some $\lambda_i\to0$, $\Lambda>0$ and $\rho_i\to\infty$.
	Suppose $\mathcal N_{x_i,1}(1)\to0$, then every smooth subsequence limit is the static Euclidean flow.
\end{Lemma}
\begin{proof}
	The uniform bounds on curvature and pointed Nash entropy, together with \hyperref[thm:no-local-collapsing]{Theorem~\ref*{thm:no-local-collapsing}}, allow us to apply \hyperref[thm:compactness]{Theorem~\ref*{thm:compactness}} to extract a subsequence that converges smoothly to a complete hypersymplectic flow
	\begin{equation*}
		(M_\infty,\varphi_\infty(t),g_\infty(t),x_\infty)_{t\in(0,1]}.
	\end{equation*}
	Since $\lambda_i\to 0$, we have $T_{\varphi_\infty}=0$, and therefore $g_\infty(t)$ is Ricci flat and static.

	Write $\tau:=1-t$ and define $u_i,f_i$ by
	\begin{equation*}
		d\nu_i=u_idV_{g_i}=(4\pi\tau)^{-7/2}e^{-f_i}dV_{g_i}.
	\end{equation*}
	The uniform heat kernel upper bound \hyperref[thm:heat-kernel-upper-bound]{Theorem~\ref*{thm:heat-kernel-upper-bound}} and local parabolic estimates show that, after passing to a further subsequence, we may assume $u_i$ converges smoothly on compact subsets to some $u_{\infty}\geq 0$ satisfying $\square^*_{g_{\infty}}u_{\infty}=0$.

	Next we show that $u_{\infty}$ is the conjugate heat kernel based at $(x_\infty,1)$.
	By \hyperref[cor:concentration-under-Ricci-bound]{Corollary~\ref*{cor:concentration-under-Ricci-bound}}, we have $d_{W_1}^{g_i}(\nu_i,\delta_{x_i})\leq C\sqrt{\tau}$ for some $C(\Lambda)>0$.
	Then it follows from Markov's inequality that
	\begin{equation}\label{ineq:tail-estimates}
		\nu_i(M_i\setminus B_{g_i(t)}(x_i,L))
		\leq\frac{C\sqrt{\tau}}{L}.
	\end{equation}
	Combining the above inequality with the $C^{\infty}_{\mathrm{loc}}$-convergence of $u_i$, we see that $\nu_{\infty}:=u_{\infty}dV_{g_{\infty}}$ has total mass one.
	Finally, the lower semicontinuity of $W_1$-distance gives $d_{W_1}^{g_\infty}(\nu_\infty,\delta_{x_\infty})\leq C\sqrt{\tau}$, showing that $\nu_\infty(t)\to\delta_{x_{\infty}}$ as $t\to 1$.
	
	Now that $u_{\infty}$ has positive total mass, the strong maximum principle gives $u_{\infty}>0$, and we may define $f_{\infty}$ by $u_{\infty}=(4\pi\tau)^{-7/2}e^{-f_{\infty}}$.
	Clearly, $f_i\to f_{\infty}$ smoothly on compact subsets.
	By \hyperref[prop:time-derivative-of-pointed-Nash-entropy]{Proposition~\ref*{prop:time-derivative-of-pointed-Nash-entropy}} and (\ref{cor:upper-bound-on-N}), we have $\mathcal{N}_{x_i,1}(\tau)\to 0$ for all $\tau$.
	A standard cut-off argument, based on (\ref{ineq:tail-estimates}), \hyperref[prop:energy-estimate]{Proposition~\ref*{prop:energy-estimate}} and the above convergence, gives that $\mathcal{N}_{x_{\infty},1}(\tau)=0$ for all $\tau$.
	This forces $\Ric_{g_\infty}+\nabla^2_{g_\infty}f_\infty-\frac{1}{2\tau}g_\infty=0$.
	Since $\Ric_{g_\infty}=0$ and $(M_{\infty},g_{\infty})$ is complete, the limit has to be Euclidean.
\end{proof}

The following theorem corresponding to \cite[Proposition 3.4]{wang2020local} serves as an illustration of our methods towards pseudolocality.
Since we lack an effective monotonicity formula for the local $\mu$-functional, it will not be used to prove our pseudolocality theorem directly as \cite{wang2020local} did for Ricci flow.
\begin{Theorem}\label{thm:mu-gap}
	For each $\alpha\in(0,1]$, $H\geq 1$ and $L\geq 1$, there exists $\delta(\alpha,H,L)>0$ such that the following holds.
	Let $(M^7,\varphi(t),g(t))_{t\in[0,T]}$ be a Laplacian flow over a compact manifold satisfying the uniform torsion bound $|T_\varphi|^2\leq\Lambda$ for some $\Lambda>0$, and let $x_0$ be a point on $M$.
	Suppose
	\begin{equation}\label{mu-gap-condition}
		\begin{cases}
			T\leq\delta\Lambda^{-1};\\
			|\Rm|(\cdot,t)\leq H/T \text{ over } B(x_0,t,\delta^{-1}\sqrt{T})
			\text{ for all } t\in[0,T];\\
			\mu(B(x_0,0,\delta^{-1}\sqrt{T}),g(0),T)\geq -\delta.
		\end{cases}
	\end{equation}
	Then
	\begin{equation*}
		\sup_{B(x_0,T,L\sqrt{T})}|\Rm|(\cdot,T)\leq\alpha/T.
	\end{equation*}
\end{Theorem}
\begin{proof}
	Assume there is no such $\delta$. Then we may find $\delta_i\to 0$ and a sequence $(M^7_i,\varphi_i(t),g_i(t))_{t\in[0,T_i]}$ of Laplacian flows over compact manifolds satisfying $|T_{\varphi_i}|^2\leq\Lambda_i$, together with $x_i\in M_i$ such that
	\begin{equation*}
		\begin{cases}
			T_i\leq\delta_i\Lambda_i^{-1};\\
			|\Rm|_{g_i}(\cdot,t)\leq H/T_i \text{ over } B(x_i,t,\delta_i^{-1}\sqrt{T_i})
			\text{ for all } t\in[0,T_i];\\
			\mu(B(x_i,0,\delta_i^{-1}\sqrt{T_i}),g_i(0),T_i)\geq -\delta_i.
		\end{cases}
	\end{equation*}
	Moreover, there exist $z_i\in B(x_i,T_i,L\sqrt{T_i})$ such that $|\Rm|_{g_i}(z_i,T_i)>\alpha/T_i$.
	
	By parabolic rescaling $\tilde{\varphi}_i(t):=T_i^{-3/2}\varphi_i(T_it)$, we obtain a sequence $(M_i,\tilde{\varphi}_i(t),\tilde{g}_i(t),x_i)_{t\in[0,1]}$ of Laplacian flows with $|T_{\tilde{\varphi}_i}|^2\leq \tilde{\Lambda}_i:=T_i\Lambda_i\leq\delta_i\to 0$ and
	\begin{equation*}
		\begin{cases}
			|\Rm|_{\tilde{g}_i}(\cdot,t)\leq H \text{ over } B(x_i,t,\delta_i^{-1})
			\text{ for all } t\in[0,1];\\
			\mu(B(x_i,0,\delta_i^{-1}),\tilde{g}_i(0),1)\geq -\delta_i.
		\end{cases}
	\end{equation*}
	Moreover, there exist $z_i\in B(x_i,1,L)$ such that $|\Rm|_{\tilde{g}_i}(z_i,1)>\alpha$. By \hyperref[thm:mu-and-N]{Theorem~\ref*{thm:mu-and-N}} (with $\beta=1$),
	\begin{equation*}
		\liminf_{i\to\infty}\mathcal{N}_{x_i,1}(1)\geq 0
	\end{equation*}
	and (\ref{cor:upper-bound-on-N}) then gives $\mathcal N_{x_i,1}(1)\to 0$.
	It follows from \hyperref[lem:pointed-entropy-rigidity]{Lemma~\ref*{lem:pointed-entropy-rigidity}} that the subsequence limit has to be Euclidean and therefore $|\Rm|_{\tilde g_i}(z_i,1)\to 0$, contradicting $|\Rm|_{\tilde{g}_i}(z_i,1)>\alpha$.
\end{proof}

\subsection{Point-picking}
The following point-picking procedure for Ricci flow is taken from \cite[Proposition 4.1]{wang2020local}.
We include its proof as an illustration, and point out the necessary modifications for our generalization \hyperref[prop:point-picking]{Proposition~\ref*{prop:point-picking}}.
\begin{Proposition}[Ricci flow point-picking]\label{prop:point-picking-Ricci-flow}
	Let $(M^n,g(t))_{t\in[0,T]}$ be a Ricci flow with complete time-slices and $K\geq 1$, $\beta\in(0,1]$ be two constants.
	Suppose
	\begin{equation*}
		\Ric(x,t)\leq\frac{(n-1)\alpha_nA'}{\beta^2t},\qquad
		\forall\, x\in B(x_0,t,\beta\sqrt{t}),    \ t\in[0,T],
	\end{equation*}
	where $A'\geq\alpha_n^{-5}K^4$ is a large constant.
	Define
	\begin{align*}
		&\Omega_t=B(x_0,t,2K\beta\sqrt{t}),\\
		&\mathcal{D}=\{(x,t):x\in\Omega_t,\,0\leq t\leq T\},\\
		&\rho(x,t)=2K\beta\sqrt{t}-d_t(x,x_0),\\
		&F(x,t)=\sqrt{|\Rm|(x,t)}\rho(x,t).
	\end{align*}
	Let $(y_0,t_0)$ be a point where $F$ achieves the maximum value on $\mathcal{D}$.
	Define $Q_0=|\Rm|(y_0,t_0)$, $\rho_0=\rho(y_0,t_0)$ and $F_0=F(y_0,t_0)$. If $F_0^2\geq\frac{\alpha_nA'}{2}$, then for every $t\in[t_0-Q_0^{-1},t_0]$, we have
	\begin{equation}\label{containing-relation}
		B(y_0,t,\tfrac{1}{5}\rho_0)
		\subset
		B(x_0,t,2K\beta\sqrt{t}-\tfrac{3}{5}\rho_0),
	\end{equation}
	i.e., $\rho(x,t)\geq\frac{3}{5}\rho_0$ for all $x\in B(y_0,t,\frac{1}{5}\rho_0)$.
	Consequently,
	\begin{equation*}
		|\Rm|(x,t)\leq 4Q_0,
		\quad
		\forall\, x\in B(y_0,t,\tfrac{1}{5}\rho_0),\  t\in[t_0-Q_0^{-1},t_0].
	\end{equation*}
\end{Proposition}
\begin{proof}
	By the choice of $(y_0,t_0)$, we have
	\begin{equation*}
		|\Rm|(x,t)
		\leq
		\left(\frac{\rho_0}{\rho(x,t)}\right)^2Q_0,
		\quad\forall(x,t)\in\mathcal{D}.
	\end{equation*}
	This simply says we may bound the curvature at points distant to $\partial\Omega_t$.
	Since $\rho(x,t_0)\geq\frac{4}{5}\rho_0$ for $x\in B(y_0,t_0,\frac{1}{5}\rho_0)$, it holds
	\begin{equation*}
		|\Rm|(x,t_0)
		\leq
		\tfrac{25}{16}Q_0<4Q_0,
		\quad
		\forall\, x\in B(y_0,t_0,\tfrac{1}{5}\rho_0).
	\end{equation*}
	Suppose (\ref{containing-relation}) does not hold.
	Then there exists $s_0\in[t_0-Q_0^{-1},t_0)$ such that (\ref{containing-relation}) holds for all $t\in(s_0,t_0]$ but fails at time $s_0$.
	This choice of $s_0$ has the following consequences.
	\begin{itemize}
		\item One may find $z_0\in B(y_0,s_0,\frac{1}{5}\rho_0)\setminus B(x_0,s_0,2K\beta\sqrt{s_0}-\frac{3}{5}\rho_0)$ and therefore conclude
		\begin{align*}
			d_{s_0}(x_0,y_0)
			&\geq	d_{s_0}(x_0,z_0)-d_{s_0}(z_0,y_0)\\
			&\geq	2K\beta\sqrt{s_0}-\tfrac{3}{5}\rho_0-\tfrac{1}{5}\rho_0\\
			&=
			2K\beta\sqrt{s_0}-\tfrac{4}{5}\rho_0.
		\end{align*}
		
		\item For all $x\in B(y_0,t,\frac{1}{5}\rho_0)$ and $t\in[s_0,t_0]$, it holds $\rho(x,t)\geq\frac{3}{5}\rho_0$  and therefore $|\Rm|(x,t)\leq\frac{25}{9}Q_0<4Q_0$. In particular,
		\begin{equation*}
			\Ric(x,t)
			\leq
			4(n-1)Q_0,
			\quad\forall\,x\in B(y_0,t,\tfrac{1}{5}\rho_0),
			\ t\in[s_0,t_0].
		\end{equation*}
	\end{itemize}
	
	Also, we can estimate $\Ric$ around $(x_0,t_0)$ as follows. Clearly, $2K\beta\sqrt{t_0}\geq\rho_0=F_0Q_0^{-1/2}$, which implies
	\begin{equation*}
		Q_0t_0
		\geq\frac{F_0^2}{4K^2\beta^2}
		\geq\frac{\alpha_nA'}{8K^2\beta^2}
		\geq\frac{2}{\beta^2}
		\geq 2.
	\end{equation*}
	Thus for $t\in[s_0,t_0]$, we have
	\begin{equation*}
		\frac{t_0}{t}   \leq   \frac{t_0}{s_0}
		\leq   \frac{t_0}{t_0-Q_0^{-1}}
		=   \frac{Q_0t_0}{Q_0t_0-1}
		\leq   2.
	\end{equation*}
	Combine the above two facts with the assumption on Ricci curvature around $x_0$, we have
	\begin{equation*}
		\Ric(x,t)
		\leq
		\frac{(n-1)\alpha_nA'}{\beta^2t}
		=
		(n-1)\cdot\frac{\alpha_nA'}{\beta^2Q_0t_0}\cdot\frac{t_0}{t}\cdot Q_0
		\leq
		16(n-1)K^2Q_0
	\end{equation*}
	for all $x\in B(x_0,t,\beta\sqrt{t})$ and $t\in[s_0,t_0]$. Since
	\begin{equation*}
		Q_0t
		=Q_0t_0-Q_0(t_0-t)
		\geq Q_0t_0-1
		\geq\frac{2}{\beta^2}-1
		\geq\frac{1}{\beta^2}
	\end{equation*}
	for all $t\in[s_0,t_0]$, we conclude $B(x_0,t,Q_0^{-1/2})\subset B(x_0,t,\beta\sqrt{t})$.
	Therefore
	\begin{equation*}
		\Ric(x,t)\leq 16(n-1)K^2Q_0,
		\quad\forall\,x\in B(x_0,t,Q_0^{-1/2}),\  t\in[s_0,t_0].
	\end{equation*}
	
	Now that we have
	\begin{equation*}
		\Ric(x,t)\leq 16(n-1)K^2Q_0,
		\quad\forall
		\,x\in B(x_0,t,Q_0^{-1/2})\cup B(y_0,t,Q_0^{-1/2}),
		\ t\in[s_0,t_0],
	\end{equation*}
	the distance distortion estimate \hyperref[lem:distance-distortion]{Lemma~\ref*{lem:distance-distortion}}(2) implies
	\begin{equation*}
		\frac{d}{dt}d_t(x_0,y_0)
		\geq
		-2(n-1)\Big(\frac{2}{3}\cdot 16K^2Q_0\cdot Q_0^{-1/2}+Q_0^{1/2}\Big)
		\geq -40nK^2Q_0^{1/2}.
	\end{equation*}
	Integrating over $[s_0,t_0]$, we get
	\begin{align*}
		d_{s_0}(x_0,y_0)
		&\leq
		d_{t_0}(x_0,y_0)+40nK^2Q_0^{-1/2}\\
		&=
		2K\beta\sqrt{t_0}-F_0Q_0^{-1/2}+40nK^2Q_0^{-1/2}.
	\end{align*}
	Combining this with the lower bound of $d_{s_0}(x_0,y_0)$, it follows
	\begin{align*}
		2K\beta\sqrt{t_0-Q_0^{-1}}-\frac{4}{5}F_0Q_0^{-1/2}
		&\leq
		2K\beta\sqrt{s_0}-\frac{4}{5}F_0Q_0^{-1/2}\\
		&\leq
		2K\beta\sqrt{t_0}-F_0Q_0^{-1/2}+40nK^2Q_0^{-1/2}.
	\end{align*}
	Since $\sqrt{1-(Q_0t_0)^{-1}}\geq 1-(Q_0t_0)^{-1}$, we conclude
	\begin{equation*}
		F_0\leq 200nK^2+\frac{10K\beta}{\sqrt{Q_0t_0}}\leq 300nK^2.
	\end{equation*}
	This contradicts the fact that
	\begin{equation*}
		F_0\geq\sqrt{\frac{\alpha_nA'}{2}}
		\geq\frac{\alpha_n^{-2}K^2}{2}=5000n^2K^2.
	\end{equation*}
\end{proof}

Next we show that \hyperref[prop:point-picking-Ricci-flow]{Proposition~\ref*{prop:point-picking-Ricci-flow}} remains valid in the case $E(g)\geq -Ag$ if the time interval of the flow is not too large compared to $A^{-1}$.
\begin{Proposition}[Point-picking]\label{prop:point-picking}
	Let $K\geq 1$ and $\beta\in(0,1]$ be constants and $(g(t))_{t\in[0,T]}$ be a smooth family of complete Riemannian metrics over $M^n$ with $E(g)\geq-Ag$ for some $A\geq 0$ satisfying $AT\leq 100nK$. Suppose
	\begin{equation*}
		\Ric(x,t)\leq\frac{(n-1)\alpha_nA'}{\beta^2 t},
		\quad
		\forall\, x\in B(x_0,t,\beta\sqrt{t}),    \ t\in[0,T],
	\end{equation*}
	where $A'\geq\alpha_n^{-5}K^4$ is a large constant. Define
	\begin{align*}
		&\Omega_t=B(x_0,t,2K\beta\sqrt{t}),\\
		&\mathcal{D}=\{(x,t):x\in\Omega_t,\,0\leq t\leq T\},\\
		&\rho(x,t)=2K\beta\sqrt{t}-d_t(x,x_0),\\
		&F(x,t)=\sqrt{|\Rm|(x,t)}\rho(x,t).
	\end{align*}
	Let $(y_0,t_0)$ be a point where $F$ achieves the maximum value on $\mathcal{D}$.
	Define $Q_0=|\Rm|(y_0,t_0)$, $\rho_0=\rho(y_0,t_0)$ and $F_0=F(y_0,t_0)$. If $F_0^2\geq\frac{\alpha_nA'}{2}$, then for every $t\in[t_0-Q_0^{-1},t_0]$, we have
	\begin{equation*}
		B(y_0,t,\tfrac{1}{5}\rho_0)
		\subset
		B(x_0,t,2K\beta\sqrt{t}-\tfrac{3}{5}\rho_0)
		\subset\Omega_t.
	\end{equation*}
	Consequently,
	\begin{equation*}
		|\Rm|(x,t)\leq 4Q_0,
		\quad
		\forall\, x\in B(y_0,t,\tfrac{1}{5}\rho_0),\  t\in[t_0-Q_0^{-1},t_0].
	\end{equation*}
\end{Proposition}
\begin{proof}
	The argument is similar to \hyperref[prop:point-picking-Ricci-flow]{Proposition~\ref*{prop:point-picking-Ricci-flow}}.
	The only place where Ricci flow enters is the distance distortion estimate.
	In our case,
	\begin{align*}
		\frac{d}{dt}d_t(x_0,y_0)
		&\geq
		-2(n-1)\Big(\frac{2}{3}\cdot 16K^2Q_0\cdot Q_0^{-1/2}+Q_0^{1/2}\Big)-\frac{A}{2}d_t(x_0,y_0)\\
		&\geq
		-40nK^2Q_0^{1/2}-\frac{A}{2}d_t(x_0,y_0).
	\end{align*}
	Equivalently,
	\begin{equation*}
		\frac{d}{dt}\left(e^{\frac{A}{2}t}d_t(x_0,y_0)\right)
		\geq
		-40nK^2Q_0^{1/2}e^{\frac{A}{2}t}.
	\end{equation*}
	Integrating over $[s_0,t_0]$, we get
	\begin{align*}
		d_{s_0}(x_0,y_0)
		&\leq
		e^{\frac{A}{2}(t_0-s_0)}d_{t_0}(x_0,y_0)
		+40nK^2Q_0^{1/2}\cdot\sigma_{A/2}(Q_0^{-1})\\
		&\leq
		e^{\frac{A}{2}Q_0^{-1}}(2K\beta\sqrt{t_0}-F_0Q_0^{-1/2})
		+40nK^2Q_0^{1/2}\cdot\sigma_{A/2}(Q_0^{-1}).
	\end{align*}
	On the other hand,
	\begin{align*}
		d_{s_0}(x_0,y_0)
		&\geq
		2K\beta\sqrt{t_0-Q_0^{-1}}-\frac{4}{5}F_0Q_0^{-1/2}\\
		&\geq
		2K\beta\sqrt{t_0}\big(1-(Q_0t_0)^{-1}\big)-\frac{4}{5}F_0Q_0^{-1/2}.
	\end{align*}
	Combining the above two inequalities, we get
	\begin{align*}
		\frac{1}{5}F_0&\leq
		\Big(1-\frac{4}{5}e^{-\frac{A}{2}Q_0^{-1}}\Big)F_0\\
		&\leq
		\big(1-e^{-\frac{A}{2}Q_0^{-1}}\big)
		2K\beta\sqrt{Q_0t_0}
		+\frac{2K\beta}{\sqrt{Q_0t_0}}e^{-\frac{A}{2}Q_0^{-1}}
		+40nK^2Q_0\cdot\sigma_{A/2}(Q_0^{-1})\cdot e^{-\frac{A}{2}Q_0^{-1}}\\
		&\leq
		\frac{AK\beta t_0}{\sqrt{Q_0t_0}}
		+\frac{2K\beta}{\sqrt{Q_0t_0}}
		+40nK^2.
	\end{align*}
	Here we use $1-e^{-x}\leq x$ twice in the third inequality.
	Since $Q_0t_0>1$ and $AT\leq 100nK$, we have
	\begin{equation*}
		F_0\leq 1000 nK^2,
	\end{equation*}
	contradicting the fact that $F_0\geq 5000n^2K^2$.
\end{proof}

\subsection{Pseudolocality}
The above point-picking procedure allows to prove the following rough pseudolocality theorem, c.f. \cite[Theorem 4.4]{wang2020local}.
\begin{Theorem}[Rough pseudolocality]\label{thm:rough-pseudolocality}
	There exists a universal constant $A$ with the following holds. Let $(M^7,\varphi(t),g(t),x_0)_{t\in[0,T]}$ be a Laplacian flow over a compact manifold satisfying the uniform torsion bound $|T_\varphi|^2\leq\Lambda$ for some $\Lambda>0$. Suppose $T\leq A^{-1}\Lambda^{-1}$ and
	\begin{equation*}
		\inf_{0<t\leq T}
		\mu\big(B(x_0,0,A^2\sqrt{t}),g(0),t\big)
		\geq-A^{-1},
	\end{equation*}
	then
	\begin{equation*}
		|\Rm|(x,t)\leq\frac{\alpha_7A^2}{t},
		\quad\forall\,x\in B(x_0,t,A^{-1/2}\sqrt{t}),\ t\in(0,T].
	\end{equation*}
\end{Theorem}
\begin{proof}
	Assume there is not such $A_0$. Then there exist $A_i\to\infty$, and a sequence
	\begin{equation*}
		(M^7_i,\varphi_i(t),g_i(t),x_i)_{t\in[0,T_i]}
	\end{equation*}
	of Laplacian flows over compact manifolds satisfying $|T_{\varphi_i}|^2\leq\Lambda_i$, $T_i\leq A_i^{-1}\Lambda_i^{-1}$ and
	\begin{equation*}
		\inf_{0<t\leq T_i}
		\mu\big(B(x_i,0,A_i^2\sqrt{t}),g_i(0),t\big)
		\geq-A_i^{-1}.
	\end{equation*}
	Moreover, for each $i$, there exists $t_i^*\in(0,T_i]$ such that (set $\beta_i:=A_i^{-1/2}$)
	\begin{equation*}
		\begin{dcases}
			|\Rm|_{{g}_i}(x,t)\leq\frac{\alpha_nA_i}{\beta_i^2t},
			\quad\forall\, x\in B(x_i,t,\beta_i\sqrt{t}),\ t\in(0,t_i^*];\\
			\exists\, z_i\in B(x_i,t_i^*,\beta_i\sqrt{t_i^*})\subset M_i\text{ s.t. } |\Rm|_{{g}_i}(z_i,t_i^*)=\frac{\alpha_nA_i}{\beta_i^2t_i^*}.
		\end{dcases}
	\end{equation*}
	
	As in \hyperref[prop:point-picking]{Proposition~\ref*{prop:point-picking}} (we take $K=1$ below), we define $F_i(x,t)=\sqrt{|\Rm|_{{g}_i}(x,t)}\rho_i(x,t)$, where $x\in B(x_i,t,2K\beta_i\sqrt{t})$ and $t\in(0,t_i^*]$.
	Let $F_i^*$ be the maximum of $F_i$.
	Then
	\begin{enumerate}
		\item either $(F_i^*)^2\geq{\alpha_nA_i}/{2}$,
		\item or $(F_i^*)^2<{\alpha_nA_i}/{2}$.
	\end{enumerate}
	In the case (2), it holds by the definition of $F_i^*$ that
	\begin{equation*}
		|\Rm|_{{g}_i}(x,t)\big(2K\beta_i\sqrt{t}-d_{{g}_i(t)}(x,x_i)\big)^2
		<\frac{\alpha_nA_i}{2},
		\quad\forall\,x\in B_{{g}_i(t)}(x_i,2K\beta_i\sqrt{t}),\ t\in(0,t_i^*].
	\end{equation*}
	In particular,
	\begin{equation*}
		|\Rm|_{{g}_i}(x,t)<\frac{\alpha_nA_i}{2K^2\beta_i^2t},
		\quad\forall\,x\in B_{{g}_i(t)}(x_i,K\beta_i\sqrt{t}),\ t\in(0,t_i^*].
	\end{equation*}
	This contradicts the existence of $z_i$.
	Therefore (2) is impossible and (1) must hold. In particular, $F_i^*\to\infty$ as $i\to\infty$.
	Let $(y_i,t_i)$ be a point where $F_i$ achieves maximum and $Q_i=|\Rm|_{g_i}(y_i,t_i)$.
	Then \hyperref[prop:point-picking]{Proposition~\ref*{prop:point-picking}} implies
	\begin{equation*}
		|\Rm|_{g_i}(x,t)\leq 4Q_i,
		\quad\forall\,x\in B_{g_i(t)}\big(y_i,\tfrac{1}{5}F_i^*Q_i^{-1/2}\big),
		\ t\in[t_i-Q_i^{-1},t_i].
	\end{equation*}
	
	Set $\sigma_i=\sigma_{\frac{2}{3}\Lambda_i}$ and $\bar{\sigma}_i=\bar{\sigma}_{\frac{2}{3}\Lambda_i}$. Recall $\bar{\sigma}_i(t)\leq t$, and $\sigma_i(t)\leq 2t$ for sufficiently large $i$. Hence
	\begin{equation*}
		\Ric_{g_i}(x,t)
		\leq
		\frac{(n-1)\alpha_n A_i}{\beta_i^2\bar{\sigma}_i(t)},
		\quad\forall\, x\in B(x_i,t,\beta_i\sqrt{\bar{\sigma}_i(t)}),\ t\in(0,t_i].
	\end{equation*}
	Define $A_i'':=(n-1)(\frac{2}{3}\alpha_nA_i+1)\beta_i^{-1}$ and $\Omega_i:=B(x_i,0,A_i^2\sqrt{t_i})$.
	Then
	\begin{equation*}
		B(x_i,0,20A_i''\sqrt{\sigma_i(t_i)})
		\subset B(x_i,0,40A_i''\sqrt{t_i})\subset \Omega_i
	\end{equation*}
	for large $i$.
	By \hyperref[thm:mu-and-N]{Theorem~\ref*{thm:mu-and-N}},
	\begin{align*}
		\mathcal{N}_{x_i,t_i}(t_i)
		&\geq-e^{2/5}A_i^{-1}
		-c_1\Lambda_i^2T_i^2-c_2\Lambda_iT_i
		-C_6(A''_i)^{-1}.
	\end{align*}
	Therefore
	\begin{equation*}
		\liminf_{i\to\infty}\mathcal{N}_{x_i,t_i}(t_i)\geq 0.
	\end{equation*}
	By \hyperref[thm:gradient-estimate-Nash-entropy]{Theorem~\ref*{thm:gradient-estimate-Nash-entropy}} (more precisely, inequality (\ref{ineq:directional-derivative-estimate-Nash-entropy})) and the fact $y_i\in B(x_i,t_i,2K\beta_i\sqrt{t_i})$,
	\begin{equation*}
		\mathcal{N}^*_0(y_i,t_i)
		\geq
		\mathcal{N}^*_0(x_i,t_i)
		-
		\frac{C(n)}{\sqrt{t_i}}   \cdot   d_{g_i(t_i)}(x_i,y_i)
		\geq
		\mathcal{N}^*_0(x_i,t_i)
		-
		C(n)   \cdot   2K\beta_i.
	\end{equation*}
	Here we use $\Lambda_iT_i\to 0$ (in particular, $\Lambda_iT_i\leq 1$ for large $i$) to drop the dependence of the constant on $|E_i|_{\mathrm{max}}T_i$. This yields
	\begin{equation*}
		\liminf_{i\to\infty}\mathcal{N}_{y_i,t_i}(Q_i^{-1})
		\geq
		\liminf_{i\to\infty}\mathcal{N}_{y_i,t_i}(t_i)
		\geq 0.
	\end{equation*}

	By parabolic rescaling $\tilde{\varphi}_i(t):=Q_i^{3/2}\varphi_i\big(Q_i^{-1}t+(t_i-Q_i^{-1})\big)$, we obtain a sequence
	\begin{equation*}
		(M_i,\tilde{\varphi}_i(t),\tilde{g}_i(t),y_i)_{t\in[0,1]}
	\end{equation*}
	of Laplacian flows with $|T_{\tilde{\varphi}_i}|^2\leq \tilde{\Lambda}_i:=Q_i^{-1}\Lambda_i\leq T_i\Lambda_i\leq A_i^{-1}\to 0$. Moreover,
	\begin{equation}\label{ineq:liminf-of-pointed-Nash-entropy}
		\liminf_{i\to\infty}\mathcal{N}_{y_i,1}(1)\geq 0,
	\end{equation}
	\begin{equation*}
		|\Rm|_{\tilde{g}_i}(y_i,1)=1
		\qquad\text{and}\qquad
		|\Rm|_{\tilde{g}_i}(x,t)\leq 4,
		\quad\forall\,x\in B_{\tilde{g}_i(t)}(y_i,\tfrac{1}{5}F_i^*),
		\ t\in[0,1].
	\end{equation*}
	Combining (\ref{ineq:liminf-of-pointed-Nash-entropy}), (\ref{cor:upper-bound-on-N}) and the fact $\tilde\Lambda_i\to 0$, we have $\mathcal N_{y_i,1}(1)\to 0$.
	It follows from \hyperref[lem:pointed-entropy-rigidity]{Lemma~\ref*{lem:pointed-entropy-rigidity}} that the subsequence limit has to be Euclidean and therefore $|\Rm|_{\tilde g_i}(y_i,1)\to0$, contradicting $|\Rm|_{\tilde g_i}(y_i,1)=1$.
\end{proof}

Next we state our pseudolocality theorem.
It can be shown directly at this stage, but we prefer to reduce it to the $\epsilon$-regularity of pointed Nash entropy (\hyperref[thm:refined-epsilon-regularity]{Theorem~\ref*{thm:refined-epsilon-regularity}}), as this clearly illustrate our basic intuition and the relationship between pseudolocality and the $\epsilon$-regularity of pointed Nash entropy.
In what follows $n$ is understood to be $7$.
\begin{Theorem}[Refined pseudolocality, c.f. \cite{wang2020local} Theorem 4.5]\label{thm:refined-pseudolocality}
	For any $\alpha\in(0,1]$, there exists $\delta>0$ depending only on $\alpha$ such that the following holds. Let $(M^7,\varphi(t),g(t))_{t\in[0,T]}$ be a Laplacian flow over a compact manifold satisfying the uniform torsion bound $|T_\varphi|^2\leq\Lambda$ for some $\Lambda>0$. Let $x_0$ be a point on $M$. Suppose $T\leq\delta\Lambda^{-1}$ and
	\begin{equation*}
		\inf_{0<t\leq T}
		\mu\big(B(x_0,0,\delta^{-1}\sqrt{t}),g(0),t\big)
		\geq-\delta,
	\end{equation*}
	then
	\begin{equation}\label{ineq:pseudolocality}
		|\Rm|(x,t)\leq\frac{\alpha}{t},
		\quad\forall\,x\in B(x_0,t,\alpha^{-1}\sqrt{t}),\ t\in(0,T].
	\end{equation}
	Moreover, we may choose $\delta$ so that for all $t\in(0,T]$ and $x\in B(x_0,t,\alpha^{-1}\sqrt{t})$ we have
	\begin{equation}\label{ineq:volume-ratio-and-inj}
		\inf_{\rho\in(0,\alpha^{-1}\sqrt{t}]}
		\frac{|B(x,t,\rho)|}{\rho^n}
		\geq(1-\alpha)\omega_n,\qquad
		\mathrm{inj}(x,t)\geq\alpha^{-1}\sqrt{t}.
	\end{equation}
\end{Theorem}
The proof is deferred to \hyperref[sec:pseudolocality-reduction]{Section~\ref*{sec:pseudolocality-reduction}}.

\section{Quantitative Rigidity of Pointed Nash Entropy}\label{sec:rigidity-of-pointed-Nash-entropy}
This section corresponds to $\S 10$ of \cite{bamler2020entropy}.
\begin{Definition}[Curvature scale radius]
	Let $I$ be an interval and $(g(t))_{t\in I}$ be a smooth family of Riemannian metrics over $M^n$.
	The curvature radius at $(x,t)\in M\times I$ is defined to be
	\begin{equation*}
		r_{\Rm}(x,t)
		=\sup\{r>0:|\Rm|\leq r^{-2}\text{ over }P(x,t;r)\}.
	\end{equation*}
	Recall our convention (\ref{eq:conventional-parabolic-neighborhood}) for the parabolic neighborhood $P(x,t;r)$.
	It depends on $I$ implicitly.
	We write $r_{\Rm}^I$ when the ambient time interval needs to be
	emphasized.
	Note that restricting the flow interval may increase the curvature radius: if $J\subset I$ and $t\in J$, then $r_{\Rm}^J(x,t)\geq r_{\Rm}^I(x,t)$.
\end{Definition}

We first prove a rough $\epsilon$-regularity theorem concerning the pointed Nash entropy.
For this, we need the following point-picking lemma.
\begin{Lemma}[point-picking]\label{lem:point-picking}
	Let $M^n$ be a compact manifold and $(g(t))_{t\in I}$ be a smooth family of Riemannian metrics over $M$ with $E(g)\geq-Ag$ for some $A\geq 0$.
	Suppose $L>0$ and $(x,t)\in M\times I$ satisfy $r_{\Rm}^I(x,t)<\infty$ and
	\begin{equation*}
		[t-(10Lr_{\Rm}^I(x,t))^2,t]\subset I.
	\end{equation*}
	Then there exists a point $(x',t')\in P^{*-}_A(x,t;10Lr_{\Rm}^I(x,t))$ such that
	\begin{equation*}
		r_{\Rm}^I(x',t')\leq r_{\Rm}^I(x,t),
		\qquad\text{and}\qquad
		r_{\Rm}^I\geq\tfrac{1}{10}r_{\Rm}^I(x',t')
		\quad\text{over }P^{*-}_A(x',t';Lr_{\Rm}^I(x',t')).
	\end{equation*}
\end{Lemma}
\begin{proof}
	We can iteratively pick a maximal sequence $(x_0',t_0'):=(x,t),(x_1',t_1'),\cdots$ such that for $r_i':=r_{\Rm}^I(x_i',t_i')$ we have
	\begin{equation*}
		(x_{i+1}',t_{i+1}')\in P^{*-}_A(x_i',t_i';Lr_i'),
		\qquad
		r_{i+1}'<\tfrac{1}{10}r_i'.
	\end{equation*}
	All these chosen points lie in the compact time interval $[t-(10Lr_{\Rm}^I(x,t))^2,t]$.
	The compactness of $M$ tells us that $r_{\Rm}^I$ has a positive lower bound, and therefore the sequence must be finite.
	We claim that its last element $(x_k',t_k')=:(x',t')$ has the desired properties.
	By the backward version of \hyperref[prop:containment-relationship]{Proposition~\ref*{prop:containment-relationship}}(5), we have for each $0\leq i\leq k-1$ that
	\begin{equation*}
		P_A^{*-}(x_{i+1}',t_{i+1}';10Lr_{i+1}')
		\subset
		P_A^{*-}(x_{i}',t_{i}';L(r_i'+10r_{i+1}'))
		\subset
		P_A^{*-}(x_{i}',t_{i}';10Lr_i').
	\end{equation*}
	This implies $(x',t')\in P^{*-}_A(x,t;10Lr_{\Rm}^I(x,t))$.
	By maximality of the sequence there is no point in $P_A^{*-}(x',t';Lr_k')$ whose curvature radius is less than $r_k'/10$.
	This is the asserted lower bound and completes the proof.
\end{proof}

\begin{Theorem}[Rough $\epsilon$-regularity]\label{thm:rough-epsilon-regularity}
	There exists a universal constant $\epsilon>0$ so that the following holds. Let $(M^7,\varphi(t),g(t))_{t\in I}$ be a Laplacian flow over a compact manifold satisfying the uniform torsion bound $|T_\varphi|^2\leq\Lambda$ for some $\Lambda>0$. Let $(x,t)\in M\times I$ and $r>0$ be such that $[t-r^2,t]\subset I$. Suppose $r\leq\epsilon\Lambda^{-1/2}$ and $\mathcal{N}_{x,t}(r^2)\geq-\epsilon$, then $r_{\Rm}(x,t)\geq\epsilon r$.
\end{Theorem}
\begin{proof}
	Assume there is no such $\epsilon$. Then we may find $\epsilon_i\to 0$ and a sequence $(M^7_i,\varphi_i(t),g_i(t))_{t\in I_i}$ of Laplacian flows over compact manifolds satisfying $|T_{\varphi_i}|^2\leq\Lambda_i$, together with $(x_i,t_i)\in M_i\times I_i$ and $r_i>0$ such that $r_i\leq\epsilon_i\Lambda_i^{-1/2}$, $\mathcal{N}_{x_i,t_i}(r_i^2)\geq-\epsilon_i$ but $r_{\Rm}^{I_i}(x_i,t_i)<\epsilon_i r_i$.
	
	By parabolic rescaling $\tilde{\varphi}_i(s):=r_i^{-3}\varphi_i(r_i^2s+(t_i-r_i^2))$, we obtain a sequence $(M_i,\tilde{\varphi}_i(s),\tilde g_i(s))_{s\in\widetilde I_i}$ of Laplacian flows with $|T_{\tilde{\varphi}_i}|^2\leq\tilde{\Lambda}_i:=r_i^2\Lambda_i\leq\epsilon_i^2\to 0$.
	Moreover, there exists $x_i\in M_i$ such that $\mathcal N_{x_i,1}(1)\geq-\epsilon_i$ and $r_{\Rm}^{\widetilde I_i}(x_i,1)<\epsilon_i$.
	Here
	\begin{equation*}
		\widetilde I_i
		:=\{s\in\mathbb R:r_i^2s+(t_i-r_i^2)\in I_i\}.
	\end{equation*}
	Clearly, $[0,1]\subset\widetilde I_i$ for all $i$.
	
	For $A_i:=\frac{2}{3}\tilde{\Lambda}_i$, we have $E(\tilde{g}_i)\geq-A_i\tilde{g}_i$.
	Set $r_i^*=r_{\Rm}^{\widetilde I_i}(x_i,1)$.
	Choose $L_i\to\infty$ such that $10L_ir_i^*\to 0$.
	By \hyperref[lem:point-picking]{Lemma~\ref*{lem:point-picking}}, there exists $(x_i',t_i')\in P^{*-}_{A_i}(x_i,1;10L_ir_i^*)$ such that
	\begin{equation*}
		r_{\Rm}^{\widetilde I_i}(x_i',t_i')\leq r_i^*,
		\qquad\text{and}\qquad
		r_{\Rm}^{\widetilde I_i}\geq\tfrac{1}{10}r_{\Rm}^{\widetilde I_i}(x_i',t_i')
		\text{ over }
		P^{*-}_{A_i}(x_i',t_i';L_ir_{\Rm}^{\widetilde I_i}(x_i',t_i')).
	\end{equation*}
	Set $r_i'=r_{\Rm}^{\widetilde I_i}(x_i',t_i')$. Then $r_i'\leq r_i^*<\epsilon_i\to 0$. We claim that for all $\tau\in(0,t_i'-\frac{1}{2}]$,
	\begin{equation}\label{claim:smallness-of-N}
		\lim_{i\to\infty}\mathcal{N}_{x_i',t_i'}(\tau)=0.
	\end{equation}
	By (\ref{cor:upper-bound-on-N}), it suffices to prove $\liminf_{i\to\infty}\mathcal{N}_{x_i',t_i'}(\tau)\geq 0$.
	To see this, one first use \hyperref[prop:time-derivative-of-pointed-Nash-entropy]{Proposition~\ref*{prop:time-derivative-of-pointed-Nash-entropy}} and the fact $|T_{\tilde{\varphi}_i}|^2\leq\tilde{\Lambda}_i\to 0$ to conclude
	\begin{equation*}
		\mathcal{N}_{x_i',t_i'}(\tau)
		\geq
		\mathcal{N}_{x_i',t_i'}(t_i'-\tfrac{1}{2})
		-o(1).
	\end{equation*}
	Similarly, $\mathcal{N}_{x_i,1}(\frac{1}{2})   \geq   \mathcal{N}_{x_i,1}(1)-o(1)$.
	Also, \hyperref[cor:oscillation-of-pointed-Nash-entropy]{Corollary~\ref*{cor:oscillation-of-pointed-Nash-entropy}} implies (taking $t^*_i:=1-(10L_ir_i^*)^2$)
	\begin{equation*}
		\mathcal{N}^*_{\frac{1}{2}}(x_i',t_i')
		\geq
		\mathcal{N}^*_{\frac{1}{2}}(x_i,1)
		-Cd_{W_1}^{g(t_i^*)}(\nu_{x_i,1;t_i^*},\nu_{x_i',t_i';t_i^*})
		-\frac{nA_i}{2}(1-t_i^*)
		-\frac{7}{2}\log\frac{t_i'-\frac{1}{2}}{t_i^*-\frac{1}{2}},
	\end{equation*}
	where $C$ is some universal constant. Then it follows from $(x_i',t_i')\in P^{*-}_{A_i}(x_i,1;10L_ir_i^*)$ that
	\begin{equation*}
		\mathcal{N}^*_{\frac{1}{2}}(x_i',t_i')
		\geq
		\mathcal{N}^*_{\frac{1}{2}}(x_i,1)-o(1).
	\end{equation*}
	In summary,
	\begin{equation*}
		\mathcal{N}_{x_i',t_i'}(\tau)
		\geq
		\mathcal{N}_{x_i',t_i'}(t_i'-\tfrac{1}{2})-o(1)
		\geq
		\mathcal{N}_{x_i,1}(\tfrac{1}{2})-o(1)
		\geq
		\mathcal{N}_{x_i,1}(1)-o(1).
	\end{equation*}
	Thus our claim (\ref{claim:smallness-of-N}) follows from the fact $\mathcal{N}_{x_i,1}(1)\geq-\epsilon_i\to 0$.
	
	By \hyperref[cor:comparison-of-parabolic-neighborhoods]{Corollary~\ref*{cor:comparison-of-parabolic-neighborhoods}}(2), we can find a sequence $L_i'\leq L_i$ with $L_i'\to\infty$ such that
	\begin{equation*}
		P^-(x_i',t_i';L_i'r_i')
		\subset
		P^{*-}_{A_i}(x_i',t_i';L_ir_i'),
	\end{equation*}
	which implies $r_{\Rm}^{\widetilde I_i}\geq\frac{1}{10}r_i'$ on $P^-(x_i',t_i';L_i'r_i')$.
	
	Another application of parabolic rescaling $\varphi_i'(t):={r_i'}{}^{-3}\varphi({r_i'}{}^{2}t+t_i')$ now yields a sequence of Laplacian flows $(M_i,\varphi_i'(s),g_i'(s),x_i')_{s\in I_i'}$ satisfying
	\begin{itemize}
		\item $|T_{\varphi_i'}|^2\leq\Lambda_i':=r_i'{}^2\tilde{\Lambda}_i$;
		\item $r_{\Rm}^{I_i'}(x_i',0)=1$;
		\item $r_{\Rm}^{I_i'}\geq\frac{1}{10}$ over $P^-(x_i',0;L_i')$;
		\item $\lim_{i\to\infty}\mathcal{N}_{x_i',0}(T)=0$ for any fixed $T>0$.
	\end{itemize}
	Here $I_i':=\{s\in\mathbb R:{r_i'}{}^2s+t_i'\in\widetilde I_i\}$ contains $[-L_i'{}^2,0]$.
	Applying \hyperref[thm:no-local-collapsing]{Theorem~\ref*{thm:no-local-collapsing}} and \hyperref[thm:compactness]{Theorem~\ref*{thm:compactness}} to the restrictions of these flows to $[-L_i'{}^2,0]$, we get a subsequence that converges smoothly to some ancient Laplacian flow $(M^7_\infty,\varphi_\infty(t),g_\infty(t),x_\infty)_{t\in(-\infty,0]}$.
	As shown in \hyperref[lem:pointed-entropy-rigidity]{Lemma~\ref*{lem:pointed-entropy-rigidity}}, this limit is the static Euclidean flow.

	Note this flatness implies that the curvature radii at $(x_i',0)$ relative to $I_i'\cap(-\infty,0]$ tend to infinity.
	It does not yet contradict $r_{\Rm}^{I_i'}(x_i',0)=1$.
	The last step is to obtain curvature bounds at positive time by pseudolocality.

	Let $A_{\rm ps}$ be the universal constant in \hyperref[thm:rough-pseudolocality]{Theorem~\ref*{thm:rough-pseudolocality}} and $\eta\in(0,\frac{1}{4})$ be a fixed constant to be determined below.
	Choose $\rho>3$ so large that
	\begin{equation}\label{ineq:choose-R-for-curvature-escape}
		\frac{\alpha_7A_{\rm ps}^2}{\rho^2-4}<\eta,\qquad
		A_{\rm ps}^{-1/2}\sqrt{\rho^2-4}>3.
	\end{equation}
	For large $i$ the flow is defined on $[-\rho^2,0]$, and $\rho^2+4\leq A_{\rm ps}^{-1}(\Lambda_i')^{-1}$.
	Since $(M_i,g_i'(-\rho^2),x_i')$ converges smoothly to the Euclidean space, \hyperref[lem:local-mu-euclidean-stability]{Lemma~\ref*{lem:local-mu-euclidean-stability}} gives
	\begin{equation*}
		\inf_{0<\theta\leq\rho^2+4}
		\mu\big(B_{g_i'(-\rho^2)}(x_i',A_{\rm ps}^2\sqrt{\theta}),
		g_i'(-\rho^2),\theta\big)\geq-A_{\rm ps}^{-1}
	\end{equation*}
	for all sufficiently large $i$.
	For $s\in[-4,4]\cap I_i'$, \hyperref[thm:rough-pseudolocality]{Theorem~\ref*{thm:rough-pseudolocality}} and (\ref{ineq:choose-R-for-curvature-escape}) gives 
	\begin{equation*}
		|\Rm|_{g_i'(s)}
		\leq\frac{\alpha_7A_{\rm ps}^2}{\rho^2+s}<\eta\qquad
		\text{on }B_{g_i'(s)}(x_i',3)\subset B_{g_i'(s)}(x_i',A_{\rm ps}^{-1/2}\sqrt{\rho^2+s}).
	\end{equation*}
	For large $i$ we also have $\Lambda_i'<\eta$.
	When $\eta$ is sufficiently small, we have
	\begin{equation*}
		B_{g_i'(0)}(x_i',2)
		\subset B_{g_i'(s)}(x_i',3),\quad
		\forall s\in[-4,4]\cap I_i'
	\end{equation*}
	for all sufficiently large $i$.  Therefore
	\begin{equation*}
		|\Rm|\leq\eta<\frac14
		\quad\text{on }B_{g_i'(0)}(x_i',2)\times([-4,4]\cap I_i').
	\end{equation*}
	Consequently $r_{\Rm}^{I_i'}(x_i',0)\geq2$, contradicting $r_{\Rm}^{I_i'}(x_i',0)=1$.
\end{proof}

The above rough $\epsilon$-regularity may be used to prove the following refined $\epsilon$-regularity theorem.
\begin{Theorem}[Refined $\epsilon$-regularity]\label{thm:refined-epsilon-regularity}
	For any $\epsilon>0$, there exists $\delta>0$ depending only on $\epsilon$ such that the following holds. Let $(M^7,\varphi(t),g(t))_{t\in I}$ be a Laplacian flow over a compact manifold satisfying the uniform torsion bound $|T_\varphi|^2\leq\Lambda$ for some $\Lambda>0$. Let $(x,t)\in M\times I$ and $r>0$ be such that $[t-r^2,t]\subset I$. Suppose $r\leq\delta\Lambda^{-1/2}$ and $\mathcal{N}_{x,t}(r^2)\geq-\delta$, then
	\begin{equation*}
		|\Rm|\leq\epsilon r^{-2}
		\quad\text{over }P(x,t;\epsilon^{-1}r,-(1-\epsilon)r^2,\epsilon^{-1}r^2).
	\end{equation*}
	Moreover, we may choose $\delta$ so that over $P(x,t;\epsilon^{-1}r,-(1-\epsilon)r^2,\epsilon^{-1}r^2)$ we have
	\begin{equation*}
		\mathcal{N}^*_{t-r^2}\geq-\epsilon,
		\qquad
		\inf_{0<\rho\leq\epsilon^{-1}r}
		\frac{|B(\cdot,\cdot,\rho)|}{\rho^n}
		\geq(1-\epsilon)\omega_n,
		\qquad
		\mathrm{inj}(\cdot,\cdot)\geq\epsilon^{-1}r.
	\end{equation*}
\end{Theorem}
\begin{proof}
	We first prove the backward curvature estimate, namely \begin{equation*}
		|\Rm|\leq\epsilon r^{-2}
		\quad\text{over }P(x,t;\epsilon^{-1}r,-(1-\epsilon)r^2)
	\end{equation*}
	for sufficiently small $\delta$. Assume there is no such $\delta$. Then we may find $\delta_i\to 0$ and a sequence $(M^7_i,\varphi_i(t),g_i(t))_{t\in I_i}$ of Laplacian flows over compact manifolds satisfying $|T_{\varphi_i}|^2\leq\Lambda_i$, together with $(x_i,t_i)\in M_i\times I_i$ and $r_i>0$ such that $r_i\leq\delta_i\Lambda_i^{-1/2}$, $\mathcal{N}_{x_i,t_i}(r_i^2)\geq-\delta_i$ but $|\Rm|(y_i,s_i)>\epsilon r_i^{-2}$ for some $(y_i,s_i)\in P(x_i,t_i;\epsilon^{-1}r_i,-(1-\epsilon)r_i^2)$.
	
	By parabolic rescaling $\tilde{\varphi}_i(t):=r_i^{-3}\varphi_i\big(r_i^2t+(t_i-r_i^2)\big)$, we obtain a sequence $(M_i,\tilde{\varphi}_i(t),\tilde{g}_i(t))_{t\in[0,1]}$ of Laplacian flows with $|T_{\tilde{\varphi}_i}|^2\leq \tilde{\Lambda}_i:=r_i^2\Lambda_i\leq\delta_i^2\to 0$. Moreover, there exist $x_i\in M_i$ and $(y_i,s_i')\in P(x_i,1;\epsilon^{-1},-(1-\epsilon))$ such that $\mathcal{N}_{x_i,1}(1)\geq-\delta_i$ and $|\Rm|(y_i,s_i')>\epsilon$. We claim for any $\theta\in(0,1]$, there exists a subsequence such that
	\begin{equation}\label{ineq:claim-10.7}
		\liminf_{i\to\infty}
		\Big(\inf_{P_{i,\theta}}\mathcal{N}^*_0\Big)
		\geq 0,
	\end{equation}
	where $P_{i,\theta}:=B(x_i,1,-\log\theta)\times[\theta,1]$.
	The proof of this claim will be deferred to \hyperref[rmk:claim-10.7]{Remark~\ref*{rmk:claim-10.7}}.
	Assuming the claim, \hyperref[thm:rough-epsilon-regularity]{Theorem~\ref*{thm:rough-epsilon-regularity}} then allows us to extract a subsequence $(M_i,\tilde{\varphi}_i(t),\tilde{g}_i(t),x_i)_{t\in[0,1]}$ that $C^\infty$-converges to some Laplacian flow $(M^7_\infty,\varphi_\infty(t),g_\infty(t),x_\infty)_{t\in(0,1]}$ with a curvature bound of the form $|\Rm|_{g_\infty}(\cdot,t)\leq C/t$.
	As in the proof of \hyperref[thm:rough-epsilon-regularity]{Theorem~\ref*{thm:rough-epsilon-regularity}}, the limit has to be Euclidean by (\ref{cor:upper-bound-on-N}) and \hyperref[lem:pointed-entropy-rigidity]{Lemma~\ref*{lem:pointed-entropy-rigidity}}.
	This contradicts $|\Rm|(y_i,s_i')>\epsilon$.
	
	Next we treat the forward part.
	By the curvature estimate in forward region of \hyperref[thm:rough-epsilon-regularity]{Theorem~\ref*{thm:rough-epsilon-regularity}}, we may in fact choose the subsequence in the above proof appropriately so that it smoothly converges over the time $(0,1+\rho]$ for some universal $\rho>0$.
	Similar to the argument in the proof of (\ref{ineq:claim-10.7}), one can show that the pointed Nash entropy $\mathcal{N}^*_0(\cdot,\cdot)$ converges locally uniformly to zero (up to taking subsequence).
	Thus for sufficiently large $i$, we may use \hyperref[thm:rough-epsilon-regularity]{Theorem~\ref*{thm:rough-epsilon-regularity}} to obtain curvature estimates on the larger time interval $(0,1+2\rho]$.
	Then by passing to a further subsequence, the sequence converges over $(0,1+2\rho]$.
	We will reach a contradiction as above after repeating this a finite number of times.
	
	The estimate $\mathcal{N}^*_{t-r^2}\geq-\epsilon$ follows from the locally uniform convergence of pointed Nash entropy.
	The bounds on volume ratio and injective radius follow as in \cite[Proposition 3.4]{wang2020local}.
\end{proof}

\begin{Remark}\label{rmk:claim-10.7}
	The proof of (\ref{ineq:claim-10.7}) is a minor modification of that for \cite[Claim 10.7]{bamler2020entropy}, which we repeat in the following.
\end{Remark}
\begin{proof}[Proof of (\ref{ineq:claim-10.7})]
	Assume the claim is false, then
	\begin{equation*}
		\theta_0:=\inf\{\theta\in(0,1]:\text{the claim holds true}\}
		>0.
	\end{equation*}
	For any $\theta\in(\theta_0,1)$, there exists (by \hyperref[thm:rough-epsilon-regularity]{Theorem~\ref*{thm:rough-epsilon-regularity}}) a universal constant $\sigma>0$ such that the curvature is uniformly bounded over $P_{i,\theta'}$ for $\theta':=(1-\sigma)\theta$. Passing to a subsequence, the flows restricted to $P_{i,\theta'}$ and pointed at $(x_i,1)$ converge to a smooth flow with bounded curvature on a parabolic neighborhood of the form $P_{\infty,\theta'}:=B(x_\infty,1,-\log\theta')\times(\theta',1]$.
	
	Fix $t^*\in(\theta',1]$ and consider the solutions $u_i\in C^\infty(M_i\times[t^*,1])$ to the heat equation $\square u_i=0$ with initial condition $u_i(\cdot,t^*)=\mathcal{N}^*_0(\cdot,t^*)$. By the maximum principle and (\ref{cor:upper-bound-on-N}),
	\begin{equation}\label{ineq:ui-to-zero}
		u_i\leq\sup_{M_i}\mathcal{N}_0^*(\cdot,t^*)\to 0
	\end{equation}
	By \hyperref[ineq:square-of-N]{Theorem~\ref*{ineq:square-of-N}} and the maximum principle,
	\begin{equation}\label{ineq:ui-estimates}
		\mathcal{N}^*_0(x,t)-C\tilde{\Lambda}_i(t-t^*)
		\leq
		u_i(x,t)
		\leq
		\mathcal{N}^*_0(x,t)+\frac{7}{2}\log\frac{t}{t^*}
		\quad\forall\,(x,t)\in M_i\times[t^*,1].
	\end{equation}
	So by \hyperref[cor:oscillation-of-pointed-Nash-entropy]{Corollary~\ref*{cor:oscillation-of-pointed-Nash-entropy}}, \hyperref[cor:comparison-of-parabolic-neighborhoods]{Corollary~\ref*{cor:comparison-of-parabolic-neighborhoods}}, we obtain uniform bounds on $u_i$ over $P_{i,\theta'}$.
	
	By passing to a further subsequence, $u_i\to u_\infty$ for some $u_\infty\in C^\infty(B(x_\infty,1,-\log\theta')\times(t^*,1])$ satisfying $\square u_\infty=0$. The inequality (\ref{ineq:ui-to-zero}) implies $u_\infty\leq 0$. Since $\mathcal{N}^*_0(x_i,1)\geq-\delta_i\to 0$, the first inequality in (\ref{ineq:ui-estimates}) tells us $u_\infty(x_\infty,1)\geq 0$. An application of the strong maximum principle then yields $u_\infty\equiv 0$. Now the second inequality in (\ref{ineq:ui-estimates}) implies that for any $\beta>0$ with $t^*+\beta\leq 1$ we have
	\begin{equation*}
		\mathcal{N}^*_0
		\geq
		u_i-\frac{7}{2}\log\frac{t^*+\beta}{t^*}
		\geq
		u_i-\frac{7}{2t^*}\beta
		\geq
		u_i-\frac{7}{2\theta'}\beta.
	\end{equation*}
	Since the choices of $t^*,\beta$ were arbitrary, our claim holds for any $\theta''\in(\theta',1]$. As we can choose appropriate $\theta$ to make $\theta'<\theta_0$, this contradicts the definition of $\theta_0$.
\end{proof}

The following theorem provides some sort of converse to Theorem \ref{thm:rough-epsilon-regularity}.
\begin{Theorem}
	For any $\epsilon>0$ and $Y<\infty$, there exists $\delta>0$ depending only on $\epsilon,Y$ such that the following holds. Let $(M^7,\varphi(t),g(t))_{t\in I}$ be a Laplacian flow over a compact manifold satisfying the uniform torsion bound $|T_\varphi|^2\leq\Lambda$ for some $\Lambda>0$. Let $(x,t)\in M\times I$ and $r>0$ be such that $[t-r^2,t]\subset I$. Suppose $r\leq\delta\Lambda^{-1/2}$, $|\Rm|\leq r^{-2}$ on $P(x,t;r,-r^2)$ and $\mathcal{N}_{x,t}(r^2)\geq -Y$. Then $\mathcal{N}_{x,t}(\delta r^2)\geq-\epsilon$.
\end{Theorem}
\begin{proof}
	The proof is identical to \cite{bamler2020entropy} Theorem 10.4.
\end{proof}

\subsection{Proof of \hyperref[thm:refined-pseudolocality]{Theorem~\ref*{thm:refined-pseudolocality}}}\label{sec:pseudolocality-reduction}
At last, we give the promised reduction of pseudolocality to $\epsilon$-regularity of pointed Nash entropy.
\begin{proof}[Proof of {\hyperref[thm:refined-pseudolocality]{Theorem~\ref*{thm:refined-pseudolocality}}}]
	Choose an auxiliary constant $\bar\delta$ so that the hypothesis implies
	\begin{equation*}
		\inf_{0<t\leq T}
		\mu\big(B(x_0,0,40\bar\delta^{-1}\sqrt{t}),g(0),t\big)
		\geq-\bar\delta.
	\end{equation*}
	We relabel $\bar\delta$ as $\delta$ below.
	For $A:=\frac{2}{3}\Lambda$, we have $E(g)\geq-Ag$.
	Assume (\ref{ineq:pseudolocality}) does not hold. Then there exists $t^*\in(0,T]$ such that
	\begin{equation*}
		\begin{dcases}
			|\Rm|(x,t)\leq\frac{\alpha}{t},
			\quad\forall\, x\in B(x_0,t,\alpha^{-1}\sqrt{t}),\ t\in(0,t^*];\\
			\exists\, z\in B(x_0,t^*,\alpha^{-1}\sqrt{t^*})\text{ s.t. } |\Rm|_{g}(z,t^*)=\frac{\alpha}{t^*}.
		\end{dcases}
	\end{equation*}
	Since $\alpha\leq 1$ and $\bar{\sigma}_A(t)\leq t$ for all $t\geq 0$, we have (for $\delta$ sufficiently small, say $\delta\leq\frac{1}{n-1}$)
	\begin{equation*}
		\Ric(x,t)\leq\frac{\delta^{-1}}{\bar{\sigma}_A(t)},
		\quad\forall\,x\in B(x_0,t,\sqrt{\bar{\sigma}_A(t)}),\ t\in(0,t^*].
	\end{equation*}
	Let $\gamma_0>0$ be such that $e^{\gamma_0}-1=2\gamma_0$. Then for $AT\leq\gamma_0$ we have $\sigma_A(t)\leq 2t$, $\forall\,t\in(0,T]$. Also, when $\delta\leq\frac{1}{3(n-1)}$, it holds $(\frac{2}{3}\delta^{-1}+n-1)\leq\delta^{-1}$. By \hyperref[thm:mu-and-N]{Theorem~\ref*{thm:mu-and-N}} (taking $\beta=1$),
	\begin{equation}\label{ineq:estimate-on-N}
		\mathcal{N}_{x_0,t^*}(t^*)\geq-O(\delta)
	\end{equation}
	Taking $\epsilon=\frac{1}{2}\alpha$ in \hyperref[thm:refined-epsilon-regularity]{Theorem~\ref*{thm:refined-epsilon-regularity}}, there exists $\delta_0(n,\alpha)$ such that $\mathcal{N}_{x_0,t^*}(t^*)\geq-\delta_0$ will imply
	\begin{equation*}
		|\Rm|\leq\frac{\alpha}{2t^*}
		\qquad\text{over }B(x_0,t^*,\alpha^{-1}\sqrt{t^*}).
	\end{equation*}
	So if at the beginning $\delta$ is chosen small enough to make (\ref{ineq:estimate-on-N}) no less than $-\delta_0$, we will get a contradiction against the existence of $z$, concluding the proof of (\ref{ineq:pseudolocality}).
	
	Using the curvature estimate (\ref{ineq:pseudolocality}), we may apply \hyperref[thm:mu-and-N]{Theorem~\ref*{thm:mu-and-N}} to show $\mathcal{N}_{x_0,t}(t)\geq-\delta_0$ for all $t\in(0,T]$ as above.
	Then (\ref{ineq:volume-ratio-and-inj}) follows from \hyperref[thm:refined-epsilon-regularity]{Theorem~\ref*{thm:refined-epsilon-regularity}}.
\end{proof}

\newpage
\appendix
In the appendices, we consider a compact manifold $M^n$ and a smooth family $(g(t))_{t\in I}$ of Riemannian metrics on it.
We always assume implicitly that the time in consideration is contained in the interval $I$.
Recall the definition (\ref{eq:definition-of-error-term}) of $E(g)$.

\section{Preliminaries}
\subsection{Heat operators and heat kernels}
Let us recall the definition and some basic properties of heat operator and conjugate heat operator.
Set $\square=\partial_t-\Delta$ and $\square^*=-\partial_t-\Delta+R-\frac{1}{2}\mathrm{tr}\,E$. Then for $u,v\in C^\infty(M\times I)$, we have
\begin{equation}\label{eq:conjugate-heat-operator}
	\frac{d}{dt}\int_M uv\,dV_g
	=\int_M(\square u)v\,dV_g
	-\int_Mu(\square^*v)\,dV_g.
\end{equation}
In particular, if $\square^*v=0$, then $\int_Mv\,dV_g$ is constant in time.
Let $\tau=\tau(t)>0$ evolve as $\frac{d\tau}{dt}=-1$.
Then $\square^*((4\pi\tau)^{-n/2}e^{-f})=0$ is equivalent to
\begin{equation}\label{eq:conjugate-heat-equation}
	\frac{\partial f}{\partial t}=-\Delta f+\abs{\nabla f}^2-R+\frac{1}{2}\mathrm{tr}\,E+\frac{n}{2\tau}.
\end{equation}

For any $x,y\in M$ and $s,t\in I$ with $s<t$, we denote by $K(x,t;y,s)$ the heat kernel of $\square$. That is, for fixed $(y,s)$,
\begin{equation*}
	\square K(\cdot,\cdot;y,s)=0,\qquad
	\lim_{t\downarrow s}K(\cdot,t;y,s)=\delta_y,
\end{equation*}
and for fixed $(x,t)$,
\begin{equation*}
	\square^* K(x,t;\cdot,\cdot)=0,\qquad
	\lim_{s\uparrow t}K(x,t;\cdot,s)=\delta_x.
\end{equation*}

Given $(x,t)\in M\times I$ and $s\in I$ with $s\leq t$, we denote by $\nu_{x,t;s}$ (or sometimes $\nu_{x,t}(s)$) the conjugate heat kernel measure based at $(x,t)$, i.e.,
\begin{equation*}
	d\nu_{x,t;s}:=K(x,t;\cdot,s)dV_{g(s)},\qquad
	\nu_{x,t;t}:=\delta_x.
\end{equation*}

\subsection{Two auxiliary functions}
For $A\geq 0$, we define
\begin{equation}\label{eq:definition-of-auxiliary-functions}
	\sigma_A(\tau):=\frac{e^{A\tau}-1}{A},
	\qquad
	\bar{\sigma}_A(t):=\frac{1-e^{-At}}{A},
\end{equation}
with the understanding that $\sigma_0(\tau):=\tau$ and $\bar{\sigma}_0(t):=t$.
Note that $\bar{\sigma}_A(t)=-\sigma_A(-t)$.

In this paper, the parameter $A$ appearing in (\ref{eq:definition-of-auxiliary-functions}) is always taken to be the constant from our assumption $E(g)\geq-Ag$.
For $\tilde{g}(t):=\lambda g(\lambda^{-1}t)$ with $\lambda>0$, we have $E(\tilde{g})\geq-(\lambda^{-1}A)\tilde{g}$.
In other words, the dimension of $A$ is $[\text{length}]^{-2}=[\text{time}]^{-1}$.
Consequently, both $\sigma_A(\tau)$ and $\bar{\sigma}_A(t)$ are of dimension $[\text{time}]$.

When $\tau$ is sufficiently small (compared to $A^{-1}$), these two functions will be comparable to the function $f(x)=x$.
For example, since $\frac{e^x-1}{x}\leq\frac{1}{1-x}$ for $x\in(0,1)$, we have $\sigma_A(\tau)\leq\frac{1}{1-A\tau}\tau$, which gives $\sigma_A(\tau)\leq 2\tau$ when $\tau\in[0,\frac{1}{2A}]$.
Note also $\sigma_A(\tau)\geq\tau$ for all $\tau\geq 0$.
When these two auxiliary functions appear in our results, such comparability justifies our philosophy that flows with $E(g)\geq-Ag$ behave like super Ricci flow over a small but definite spacetime scale.

Define $\partial_t^A:=e^{-At}\partial_t(e^{At}\cdot)=\partial_t+A$. Then $\bar{\sigma}_A$ solves the equation
\begin{equation*}
	2\partial_t^{A/2}\sqrt{\bar{\sigma}_A(t)}=\frac{1}{\sqrt{\bar{\sigma}_A(t)}}.
\end{equation*}

\subsection{Variance of conjugate heat kernels}
This subsection corresponds to $\S 3$ of \cite{bamler2020entropy}. Since the required modification is rather straightforward, we summarize the results in this preliminary section and only sketch the proofs.

The variance between two probability measures $\mu_1,\mu_2$ on a Riemannian manifold $(M^n,g)$ is defined as
\begin{equation*}
	\mathrm{Var}(\mu_1,\mu_2)
	:=
	\int_M\int_M   d^2(x_1,x_2)\,d\mu_1(x_1)d\mu_2(x_2).
\end{equation*}
In the case $\mu_1=\mu_2=\mu$, we write $\mathrm{Var}(\mu)$ for $\mathrm{Var}(\mu,\mu)$. In the context of flow, we write $\mathrm{Var}_t$ for the variance with respect to the metric $g(t)$.

\begin{Theorem}\label{thm:square-of-distance}
	Suppose $E(g)\geq-Ag$ for some $A\geq 0$. Then
	\begin{equation*}
		(\partial_t-\Delta_x-\Delta_y)d^2_{t}(x,y)\geq-\frac{(n-1)\pi^2}{2}-4-Ad^2_t(x,y)
	\end{equation*}
	in the barrier and viscosity sense.
\end{Theorem}
\begin{Corollary}
	Consider two non-negative solutions $v_1,v_2$ over $M$ to the conjugate heat equation $\square^*v_i=0$ with $\int_Mv_i\,dV_g=1$. Set $\mu_i=v_idV_g$ and $V(t)=\mathrm{Var}_t(\mu_1(t),\mu_2(t))$. Then for $t\leq t_0$,
	\begin{equation*}
		V(t)
		\leq
		e^{A(t_0-t)}V(t_0)
		+H_n\sigma_A(t_0-t).
	\end{equation*}
\end{Corollary}
\begin{proof}
	If we pretend \hyperref[thm:square-of-distance]{Theorem~\ref*{thm:square-of-distance}} holds pointwisely, then
	\begin{equation*}
		\frac{d}{dt}V(t)\geq-AV(t)-H_n
	\end{equation*}
	and our corollary follows by integration.
	This argument can be turned rigorous as in \cite{bamler2020entropy}.
\end{proof}

\begin{Corollary}
	For conjugate heat kernel measures,
	\begin{align*}
		\mathrm{Var}_t(\nu_{x_1,t_0;t},\nu_{x_2,t_0;t})
		&\leq e^{A(t_0-t)}d^2_{t_0}(x_1,x_2)+H_n\sigma_A(t_0-t);\\
		\mathrm{Var}_t(\nu_{x_0,t_0;t})&\leq H_n\sigma_A(t_0-t).
	\end{align*}
\end{Corollary}

\begin{Definition}
	A point $(z,t)\in M\times I$ is called an $(H_n,A)$-center of the point $(x_0,t_0)\in M\times I$ if $\mathrm{Var}_t(\delta_z,\nu_{x_0,t_0;t})\leq H_n\sigma_A(t_0-t)$.
\end{Definition}

Here are some basic properties of $(H_n,A)$-center.
\begin{Proposition}\label{prop:basic-properties-of-center}
	Suppose $E(g)\geq-Ag$ for some $A\geq 0$.
	\begin{enumerate}
		\item Given $(x_0,t_0)$ and $t\leq t_0$, there exists $z\in M$ so that $(z,t)$ is an $(H_n,A)$-center of $(x_0,t_0)$.
		
		\item For any two $(H_n,A)$-center $(z_1,t)$ and $(z_2,t)$ at time $t$, $d_t(z_1,z_2)\leq 2\sqrt{H_n\sigma_A(t_0-t)}$.
		
		\item Let $(z,t)$ be an $(H_n,A)$-center of $(x_0,t_0)$. Then for $\delta>0$,
		\begin{equation*}
			\nu_{x_0,t_0;t}\big(B(z,t,\sqrt{\delta H_n\sigma_A(t_0-t)})\big)
			\geq
			1-\frac{1}{\delta}.
		\end{equation*}
		
		\item Suppose $(z,t)$ is an $(H_n,A)$-center of $(x_0,t_0)$, then for all $r\geq 0$ and $\epsilon>0$, we have the following stronger integral Gaussian bound
		\begin{align*}
			\nu_{x_0,t_0;t}\big(M\setminus B(z,t,r)\big)
			&\leq
			2\exp\bigg(
			-\frac{\big(r-\sqrt{2H_n\sigma_A(t_0-t)}\big)_+^2}{8\sigma_A(t_0-t)}
			\bigg)\\
			&\leq
			C(n,\epsilon)\exp\bigg(
			-\frac{r^2}{(8+\epsilon)\sigma_A(t_0-t)}
			\bigg).
		\end{align*}
	\end{enumerate}
\end{Proposition}
\begin{proof}
	Assertion (4) follows from \hyperref[cor:Gaussian-Concentration]{Corollary~\ref*{cor:Gaussian-Concentration}}, which will be proved later.  The estimates in (4) are useful only for large $r$.
\end{proof}

\subsection{Local entropy}
In \cite{wang2018local,wang2020local}, the entropy functionals of Perelman are localized and used to generalize Perelman's no local collapsing theorem and pseudo-locality theorem.
Here we shall quickly review their definitions.

Let $(M^n,g)$ be a complete Riemannian manifold and $\Omega$ be a domain with smooth boundary. Set
\begin{equation*}
	\mathscr{S}(\Omega)
	:=
	\left\{\varphi\in W_0^{1,2}(\Omega):\varphi\geq 0, \int_\Omega\varphi^2dV_g=1\right\}.
\end{equation*}
For $\varphi\in\mathscr{S}(\Omega)$ and $\tau>0$, the local $\mathcal{W}$-functional is defined by
\begin{equation*}
	\mathcal{W}(\Omega,g,\varphi,\tau)
	:=
	-n-\frac{n}{2}\log(4\pi\tau)
	+\int_\Omega
	\big(\tau(4|\nabla\varphi|^2+R\varphi^2)-\varphi^2\log(\varphi^2)\big)dV_g.
\end{equation*}
If $\Omega$ is taken to be the whole manifold $M$ and $\varphi^2=(4\pi\tau)^{-n/2}e^{-f}$, then we recover Perelman's $\mathcal{W}$-functional
\begin{equation*}
	\mathcal{W}(g,f,\tau)
	=
	\int_M
	\big(
	\tau(|\nabla f|^2+R)+f-n
	\big)
	(4\pi\tau)^{-n/2}e^{-f}dV_g.
\end{equation*}
The local $\mu$-functional and $\nu$-functional are defined by
\begin{equation*}
	\mu(\Omega,g,\tau)
	:=
	\inf_{\varphi\in\mathscr{S}(\Omega)}\mathcal{W}(\Omega,g,\varphi,\tau),
	\qquad
	\nu(\Omega,g,\tau)
	:=
	\inf_{s\in(0,\tau]}\mu(\Omega,g,s).
\end{equation*}
Suppose $\Omega_1$ and $\Omega_2$ are two domains with $\Omega_1\subset\Omega_2$, then obviously
\begin{equation*}
	\mu(\Omega_1,g,\tau)
	\geq
	\mu(\Omega_2,g,\tau),
	\qquad
	\nu(\Omega_1,g,\tau)
	\geq
	\nu(\Omega_2,g,\tau).
\end{equation*}

\subsection{Laplacian flow}
Some basic definitions and facts concerning $G_2$-structure and Laplacian flow are given in the introduction \hyperref[sec:introduction-Laplacian-flow]{Section~\ref*{sec:introduction-Laplacian-flow}}.
Here is a brief supplement.
We refer the readers to \cite{lotay2017laplacian} for more information.

We will often write a Laplacian flow as $(M^7,\varphi(t),g(t))_{t\in I}$.
By this we mean $g(t)$ is the metric associated to the $G_2$-structure $\varphi(t)$.
The evolution equation of $g(t)$ is given by
\begin{equation*}
	\frac{\partial}{\partial t}g_{ij}
	=
	-2R_{ij}-\frac{2}{3}|T|^2g_{ij}-4{T_i}^kT_{kj}.
\end{equation*}
Since the torsion $T$ is a skew-symmetric $(0,2)$-tensor, the term $-4{T_i}^kT_{kj}$ is non-negative definite.
Suppose $|T|^2\leq\Lambda$ for some $\Lambda\geq 0$, then $E(g)\geq-\frac{2}{3}\Lambda g$.
Recall for closed $G_2$-structure we have $R=-|T|^2$. Hence the scalar curvature will be bounded if the torsion tensor is.

Next we recall the Shi-type estimates and compactness theorem established in \cite{lotay2017laplacian}.
\begin{Theorem}
	Let $(M^7,\varphi(t),g(t))_{t\in[0,\frac{1}{\Lambda}]}$ be a Laplacian flow over a compact manifold.
	\begin{enumerate}
		\item (Global version). Suppose $|\Rm|\leq\Lambda$ over $M\times[0,\frac{1}{\Lambda}]$, then
		\begin{equation*}
			|\nabla^k\Rm|\leq C(k)\Lambda t^{-k/2}
			\qquad\text{over }
			M\times[0,\tfrac{1}{\Lambda}].
		\end{equation*}

		\item (Local version). Suppose $x\in M$ and $|\Rm|\leq\Lambda$ over $B_{g(0)}(x,r)\times[0,\frac{1}{\Lambda}]$, then
		\begin{equation*}
			|\nabla^k\Rm|\leq C(\Lambda,r,k)t^{-k/2}
			\qquad\text{over }
			B_{g(0)}(x,\tfrac{r}{2})\times[0,\tfrac{1}{\Lambda}].
		\end{equation*}
	\end{enumerate}
\end{Theorem}

\begin{Theorem}\label{thm:compactness}
	Let $(M^7_i,\varphi_i(t),g_i(t),x_i)_{t\in(0,T]}$ be a sequence of pointed Laplacian flows over compact manifolds.
	Suppose $|B_{g_i(T)}(x_i,1)|_{g_i(T)}\geq v$ and
	\begin{equation}\label{eq:compactness-curvature-condition}
		|\Rm|(g_i)\leq\Lambda\quad\text{over }M_i\times(0,T]
	\end{equation}
	for some $\Lambda<\infty$ and $v>0$.
	Then there exists a subsequence converges smoothly to some Laplacian flow $(M^7_\infty,\varphi_\infty(t),g_\infty(t),x_\infty)_{t\in(0,T]}$ with complete time-slices.

	The conclusion remains valid when (\ref{eq:compactness-curvature-condition}) is replaced by the following local curvature bounds.
	There exists a sequence $\rho_i\to\infty$ such that, for each compact interval $J\Subset(0,T]$, we have
	\begin{equation*}
		|\Rm|\leq C_J,\quad
		\text{on }B_{g_i(t)}(x_i,\rho_i),\forall\,t\in J
	\end{equation*}
	for some constant $C_J<\infty$.
\end{Theorem}
\begin{proof}
	This follows from \cite[Theorem 7.2]{lotay2017laplacian} and the injective radius estimate of Cheeger--Gromov--Taylor.
\end{proof}

\subsection{Difficulty towards pseudolocality}\label{sec:preliminaries-difficulty}
Here we briefly explain the difficulty in applying the proof for Ricci flow directly to our setting.

Let $\tau=\tau(t)>0$ evolve as $\frac{d\tau}{dt}=-1$ and $u=(4\pi\tau)^{-n/2}e^{-f}$ solve $\square^*u=0$. Recall Perelman's Harnack quantity is defined by
\begin{equation*}
	v:=\big\{\tau(2\Delta f-|\nabla f|^2+R)+f-n\big\}u.
\end{equation*}
Direct computation gives (one can see \cite{ishida2024logarithmic} for more details)
\begin{equation*}
	\square^*v
	=
	-\bigg\{
	2\tau\Big|\Ric+\nabla^2f-\frac{g}{2\tau}-\frac{E}{4}\Big|^2
	-\frac{\tau}{8}|E|^2
	+\tau\Big[-\langle E,\nabla^2f\rangle
	+2\langle\delta E,\nabla f\rangle
	+E(\nabla f,\nabla f)
	+\delta^2 E
	\Big]
	\bigg\}u.
\end{equation*}
To the author's knowledge, up to now the known proofs of pseudolocality for Ricci flow (e.g. \cite{perelman2002entropy,wang2020local}) all depend on the pointwise estimate $\square^*v\leq 0$, which is not available for Laplacian flow.
In fact, the proof in \cite{wang2020local} only require a local integral bound on $v$ to establish an effective monotonicity formula for the local $\mu$-functional.
But this still does not work for Laplacian flow.
Let $h:M\to[0,1]$ be compactly supported.
Integration by parts shows that
\begin{align*}
	\int_M(\delta^2E)hu\,dV_g
	&=
	\int_M
	\big\{E_{ij}(-f_{ij}+f_if_j)h-2E_{ij}f_ih_j+E_{ij}h_{ij}\big\}
	u\,dV_g,\\
	\int_M2\langle\delta E,\nabla f\rangle hu\,dV_g
	&=
	\int_M\big\{2E_{ij}f_{ij}h+2E_{ij}h_if_j-2E_{ij}f_if_jh\big\}u\,dV_g.
\end{align*}
Therefore
\begin{align*}
	\int_M(\square^*v)h\,dV_g
	&=-\int_M
	\bigg\{
	2\tau\Big|\Ric+\nabla^2f-\frac{g}{2\tau}-\frac{E}{4}\Big|^2
	-\frac{\tau}{8}|E|^2
	\bigg\}h\,d\nu
	-\tau\int_M\inn{E}{\nabla^2h}\,d\nu\\
	&\leq
	\frac{\tau}{8}|E|^2_{\mathrm{max}}\int_Mh\,d\nu
	+\tau|E|_{\mathrm{max}}\int_M|\nabla^2h|\,d\nu.
\end{align*}
The appearance of $\nabla^2h$ makes the above expression difficult to control.

Taking $h\equiv 1$, we recover the effective monotonicity formula of $\mathcal{W}$-functional in \cite{chen2018shi}.
\begin{Proposition}
	Let $\tau=\tau(t)>0$ evolve as $\frac{d\tau}{dt}=-1$ and $f$ evolve as (\ref{eq:conjugate-heat-equation}). Then
	\begin{equation}\label{eq:derivative-of-W}
		\frac{d}{dt}\mathcal{W}(g,f,\tau)
		=
		\int_M
		\Big(2\tau\Big|\Ric+\nabla^2f-\frac{g}{2\tau}-\frac{E}{4}\Big|^2   -\frac{\tau}{8}|E|^2\Big)   (4\pi\tau)^{-n/2}e^{-f}dV_g.
	\end{equation}
\end{Proposition}

\section{Gradient Estimates}\label{sec:gradient-estimates}
We first record several evolution equations for future reference.
\begin{Proposition}[Evolution equation]\label{prop:evolution-of-various-quantities}
	For any spacetime function $u$ and $\phi,\psi:\R\to\R$,
	\begin{align*}
		\tfrac{1}{2}\square|\nabla u|^2
		=&-|\nabla^2u|^2  -\tfrac{1}{2}E^{ij}u_iu_j  +\nabla\square u\cdot\nabla u,  \\
		\square\phi(u)
		=&-\phi''(u)|\nabla u|^2+\phi'(u)\square u,  \\
		\square(\psi(u)|\nabla u|^2)
		=&-2\psi(u)|\nabla^2u|^2-4\psi'(u)u_{ij}u_iu_j-\psi''(u)|\nabla u|^4-\psi(u)E^{ij}u_iu_j  \\
		&+\psi'(u)|\nabla u|^2\square u+2\psi(u)\nabla\square u\cdot\nabla u.
	\end{align*}
\end{Proposition}

\begin{Corollary}\label{cor:evolution}
	Suppose $\partial_tu=\Delta u$, $u>0$. Then
	\begin{align}
		\square(u\log u)
		&=-\frac{|\nabla u|^2}{u},   \label{eq:ulogu}\\
		\square\frac{|\nabla u|^2}{u}
		&=-\frac{E^{ij}u_iu_j}{u}-\frac{2}{u}\Big|\nabla^2u-\frac{du\otimes du}{u}\Big|^2
		=-\frac{E^{ij}u_iu_j}{u}-2u|\nabla^2\log u|^2. \notag
	\end{align}
\end{Corollary}

\begin{Proposition}\label{prop:gradient-estimate-I}
	Suppose $E(g)\geq -Ag$ for some $A\geq 0$. Then for any solution $u$ to the heat equation $\square u=0$, we have $\square|\nabla u|\leq\frac{A}{2}|\nabla u|$. In particular, if $u(\cdot,t_0)$ is $L$-Lipschitz, then $u(\cdot,t)$ is $(e^{\frac{A}{2}(t-t_0)}L)$-Lipschitz for $t>t_0$.
\end{Proposition}
\begin{proof}
	By \hyperref[prop:evolution-of-various-quantities]{Proposition~\ref*{prop:evolution-of-various-quantities}}, $\square |\nabla u|^2=-E^{ij}u_iu_j-2|\nabla^2u|^2$.
	By Kato's inequality,
	\begin{align*}
		\square|\nabla u|&=\frac{1}{4|\nabla u|^3}\big|\nabla|\nabla u|^2\big|^2+\frac{1}{2|\nabla u|}\square|\nabla u|^2\\
		&=\frac{|\nabla|\nabla u||^2}{|\nabla u|}-\frac{|\nabla^2u|^2}{|\nabla u|}-\frac{E^{ij}u_iu_j}{2|\nabla u|}
		\leq-\frac{E^{ij}u_iu_j}{2|\nabla u|}
		\leq
		\frac{A}{2}|\nabla u|.
	\end{align*}
	The second assertion follows from maximum principle.
\end{proof}

\begin{Definition}
	Let $(M^n,g)$ be a complete Riemannian manifold. The $W_1$-distance between two probability measures $\mu_1,\mu_2$ over $M$ is defined by
	\begin{equation}\label{eq:definition-of-W1}
		d^g_{W_1}(\mu_1,\mu_2)
		:=
		\sup_f\left(\int_M f \,d\mu_1-\int_M f \,d\mu_2\right),
	\end{equation}
	where the supremum is taken over all bounded $1$-Lipschitz functions $f:M\to\R$. Equivalently, we may take the supremum over all bounded smooth function with $|\nabla f|\leq 1$.
\end{Definition}

\begin{Corollary}\label{cor:monotonicity-of-W1}
	Let $(v_i)_{i=1,2}$ be two nonnegative solutions to the conjugate heat equation $\square^*v_i=0$ such that $\int_Mv_i\,dV_g=1$.
	Let $d\mu_i(t)=v_i(\cdot,t)dV_{g(t)}$ be the associated probability measures.
	Suppose $E(g)\geq -Ag$ for some $A\geq 0$, then $e^{\frac{A}{2}t} d_{W_1}^g(\mu_1,\mu_2)(t)$ is non-decreasing, i.e.,
	\begin{equation}\label{ineq:monotonicity-of-W1}
		d_{W_1}^{g(t_1)}(\mu_1(t_1),\mu_2(t_1))
		\leq
		e^{\frac{A}{2}(t_2-t_1)}   d_{W_1}^{g(t_2)}(\mu_1(t_2),\mu_2(t_2)),
		\quad\forall\,t_1\leq t_2.
	\end{equation}
	In particular, for any $x_1,x_2\in M$ and $t_0\in I$, it holds
	\begin{equation*}
		d_{W_1}^{g(t)}(\nu_{x_1,t_0;t},\nu_{x_2,t_0;t})
		\leq
		e^{\frac{A}{2}(t_0-t)}   d_{t_0}(x_1,x_2),
		\quad\forall\,t\leq t_0.
	\end{equation*}
\end{Corollary}
\begin{proof}
	Consider $f\in C^\infty(M)$ satisfying $|\nabla f|_{g(t_1)}\leq 1$. Let $u\in C^\infty(M\times[t_1,t_2])$ be the solution to the heat equation $\square u=0$ with initial condition $u(\cdot,t_1)=f$. By \hyperref[prop:gradient-estimate-I]{Proposition~\ref*{prop:gradient-estimate-I}}, we have $|\nabla u|_{g(t)}\leq e^{\frac{A}{2}(t-t_1)}$. Therefore
	\begin{align*}
		\int_M f \,d\mu_1(t_1)-\int_M f \,d\mu_2(t_1)
		&=
		\int_M u(\cdot,t_2) \,d\mu_1(t_2)-\int_M u(\cdot,t_2) \,d\mu_2(t_2)\\
		&\leq
		e^{\frac{A}{2}(t_2-t_1)}   d_{W_1}^{g(t_2)}(\mu_1(t_2),\mu_2(t_2)).
	\end{align*}
	Then (\ref{ineq:monotonicity-of-W1}) follows by taking supremum over all such $f$.
\end{proof}

We have the following straightforward generalization of \cite[Theorem 3.2]{zhang2006some}.
One can compare this with \hyperref[thm:heat-kernel-gradient-estimate]{Theorem~\ref*{thm:heat-kernel-gradient-estimate}}.
\begin{Proposition}\label{prop:gradient-estimate-Zhang}
	Suppose $E(g)\geq -Ag$ for some $A\geq 0$.
	Then for any positive solution $u$ to the heat equation $\square u=0$ over $C^\infty(M\times(s,t])$, we have
	\begin{equation*}
		\frac{|\nabla u|}{u}(\cdot,t)
		\leq
		\sqrt{\frac{1+A(t-s)}{t-s}}\sqrt{\log\frac{u_{\mathrm{max}}}{u(\cdot,t)}},
	\end{equation*}
	where $u_{\mathrm{max}}=\sup_{M\times(s,t]}u$.
\end{Proposition}
\begin{proof}
	We first deal with the case $u\in C^\infty(M\times[s,t])$. By parabolic rescaling and time-shifting, we may assume $[s,t]=[0,1]$. By \hyperref[cor:evolution]{Corollary~\ref*{cor:evolution}},
	\begin{equation*}
		\square\left(
		t\frac{|\nabla u|^2}{u}-(1+A)u\log\frac{u_{\mathrm{max}}}{u}
		\right)
		\leq
		-(1-t)A  \frac{|\nabla u|^2}{u}
		-2t  u|\nabla^2\log u|^2
		\leq 0.
	\end{equation*}
	The result follows by maximum principle. The general case can be proved by applying what we have obtained over the time interval $[s+\epsilon,t]$ and taking $\epsilon\to 0$.
\end{proof}

Let $\Phi:\R\to(0,1)$ denote the function satisfying
\begin{equation*}
	\Phi'(x)=(4\pi)^{-1/2} e^{-x^2/4},
	\qquad \lim_{x \to -\infty}\Phi(x)=0,
	\qquad \lim_{x \to \infty} \Phi(x)=1.
\end{equation*}
Set $\Phi_t(x):=\Phi(t^{-1/2}x)$.
Then $\Phi_t(x)$ solves the 1-dimensional heat equation with initial condition $\mathds{1}_{[0,\infty)}$.
A direct computation gives $\Phi_t'(x)=(4\pi t)^{-1/2}e^{-x^2/4t}$.
\begin{Theorem}[Improved gradient estimate]\label{thm:improved-gradient-estimate}
	Suppose $E(g)\geq-Ag$ for some $A\geq 0$.
	Assume $u\in C^\infty(M\times(t_0,t_1])$ satisfies $\partial_tu=\Delta u$ and $0<u<1$.
	Then $|\nabla\Phi^{-1}_{t-t_0}(u(\cdot,t))|\leq\sqrt{1+A(t-t_0)}$ for all $t\in(t_0,t_1]$.
\end{Theorem}
\begin{proof}
	We first prove the case $u\in C^\infty(M\times[t_0,t_1])$. Choose $T$ so that $T+t>0$ for $t\in[t_0,t_1]$. Define $h$ by $u(x,t)=\Phi_{T+t}(h(x,t))$. Note here the definition of $h$ depends on the choice of $T$. So it is more accurate to write $h_T(x,t)$. Then
	\begin{equation*}
		\frac{1}{2}\partial_t|\nabla h|^2
		=
		-\frac{1}{2}E^{ij}h_ih_j+\frac{1}{2}\Delta|\nabla h|^2-|\nabla^2 h|^2-\frac{1}{2(T+t)}\nabla h\cdot h\nabla|\nabla h|^2+\frac{1}{2(T+t)}(1-|\nabla h|^2)|\nabla h|^2.
	\end{equation*}
	When combined with the assumption $E\geq-Ag$, this yields
	\begin{equation*}
		\frac{1}{2}\partial_t|\nabla h|^2
		\leq
		\frac{1}{2}\Delta|\nabla h|^2-\frac{1}{2(T+t)}\nabla h\cdot h\nabla|\nabla h|^2
		+\frac{1}{2}|\nabla h|^2\Big(A+\frac{1}{T+t}(1-|\nabla h|^2)\Big).
	\end{equation*}
	By the maximum principle, if $|\nabla h|(t_0)\leq\sqrt{1+A(T+t_0)}$, then $|\nabla h|(t)\leq\sqrt{1+A(T+t)}$ for all $t\in[t_0,t_1]$. That is, whenever $T+t_0>0$,
	\begin{equation*}
		|\nabla\Phi_{T+t_0}^{-1}(u(\cdot,t_0))|\leq\sqrt{1+A(T+t_0)}
		\implies
		|\nabla\Phi_{T+t}^{-1}(u(\cdot,t))|\leq \sqrt{1+A(T+t)},
		\quad\forall\,t\in[t_0,t_1].
	\end{equation*}
	Equivalently, if $|\nabla\Phi_{\delta}^{-1}(u(\cdot,t_0))|\leq\sqrt{1+A\delta}$ for some $\delta>0$, then
	\begin{equation}\label{ineq:improved-gradient-estimate}
		|\nabla\Phi_{\delta+t-t_0}^{-1}(u(\cdot,t))|\leq \sqrt{1+A(\delta+t-t_0)},
		\quad\forall\,t\in[t_0,t_1].
	\end{equation}
	
	As $\delta\to 0$, $\Phi_\delta'(x)=(4\pi\delta)^{-1/2}e^{-x^2/4\delta}$ tends to the point mass at $0$. Therefore the corresponding distribution function $\Phi_\delta$ tends to $\mathds{1}_{[0,\infty)}$. In $\R^2$, the graph of $\Phi_\delta^{-1}$ is simply the reflection of that of $\Phi_\delta$ along the line $y=x$. Thus $\Phi_\delta^{-1}\to 0$ uniformly over any compact interval $[\epsilon,1-\epsilon]$. By the compactness of $M$, there exists $\epsilon>0$ and $C<\infty$ such that $u(\cdot,t_0)\in[\epsilon,1-\epsilon]$ and $|\nabla u(\cdot,t_0)|\leq C$. Therefore $|\nabla\Phi_{\delta}^{-1}(u(\cdot,t_0))|=(\Phi_\delta^{-1})'(u(\cdot,t_0))|\nabla u(\cdot,t_0)|\to 0$ uniformly as $\delta\to 0$. In sum, the condition $|\nabla\Phi_{\delta}^{-1}(u(\cdot,t_0))|\leq\sqrt{1+A\delta}$ is always satisfied for sufficiently small $\delta$.
	
	The result (when $u$ extends smoothly to time $t=t_0$) follows by taking $\delta\to 0$ in (\ref{ineq:improved-gradient-estimate}). For the general case, apply what we have obtained to $u$ over the time interval $[t_0+\epsilon,t_1]$. Then we get for any $t\in(t_0+\epsilon,t_1]$ that
	\begin{equation*}
		|\nabla\Phi^{-1}_{t-t_0-\epsilon}(u(\cdot,t))|
		\leq
		\sqrt{1+A(t-t_0-\epsilon)}.
	\end{equation*}
	For each fixed $t\in(t_0,t_1]$, we can obtain our desired result by taking $\epsilon\to 0$.
\end{proof}
\begin{Remark}
	Set $f(x,t)=\Phi^{-1}_{t-t_0}(u(x,t))$. Then
	\begin{equation*}
		|\nabla u|(x,t)
		\leq
		\sqrt{1+A(t-t_0)}\Phi'_{t-t_0}(f(x,t)).
	\end{equation*}
	In particular,
	\begin{equation*}
		|\nabla u|(x,t)
		\leq
		\sqrt{1+A(t-t_0)}
		\frac{1}{\sqrt{4\pi(t-t_0)}}
		\exp\left(-\frac{f^2(x,t)}{4(t-t_0)}\right)
		\leq
		\frac{1}{\sqrt{4\pi}}
		\sqrt{\frac{1}{t-t_0}+A}.
	\end{equation*}
\end{Remark}

\begin{Proposition}[Integral gradient estimates on heat kernel]\label{prop:integral-gradient-bound}
	Suppose $E(g)\geq-Ag$ for some $A\geq 0$.
	Then for any $1\leq p<\infty$ and measurable subset $X\subset M$, we have
	\begin{align}
		&(t-s)^{p/2}\int_X\Big(\frac{|\nabla_xK(x,t;y,s)|}{K(x,t;y,s)}\Big)^pd\nu_{x,t;s}(y)  \notag  \\
		\leq\,&
		C(n,p)\big(1+A(t-s)\big)^{p/2}\nu_{x,t;s}(X)\big(-\log(\nu_{x,t;s}(X)/2)\big)^{p/2}. \label{ineq:integral-bound-on-gradient}
	\end{align}
	In particular,
	\begin{equation}\label{ineq:integral-bound-gradient-heat-kernel}
		(t-s)^{p/2}\int_M\Big(\frac{|\nabla_xK(x,t;y,s)|}{K(x,t;y,s)}\Big)^pd\nu_{x,t;s}(y)
		\leq
		C(n,p)\big(1+A(t-s)\big)^{p/2}.
	\end{equation}
	Moreover, for $p=2$, we may take $C(n,p)=\frac{n}{2}$ in (\ref{ineq:integral-bound-gradient-heat-kernel}) and we even have
	\begin{equation*}
		(t-s)\int_M\Big(\frac{\partial_vK(x,t;y,s)}{K(x,t;y,s)}\Big)^2d\nu_{x,t;s}(y)
		\leq
		\frac{1}{2}\big(1+A(t-s)\big)
	\end{equation*}
	for any $v\in T_xM$ with $|v|_{g(t)}=1$.
\end{Proposition}
\begin{proof}
	We write $\nu$ for $\nu_{x,t;s}$ for simplicity. By parabolic rescaling and time-shifting, we may assume $[s,t]=[0,1]$. Fix $v\in T_xM$ with $|v|_{g(1)}=1$ and write
	\begin{equation*}
		q:=\frac{\partial_vK(x,1;\cdot,0)}{K(x,1;\cdot,0)},
	\end{equation*}
	where $\partial_v$ is taken with respect to the first entry.
	
	We first prove (\ref{ineq:integral-bound-on-gradient}). For this, it is enough to establish
	\begin{equation*}
		\int_X |q|^p \,d\nu
		\leq
		C(n,p)(1+A)^{p/2}
		\nu(X)\big(-\log(\nu(X)/2)\big)^{p/2}
	\end{equation*}
	and let $v$ range through an orthonormal basis of $T_xM$. Note $|q|^p=q_+^p+q_-^p$, so by replacing $v$ with $-v$, it suffices to prove
	\begin{equation*}
		\int_X q_+^p \,d\nu
		\leq
		C(n,p)(1+A)^{p/2}
		\nu(X)\big(-\log(\nu(X)/2)\big)^{p/2}.
	\end{equation*}
	Moreover, we may assume $\nu(X)\leq\min\{\frac{1}{4},e^{-p}\}$. Otherwise $\nu(X)\big(-\log(\nu(X)/2)\big)^{p/2}\geq c(p)$ and we can divide $X$ into at most $N(p)$ parts, each of which has measure $\leq\min\{\frac{1}{4},e^{-p}\}$.
	
	Let $u\in C^\infty(M\times(0,1])$ solve the heat equation $\square u=0$ with initial condition $u(\cdot,0)=\mathds{1}_X$. Then
	\begin{equation*}
		u(x,1)
		=
		\int_M  K(x,1;\cdot,0)\mathds{1}_X  \,dV_{g(0)}
		=
		\nu(X),
		\quad
		\partial_v  u(x,1)
		=
		\int_M  \partial_vK(x,1;\cdot,0)\mathds{1}_X  \,dV_{g(0)}
		=
		\int_X  q  \,d\nu.
	\end{equation*}
	Applying \hyperref[thm:improved-gradient-estimate]{Theorem~\ref*{thm:improved-gradient-estimate}} to $(1-2\epsilon)u+\epsilon$ and taking $\epsilon\to 0$, we get
	\begin{equation}\label{ineq:integral-of-q}
		\int_X  q  \,d\nu
		\leq
		\sqrt{1+A}\,\Phi'(\Phi^{-1}(\nu(X))).
	\end{equation}
	Set $F:=\Phi'\circ\Phi^{-1}$. Define $h:(0,1)\to\R$ by
	\begin{equation*}
		h(a):=\sup\{h'\in\R:\nu(\{q\geq h'\})\geq a\}.
	\end{equation*}
	In \cite[Claim 4.6]{bamler2020entropy} , it's proved that for any $b\in\R$, measurable subset $Y\subset M$ with
	\begin{equation*}
		\{q>b\}\subset Y\subset\{q\geq b\}
	\end{equation*}
	and continuous function $f:\R\to\R$, we have
	\begin{equation}\label{eq:compute-integral}
		\int_Y  f(q)  \,d\nu
		=
		\int_0^{\nu(Y)}  f(h(\tilde{a}))  \,d\tilde{a}.
	\end{equation}
	Based on this, \cite{bamler2020entropy} then shows
	\begin{equation}\label{ineq:temp-integral-gradient}
		\int_X  q_+^p  \,d\nu
		\leq
		\int_0^{\nu(X)}  h_+^p(\tilde{a})  \,d\tilde{a},
	\end{equation}
	For any fixed $a\in(0,1)$, choose $b\in\R$ and $\{q>b\}\subset Y\subset\{q\geq b\}$ such that $\nu(Y)=a$. Then by the fact that $h$ is non-increasing, (\ref{eq:compute-integral}) and (\ref{ineq:integral-of-q}) (in which the subset is arbitrary), we get
	\begin{equation*}
		ah(a)
		\leq
		\int_0^a  h(\tilde{a})  \,d\tilde{a}
		=
		\int_Y  q  \,d\nu
		\leq
		\sqrt{1+A}\,F(a).
	\end{equation*}
	Moreover, \cite[(4.15)]{bamler2020entropy}, for $a\leq\frac{1}{4}$ we have $F(a)\leq C(-\log a)^{1/2}a$. Hence
	\begin{equation*}
		h(a)\leq C\sqrt{1+A}\,(-\log a)^{1/2}.
	\end{equation*}
	Then for $a:=\nu(X)\leq\min\{\frac{1}{4},e^{-p}\}$, we have by (\ref{ineq:temp-integral-gradient}) and the above inequality that
	\begin{align*}
		\int_X  q_+^p  \,d\nu
		&\leq
		\int_0^{a}  h_+^p(\tilde{a})  \,d\tilde{a}
		\leq
		C(1+A)^{p/2}  \int_0^{a}  \,(-\log\tilde{a})^{p/2}
		d\tilde{a}\\
		&\leq
		C(1+A)^{p/2}\cdot
		2\int_0^a\Big(
		(-\log\tilde{a})^{p/2}
		-\frac{p}{2}(-\log\tilde{a})^{p/2-1}
		\Big)
		d\tilde{a}\\
		&=
		C(1+A)^{p/2}\cdot 2a(-\log a)^{p/2}.
	\end{align*}
	This completes the proof of (\ref{ineq:integral-bound-on-gradient}). Taking $X$ to be the whole manifold, we obtain (\ref{ineq:integral-bound-gradient-heat-kernel}).
	
	It remains to prove the statement concerning $p=2$. By (\ref{eq:compute-integral}), it suffices to prove
	\begin{equation*}
		\int_0^1  h^2(a)  \,da
		\leq\frac{1}{2}(1+A).
	\end{equation*}
	Define $H:[0,1]\to\R$ by $H(a):=\int_0^a h(\tilde{a})\,d\tilde{a}$. The rest of proof is almost identical to the one in \cite[Proposition 4.2]{bamler2020entropy} , with the only difference being $H(a)\leq\sqrt{1+A}\,F(a)$ now.
\end{proof}

\section{Logarithmic Sobolev Inequality and \texorpdfstring{Poincar\'{e}}{Poincare} Inequality}
As a consequence of the above gradient estimates, we have the following Poincar\'{e} inequality and log-Sobolev inequality (as in \cite{hein2014new}) along any flow with $E(g)\geq-Ag$ for some $A\geq 0$.
\begin{Theorem}[Poincar\'{e} inequality]\label{thm:Poincare-inequality}
	Assume $E(g)\geq-Ag$ for some $A\geq 0$. Fix $(x_0,t_0)$ and let $d\nu(t)=K(x_0,t_0;\cdot,t)dV_{g(t)}$ be the conjugate heat kernel measure based at $(x_0,t_0)$. Then for any $s\leq t_0$ and $u\in C_c^\infty(M)$, it holds
	\begin{equation}\label{ineq:Poincare}
		\int_Mu^2d\nu(s)-\left(\int_Mu\,d\nu(s)\right)^2
		\leq
		2\sigma_A(t_0-s)\int_M|\nabla u|^2d\nu(s).
	\end{equation}
\end{Theorem}
\begin{proof}
	Define $\bar{u}(x,t)=\int_MK(x,t;y,s)u(y,s)\,dV_{g(s)}(y)$ for $(x,t)\in M\times[s,t_0]$. Then
	\begin{equation*}
		\int_Mu^2d\nu(s)-\left(\int_Mu\,d\nu(s)\right)^2
		=
		-\int_s^{t_0}\left(\frac{d}{dt}\int_M\bar{u}(x,t)^2K(x_0,t_0;x,t)\,dV_{g(t)}(x)\right)dt.
	\end{equation*}
	Since $K(x_0,t_0;x,t)$ solves the conjugate heat equation,
	\begin{align*}
		\frac{d}{dt}\int_M\bar{u}(x,t)^2K(x_0,t_0;x,t)\,dV_{g(t)}(x)
		&=
		\int_M\square(\bar{u}(x,t)^2)K(x_0,t_0;x,t)\,dV_{g(t)}(x)\\
		&=\int_M\square(\bar{u}^2)d\nu(t)
		=-2\int_M|\nabla\bar{u}|^2d\nu(t).
	\end{align*}
	Therefore
	\begin{equation}\label{eq:homotopy}
		\int_Mu^2d\nu(s)-\left(\int_Mu\,d\nu(s)\right)^2
		=
		2\int_s^{t_0}\left(\int_M|\nabla\bar{u}|^2d\nu(t)\right)dt.
	\end{equation}
	The integrand may be estimated as follows. Since
	\begin{align*}
		\frac{d}{dt}\int_M|\nabla\bar{u}|^2d\nu(t)
		&=\int_M\square(|\nabla\bar{u}|^2)d\nu(t)\\
		&=\int_M(-2|\nabla^2\bar{u}|^2-E^{ij}\bar{u}_i\bar{u}_j)d\nu(t)\\
		&\leq A\int_M|\nabla\bar{u}|^2d\nu(t),
	\end{align*}
	we have
	\begin{equation*}
		\int_M|\nabla\bar{u}|^2d\nu(t)
		\leq
		e^{A(t-s)}\int_M|\nabla\bar{u}|^2d\nu(s).
	\end{equation*}
	The result follows by substituting this into (\ref{eq:homotopy}).
\end{proof}

A similar modification of the proof in \cite{hein2014new} yields the following.
\begin{Theorem}[Log-Sobolev inequality]
	The setting is the same as in \hyperref[thm:Poincare-inequality]{Theorem~\ref*{thm:Poincare-inequality}}.
	For any $s\leq t_0$ and $u\in C_c^\infty(M)$ with $u>0$, we have
	\begin{equation}\label{ineq:log-Sobolev}
		\int_Mu\log u\,d\nu(s)
		-\left(\int_Mu\,d\nu(s)\right)\log\left(\int_Mu\,d\nu(s)\right)
		\leq
		\sigma_A(t_0-s)
		\int_M\frac{|\nabla u|^2}{u}d\nu(s).
	\end{equation}
\end{Theorem}
\begin{proof}[Sketch of Proof]
	Set $\phi(x)=x\log x$ and $e(t)=\int_M\frac{|\nabla\bar{u}|^2}{\bar{u}}d\nu(t)$. Then $e'(t)\leq Ae(t)$ and
	\begin{equation*}
		\int_M\phi(u)\,d\nu(s)-\phi\left(\int_Mu\,d\nu(s)\right)
		=\int_s^{t_0}e(t)dt.
	\end{equation*}
\end{proof}

\begin{Corollary}[Gaussian concentration]\label{cor:Gaussian-Concentration}
	For any $s<t_0$ and measurable subsets $X,Y\subset M$,
	\begin{equation*}
		\nu_s(X)\nu_s(Y)\leq\exp\Big(-\frac{d_{g(s)}(X,Y)^2}{8\sigma_A(t_0-s)}\Big).
	\end{equation*}
\end{Corollary}
\begin{proof}
	We follow the proof of \cite[Theorem 1.13]{hein2014new}.
	For $F\in C^\infty(M)$, define $\psi:(0,\infty)\to\R$ by $\psi(\lambda)=\frac{1}{\lambda}\log(\int_Me^{\lambda F}d\nu_s)$.
	Set $\mu=\int_Me^{\lambda F}d\nu_s$.
	By (\ref{ineq:log-Sobolev}),
	\begin{align*}
		\frac{d\psi}{d\lambda}
		&=
		\frac{1}{\lambda^2\mu}\left(
		\int_M\lambda Fe^{\lambda F}d\nu_s-\mu\log\mu
		\right)  \\
		&\leq
		\frac{1}{\lambda^2\mu}
		\sigma_A(t_0-s)
		\int_M  \frac{|\nabla e^{\lambda F}|^2}{e^{\lambda F}}  d\nu_s  \\
		&=
		\sigma_A(t_0-s)
		\frac{1}{\mu}
		\int_M  |\nabla F|^2e^{\lambda F}  d\nu_s
	\end{align*}
	When $|\nabla F|\leq 1$, it follows $\frac{d\psi}{d\lambda}\leq\sigma_A(t_0-s)$, which in fact holds for all $1$-Lipschitz function $F$ by approximation.
	If moreover we have $\int_MFd\nu_s=0$, then $\lim_{\lambda\to 0}\psi(\lambda)=0$ and we may integrate the bound on $\frac{d\psi}{d\lambda}$ to obtain
	\begin{equation*}
		\int_Me^{\lambda F}d\nu_s
		\leq
		\exp(\sigma_A(t_0-s)\lambda^2).
	\end{equation*}
	Applying the preceding inequality to
	\begin{equation*}
		F=d_{g(s)}(\cdot,Y)-\int_Md_{g(s)}(x,Y)d\nu_s(x)
	\end{equation*}
	and its negative $-F$, we get
	\begin{equation*}
		\exp\big(\lambda d_{g(s)}(X,Y)\big)
		\nu_s(X)\nu_s(Y)
		\leq
		\int_X\int_Ye^{\lambda(F(x_1)-F(x_2))}d\nu_s(x_1)d\nu_s(x_2)
		\leq
		\exp\big(2\sigma_A(t_0-s)\lambda^2\big).
	\end{equation*}
	The result follows by optimizing in $\lambda$.
\end{proof}
\begin{Corollary}[Averaged heat kernel upper bound]
	For any $s\in[t_0-\frac{1}{2A},t_0)$, set $r^2=t_0-s$. Then
	\begin{equation*}
		\frac{1}{|B_r(x_1,s)|}\int_{B_r(x_1,s)}K(x_0,t_0;y,s)\,dV_{g(s)}(y)
		\leq
		\frac{C(t_0-s)^{-n/2}}{\nu_{x_0,t_0;s}(B_{r}(x_0,s))}\exp\Big(-\frac{d_{g(s)}(x_0,x_1)^2}{C(t_0-s)}\Big).
	\end{equation*}
\end{Corollary}

\begin{Corollary}[Hyper-contractivity]
	Suppose $0<\tau_1<\tau_2$ and $1<q\leq p<\infty$ satisfy
	\begin{equation*}
		\frac{\bar\sigma_A(\tau_2)}{\bar\sigma_A(\tau_1)}\geq\frac{p-1}{q-1}.
	\end{equation*}
	Then for $u\in C^2(M\times[t_0-\tau_2,t_0-\tau_1])$ satisfying $\square u=0$ (or $\square u\leq 0$ and $u\geq 0$), we have
	\begin{equation*}
		\left(\int_Mu^{p}d\nu_{t_0-\tau_1}\right)^{1/p}\leq
		\left(\int_Mu^{q}d\nu_{t_0-\tau_2}\right)^{1/q}.
	\end{equation*}
	In particular, $\normm{u}_{L^p(d\nu)}({t_0-\tau})$ is non-decreasing in $\tau$ for any $1<p<\infty$.
\end{Corollary}
\begin{proof}
	Set $c=(p-1)\bar\sigma_A(\tau_1)$ and
	\begin{equation*}
		p_*(t)=1+\frac{c}{\bar\sigma_A(t_0-t)}.
	\end{equation*}
	Since $\sigma_A(\tau)=e^{A\tau}\bar\sigma_A(\tau)$, a direct computation gives $\dot p_*(t)\sigma_A(t_0-t)=p_*(t)-1$.
	The proof of \cite[Theorem 12.1]{bamler2020entropy}, using the log-Sobolev constant
	from (\ref{ineq:log-Sobolev}), therefore gives
	\begin{equation*}
		\frac{d}{dt}\left(\int_Mu^{p_*(t)}d\nu_t\right)^{1/p_*(t)}\leq0.
	\end{equation*}
	Here $p_*(t_0-\tau_1)=p$, while the hypothesis gives
	$p_*(t_0-\tau_2)\leq q$.  Hence
	\begin{align*}
		\left(\int_Mu^{p}d\nu_{t_0-\tau_1}\right)^{1/p}
		&\leq
		\left(\int_Mu^{p_*(t_0-\tau_2)}d\nu_{t_0-\tau_2}\right)^{1/p_*(t_0-\tau_2)}
		\leq
		\left(\int_Mu^{q}d\nu_{t_0-\tau_2}\right)^{1/q}.
	\end{align*}
\end{proof}

\begin{Theorem}[$L^p$-Poincar\'{e} inequality]
	The setting is the same as \hyperref[thm:Poincare-inequality]{Theorem~\ref*{thm:Poincare-inequality}}.
	Then for any $p\geq 1$, $s\leq t_0$ and $h\in C^1(M)$ with $\int_Mh\,d\nu(s)=0$, it holds
	\begin{equation}\label{ineq:Lp-Poincare-inequality}
		\int_M  |h|^p  d\nu(s)
		\leq
		C(p)e^{pA(t_0-s)}(t_0-s)^{p/2}\int_M|\nabla h|^pd\nu(s).
	\end{equation}
	Equivalently, for any $h\in C^1(M)$, it holds
	\begin{equation}\label{ineq:Lp-Poincare-inequality-second-version}
		\int_M  |h|^p  d\nu(s)
		\leq
		C(p)e^{pA(t_0-s)}(t_0-s)^{p/2}\int_M|\nabla h|^pd\nu(s)
		+C(p)\left(\int_M   h   \,d\nu(s)\right)^p.
	\end{equation}
\end{Theorem}
\begin{proof}
	We first prove the inequality for $p=1$. Let $u\in C^\infty(M)$ with $0<u<1$. Let $\bar{h}$ (resp. $\bar{u}$) solves the heat equation $\square\bar{h}=0$ (resp. $\square\bar{u}=0$), starting from $h$ (resp. $u$) at time $s$. Then $\bar{h}(x_0,t_0)=\int_Mh\,d\nu(s)=0$. The bound $0<u<1$ is preserved by the heat equation and therefore \hyperref[thm:improved-gradient-estimate]{Theorem~\ref*{thm:improved-gradient-estimate}} implies
	\begin{equation*}
		|\nabla\bar{u}|(\cdot,t)
		\leq
		\sqrt{1+A(t-s)}\big(4\pi(t-s)\big)^{-1/2}.
	\end{equation*}
	By \hyperref[prop:gradient-estimate-I]{Proposition~\ref*{prop:gradient-estimate-I}},
	\begin{equation*}
		\frac{d}{dt}\int_M   |\nabla \bar{h}|   \,d\nu(t)
		=
		\int_M   \square|\nabla \bar{h}|   \,d\nu(t)
		\leq
		\frac{A}{2}   \int_M   |\nabla \bar{h}|   \,d\nu(t).
	\end{equation*}
	Hence
	\begin{equation*}
		\int_M   |\nabla \bar{h}|   \,d\nu(t)
		\leq
		e^{\frac{A}{2}(t-s)}\int_M   |\nabla \bar{h}|   \,d\nu(s).
	\end{equation*}
	Therefore
	\begin{align*}
		\pm\frac{d}{dt}\int_M   \bar{h}\bar{u}   \,d\nu(t)
		&=
		\pm\int_M   \square(\bar{h}\bar{u})   \,d\nu(t)
		=
		\mp2\int_M   \nabla\bar{h}\cdot\nabla\bar{u}   \,d\nu(t)
		\geq
		-2\sup_M|\nabla\bar{u}|(\cdot,t)
		\int_M   |\nabla \bar{h}|   \,d\nu(t)\\
		&\geq
		-2\sqrt{1+A(t-s)}\big(4\pi(t-s)\big)^{-1/2}
		e^{\frac{A}{2}(t-s)}\int_M   |\nabla \bar{h}|   \,d\nu(s)\\
		&\geq
		-2\big(4\pi(t-s)\big)^{-1/2}
		e^{A(t_0-s)}\int_M   |\nabla \bar{h}|   \,d\nu(s).
	\end{align*}
	Integrate this over $t\in[s,t_0]$ yields
	\begin{equation*}
		\pm\int_M   {h}{u}   \,d\nu(s)
		\leq
		\frac{2}{\sqrt{\pi}}(t_0-s)^{1/2}e^{A(t_0-s)}
		\int_M   |\nabla{h}|   \,d\nu(s).
	\end{equation*}
	Thus for any $u_1,u_2\in C^\infty(M)$ with $0<u_1,u_2<1$, it holds
	\begin{equation*}
		\int_Mh(u_1-u_2)d\nu(s)
		\leq
		\frac{4}{\sqrt{\pi}}(t_0-s)^{1/2}e^{A(t_0-s)}
		\int_M   |\nabla{h}|   \,d\nu(s).
	\end{equation*}
	Then (\ref{ineq:Lp-Poincare-inequality}) follows by taking $(u_1-u_2)\to(\mathds{1}_{\{h>0\}}-\mathds{1}_{\{h<0\}})=\mathrm{sign}(h)$.
	
	Next we prove the inequality for $p>1$. Choose $a\in\R$ such that $\nu_s(\{h\geq a\})$, $\nu_s(\{h\leq a\})\geq 1/2$. Then
	\begin{align*}
		\frac{|a|}{2}
		\leq   \int_M   |h|   \,d\nu(s)
		&\leq   C(1)e^{A(t_0-s)}(t_0-s)^{1/2}   \int_M|\nabla h|\,d\nu(s)\\
		&\leq   C(1)e^{A(t_0-s)}(t_0-s)^{1/2}
		\left(   \int_M   |\nabla h|^p   d\nu(s)   \right)^{1/p}.
	\end{align*}
	Set $b=\int_M(h-a)_+^pd\nu(s)$. Since $\nu(\{(h-a)_+=0\})=\nu(\{h\leq a\})\geq1/2$, we have
	\begin{align*}
		\frac{b}{2}
		&\leq
		\int_M   |(h-a)_+^p-b|   \,d\nu(s)\\
		&\leq
		C(1)e^{A(t_0-s)}(t_0-s)^{1/2}
		\int_M   p|\nabla(h-a)_+|(h-a)_+^{p-1}   d\nu(s)\\
		&\leq
		C(1)pe^{A(t_0-s)}(t_0-s)^{1/2}
		\left(\int_M|\nabla(h-a)_+|^pd\nu(s)\right)^{\frac{1}{p}}
		\left(\int_M(h-a)_+^pd\nu(s)\right)^{\frac{p-1}{p}}.
	\end{align*}
	Therefore
	\begin{align*}
		\int_M(h-a)_+^pd\nu(s)
		&\leq
		\left(2C(1)pe^{A(t_0-s)}(t_0-s)^{1/2}\right)^p
		\int_M|\nabla(h-a)_+|^pd\nu(s)\\
		&=
		C(p)e^{pA(t_0-s)}(t_0-s)^{p/2}
		\int_{\{h\geq a\}}|\nabla h|^pd\nu(s).
	\end{align*}
	Similarly,
	\begin{equation*}
		\int_M(h-a)_-^pd\nu(s)
		\leq
		C(p)e^{pA(t_0-s)}(t_0-s)^{p/2}
		\int_{\{h\leq a\}}|\nabla h|^pd\nu(s).
	\end{equation*}
	Summing the above two inequalities, we get
	\begin{equation*}
		\int_M   |h-a|^p   d\nu(s)
		\leq
		C(p)e^{pA(t_0-s)}(t_0-s)^{p/2}
		\int_{M}|\nabla h|^pd\nu(s).
	\end{equation*}
	When combined with the estimate on $|a|$, this yields (\ref{ineq:Lp-Poincare-inequality}). Finally, (\ref{ineq:Lp-Poincare-inequality-second-version}) follows by applying (\ref{ineq:Lp-Poincare-inequality}) to $h-a$ with $a:=\int_M h\,d\nu(s)$.
\end{proof}

\section{Pointed Nash Entropy}
We first recall the definition of (pointed) Nash entropy.
Let $(X^n,g)$ be a Riemannian manifold.
For $\tau>0$ and a probability measure $d\nu=(4\pi\tau)^{-n/2}e^{-f}dV_g$, the Nash entropy is defined by
\begin{equation*}
	\mathcal{N}(g,f,\tau)=\int_X f \,d\nu-\frac{n}{2}.
\end{equation*}
Return to our flow case.
Let $d\nu_{x_0,t_0}=K(x_0,t_0;\cdot,\cdot)dV_g$ be the conjugate heat kernel measure based at $(x_0,t_0)\in M\times I$. Set $\tau(t)=t_0-t$ and define $f$ by $(4\pi\tau)^{-n/2}e^{-f}=K(x_0,t_0;\cdot,\cdot)$. The pointed Nash entropy at $(x_0,t_0)$ is defined as
\begin{equation*}
	\mathcal{N}_{x_0,t_0}(\tau)
	:=\mathcal{N}(g(t_0-\tau),f(t_0-\tau),\tau).
\end{equation*}
For $s<t_0$, $s\in I$, we also write
\begin{equation*}
	\mathcal{N}^*_s(x_0,t_0)
	:=\mathcal{N}_{x_0,t_0}(t_0-s).
\end{equation*}
One can show that $\mathcal{N}_{x_0,t_0}(0)=0$.

Next we turn to some basic properties of (pointed) Nash entropy.
\begin{Proposition}\label{prop:time-derivative}
	Let $\tau=\tau(t)>0$ evolve by $\frac{d\tau}{dt}=-1$ and let $(4\pi\tau)^{-n/2}e^{-f}$ solve the conjugate heat equation so that $d\nu=(4\pi\tau)^{-n/2}e^{-f}dV_g$ is a probability measure. Then
	\begin{align}
		&\frac{d}{d\tau}(\tau\mathcal{N}(g,f,\tau))
		=
		\mathcal{W}(g,f,\tau)-\frac{\tau}{2}\int_M\mathrm{tr}\,E\,   d\nu,\label{eq:derivative-of-pointed-Nash-entropy}\\
		&\frac{d^2}{d\tau^2}(\tau\mathcal{N}(g,f,\tau))
		=
		\frac{d}{d\tau}\mathcal{W}(g,f,\tau)-
		\frac{d}{d\tau}\left(\frac{\tau}{2}\int_M\mathrm{tr}\,E\,   d\nu\right).\label{eq:second-derivative-of-pointed-Nash-entropy}
	\end{align}
\end{Proposition}
\begin{proof}
	The second identity follows directly from the first. By (\ref{eq:conjugate-heat-operator}) and (\ref{eq:conjugate-heat-equation}),
	\begin{align*}
		\frac{d}{d\tau}(\tau\mathcal{N}(g,f,\tau))
		&=\mathcal{N}(g,f,\tau)
		+\tau\frac{d}{d\tau}\mathcal{N}(g,f,\tau)\\
		&=\int_M   f   \,d\nu-\frac{n}{2}
		-\tau\int_M\square f\,d\nu\\
		&=\int_M   f   \,d\nu-\frac{n}{2}
		-\tau\int_M\left(-2\Delta f+\abs{\nabla f}^2-R+\frac{1}{2}\mathrm{tr}\,E+\frac{n}{2\tau}\right)d\nu\\
		&=\int_M  \big(\tau(|\nabla f|^2+R)+f-n\big)  d\nu
		-\frac{\tau}{2}\int_M\mathrm{tr}\,E\,   d\nu,
	\end{align*}
	where the last equality follows from an integration by parts.
\end{proof}

\begin{Corollary}\label{cor:difference-of-first-derivative}
	Let $g,f,\tau$ evolve as in \hyperref[prop:time-derivative]{Proposition~\ref*{prop:time-derivative}}. For any $0<\tau_1<\tau_2$, it holds
	\begin{equation}\label{ineq:difference-of-first-derivative}
		\frac{d}{d\tau}(\tau\mathcal{N}(g,f,\tau))\Big|^{\tau_2}_{\tau_1}
		\leq
		\frac{1}{16}|E|^2_{\mathrm{max}}(\tau_2^2-\tau_1^2)
		-\frac{1}{2}(\mathrm{tr}\,E)_{\mathrm{min}}\tau_2
		+\frac{1}{2}(\mathrm{tr}\,E)_{\mathrm{max}}\tau_1.
	\end{equation}
\end{Corollary}
\begin{proof}
	Integrating (\ref{eq:second-derivative-of-pointed-Nash-entropy}) and using (\ref{eq:derivative-of-W}), we get
	\begin{align*}
		\frac{d}{d\tau}(\tau\mathcal{N}(g,f,\tau))\Big|^{\tau_2}_{\tau_1}
		&=
		\int_{\tau_1}^{\tau_2}\Big(\frac{d}{d\tau}\mathcal{W}(g,f,\tau)\Big)d\tau
		-
		\left(\frac{\tau}{2}\int_M\mathrm{tr}\,E\,d\nu\right)\Big|^{\tau_2}_{\tau_1}\\
		&=
		-\int_{\tau_1}^{\tau_2}\int_M   \Big(   2\tau\Big|\Ric+\nabla^2f-\frac{g}{2\tau}-\frac{E}{4}\Big|^2   -\frac{\tau}{8}|E|^2   \Big)   d\nu d\tau
		-
		\left(\frac{\tau}{2}\int_M\mathrm{tr}\,E\,d\nu\right)\Big|^{\tau_2}_{\tau_1}\\
		&\leq
		\int_{\tau_1}^{\tau_2}\int_M\frac{\tau}{8}|E|^2d\nu d\tau-\left(\frac{\tau}{2}\int_M\mathrm{tr}\,E\,d\nu\right)\Big|^{\tau_2}_{\tau_1}\\
		&\leq
		\frac{1}{16}|E|^2_{\mathrm{max}}(\tau_2^2-\tau_1^2)
		-\frac{\tau_2}{2}(\mathrm{tr}\,E)_{\mathrm{min}}
		+\frac{\tau_1}{2}(\mathrm{tr}\,E)_{\mathrm{max}}.
	\end{align*}
\end{proof}

\begin{Corollary}\label{cor:N-and-W}
	Let $(x_0,t_0)\in M\times I$, $\tau(t)=t_0-t$, $d\nu_{x_0,t_0}=(4\pi\tau)^{-n/2}e^{-f}dV_g$. Then
	\begin{equation*}
		\mathcal{W}(g,f,\tau)
		\leq\mathcal{N}_{x_0,t_0}(\tau)
		+C(|E|_{\mathrm{max}}\tau\mid n).
	\end{equation*}
\end{Corollary}
\begin{proof}
	Fix $\tau_2>0$. By (\ref{ineq:difference-of-first-derivative}), for any $\tau\in(0,\tau_2)$ we have
	\begin{equation*}
		\frac{d}{d\tau}(\tau\mathcal{N}(g,f,\tau))\Big|_{\tau_2}
		-\frac{d}{d\tau}(\tau\mathcal{N}(g,f,\tau))
		\leq
		\frac{1}{16}|E|^2_{\mathrm{max}}(\tau_2^2-\tau^2)
		-\frac{1}{2}(\mathrm{tr}\,E)_{\mathrm{min}}\tau_2
		+\frac{1}{2}(\mathrm{tr}\,E)_{\mathrm{max}}\tau.
	\end{equation*}
	Integrating this over $\tau\in(0,\tau_2)$, we get
	\begin{equation*}
		\tau_2\frac{d}{d\tau}(\tau\mathcal{N}(g,f,\tau))\Big|_{\tau_2}
		-(\tau\mathcal{N}(g,f,\tau))\big|_{\tau_2}
		\leq
		\frac{1}{24}|E|^2_{\mathrm{max}}\tau_2^3
		-\frac{1}{2}(\mathrm{tr}\,E)_{\mathrm{min}}\tau_2^2
		+\frac{1}{4}(\mathrm{tr}\,E)_{\mathrm{max}}\tau_2^2,
	\end{equation*}
	which is equivalent to
	\begin{equation}\label{ineq:temp-time-derivative}
		\frac{d}{d\tau}(\tau\mathcal{N}(g,f,\tau))\Big|_{\tau_2}
		\leq
		\mathcal{N}(g,f,\tau)\big|_{\tau_2}
		+\frac{1}{24}|E|^2_{\mathrm{max}}\tau_2^2
		-\frac{1}{2}(\mathrm{tr}\,E)_{\mathrm{min}}\tau_2
		+\frac{1}{4}(\mathrm{tr}\,E)_{\mathrm{max}}\tau_2.
	\end{equation}
	Our result follows by substituting  (\ref{eq:derivative-of-pointed-Nash-entropy}) into the left hand side.
\end{proof}

\begin{Proposition}\label{prop:time-derivative-of-pointed-Nash-entropy}
	Let $(x_0,t_0)\in M\times I$, $\tau(t)=t_0-t$, $d\nu_{x_0,t_0}=(4\pi\tau)^{-n/2}e^{-f}dV_g$. Then
	\begin{equation*}
		-\frac{n}{2}
		+\Big(R-\frac{\mathrm{tr}\,E}{2}\Big)_{\mathrm{min}}\tau
		\leq\tau\frac{d}{d\tau}\mathcal{N}_{x_0,t_0}(\tau)
		\leq C(|E|_{\mathrm{max}}\tau\mid n).
	\end{equation*}
	In particular,
	\begin{equation}\label{cor:upper-bound-on-N}
		\mathcal{N}_{x_0,t_0}(\tau)
		\leq C(|E|_{\mathrm{max}}\tau\mid n).
	\end{equation}
\end{Proposition}
\begin{proof}
	For simplicity, we write $\mathcal{N}(\tau)$ for $\mathcal{N}_{x_0,t_0}(\tau)$. The first inequality follows from (\ref{eq:derivative-of-pointed-Nash-entropy}) as
	\begin{align*}
		\mathcal{N}(\tau)+\tau\frac{d\mathcal{N}}{d\tau}(\tau)
		&=
		\int_M\tau(|\nabla f|^2+R)\,d\nu+\mathcal{N}(\tau)-\frac{n}{2}-\frac{\tau}{2}\int_M\mathrm{tr}\,E\,d\nu\\
		&\geq
		\mathcal{N}(\tau)+\tau\Big(R-\frac{\mathrm{tr}\,E}{2}\Big)_{\mathrm{min}}-\frac{n}{2}.
	\end{align*}
	The second inequality is precisely (\ref{ineq:temp-time-derivative}).
\end{proof}

\begin{Proposition}\label{prop:energy-estimate}
	Let $d\nu=(4\pi\tau)^{-n/2}e^{-f}dV_g$ be the conjugate heat kernel measure based at $(x_0,t_0)\in M\times I$, where $\tau(t)=t_0-t$. Then
	\begin{equation*}%\label{ineq:energy-estimate}
		\int_M\tau(|\nabla f|^2+R)\,d\nu
		\leq\frac{n}{2}+C(|E|_{\mathrm{max}}\tau\mid n),
	\end{equation*}
	\begin{equation*}
		\int_M   |\nabla f|^2   d\nu
		\leq\frac{n}{2\tau}-R_{\mathrm{min}}
		+\frac{C(|E|_{\mathrm{max}}\tau\mid n)}{\tau},
	\end{equation*}
	\begin{equation*}
		\int_M  \Big(f-\mathcal{N}_{x_0,t_0}(\tau)-\frac{n}{2}\Big)^2
		d\nu_{t_0-\tau}
		\leq
		2\sigma_A(\tau)
		\left(  \frac{n}{2\tau}  -R_{\mathrm{min}}  +\frac{C(|E|_{\mathrm{max}}\tau\mid n)}{\tau}  \right).
	\end{equation*}
\end{Proposition}
\begin{proof}
	For the first inequality, we use \hyperref[cor:N-and-W]{Corollary~\ref*{cor:N-and-W}} to conclude
	\begin{equation*}
		\int_M\tau(|\nabla f|^2+R)\,d\nu
		=\frac{n}{2}+\mathcal{W}(g,f,\tau)-\mathcal{N}(g,f,\tau)
		\leq\frac{n}{2}+C(|E|_{\mathrm{max}}\tau\mid n).
	\end{equation*}
	The second inequality follows from the first one directly. By \hyperref[thm:Poincare-inequality]{Theorem~\ref*{thm:Poincare-inequality}},
	\begin{equation*}
		\int_M  \Big(f-\mathcal{N}_{x_0,t_0}(\tau)-\frac{n}{2}\Big)^2
		d\nu_{t_0-\tau}
		\leq
		2\sigma_A(\tau)
		\int_M   |\nabla f|^2   d\nu.
	\end{equation*}
	Then our last inequality follows from the second one.
\end{proof}

\begin{Theorem}[Gradient estimate]\label{thm:gradient-estimate-Nash-entropy}
	There exists a constant $C$ depending only on $n$, $R_{\mathrm{min}}(t-s)$ and $|E|_{\mathrm{max}}(t-s)$ such that
	\begin{equation*}
		|\nabla\mathcal{N}^*_s(\cdot,t)|
		\leq
		\frac{C}{(t-s)^{1/2}}.
	\end{equation*}
\end{Theorem}
\begin{proof}
The proof of \cite[Theorem 5.9]{bamler2020entropy} gives
\begin{equation*}
	|\partial_v\mathcal{N}^*_s(x,t)|_{g(t)}
	\leq
	\left(\int_{M}\Big(\frac{\partial_{v}K(x,t;y,s)}{K(x,t;y,s)}\Big)^{2}d\nu_{s}(y)\right)^{1/2}
	\left(\int_{M}\left(f-\mathcal{N}_{s}^{*}(x,t)-\frac{n}{2}\right)^{2}d\nu_{s}\right)^{1/2}.
\end{equation*}
Set $\tau:=t-s$.
Then it follows from \hyperref[prop:energy-estimate]{Proposition~\ref*{prop:energy-estimate}} and (\ref{ineq:integral-bound-gradient-heat-kernel}) that
\begin{equation}\label{ineq:directional-derivative-estimate-Nash-entropy}
	|\partial_v\mathcal{N}^*_s(x,t)|^2
	\leq
	C(n,|E|_{\mathrm{max}}\tau,R_{\mathrm{min}}\tau)\frac{1+A\tau}{\tau}  \frac{\sigma_A(\tau)}{\tau}.
\end{equation}
\end{proof}

\begin{Theorem}\label{ineq:square-of-N}
	Suppose $E(g)\geq-Ag$ for some $A\geq 0$. Then
	\begin{equation*}
		-\frac{n}{2(t-s)}
		\leq
		\square\mathcal{N}^*_s(\cdot,t)
		\leq
		\frac{nA}{2}.
	\end{equation*}
\end{Theorem}
\begin{proof}
	By the definition of pointed Nash entropy and (\ref{eq:ulogu}),
	\begin{align*}
		\square\mathcal{N}^*_s(x,t)
		&=-\int_M\square_{x,t}\big(K(x,t;y,s)\log K(x,t;y,s)\big)dV_{g(s)}(y)-\frac{n}{2(t-s)}\\
		&=\int_M\Big(\frac{|\nabla_xK(x,t;y,s)|}{K(x,t;y,s)}\Big)^2d\nu_{x,t;s}(y)-\frac{n}{2(t-s)}.
	\end{align*}
	The first inequality is obvious and in fact does not need $E\geq-Ag$. The second inequality follows from (\ref{ineq:integral-bound-gradient-heat-kernel}).
\end{proof}

\begin{Corollary}
	Suppose $E(g)\geq-Ag$ for some $A\geq 0$. Then for any $s<t^*\leq t$ and $x\in M$,
	\begin{equation}\label{ineq:averaged-Nash-entropy}
		\mathcal{N}_s^*(x,t)-\frac{nA}{2}(t-t^*)
		\leq
		\int_M\mathcal{N}_s^*(\cdot,t^*)d\nu_{x,t;t^*}
		\leq
		\mathcal{N}_s^*(x,t)+\frac{n}{2}\log\frac{t-s}{t^*-s}.
	\end{equation}
\end{Corollary}
\begin{proof}
	As a function of $(x,t)$, $\int_M(\mathcal{N}_s^*(\cdot,t^*)-\frac{nA}{2}t^*)d\nu_{x,t;t^*}$ solves the heat equation staring from $\mathcal{N}_s^*(\cdot,t^*)-\frac{nA}{2}t^*$ at time $t^*$, i.e.,
	\begin{equation*}
		(\partial_t-\Delta_x)\int_M\Big(\mathcal{N}_s^*(\cdot,t^*)-\frac{nA}{2}t^*\Big)d\nu_{x,t;t^*}=0,
	\end{equation*}
	Combined with
	\begin{equation*}
		(\partial_t-\Delta_x)\Big(\mathcal{N}_s^*(x,t)-\frac{nA}{2}t\Big)\leq 0,
	\end{equation*}
	we get for all $(x,t)$ with $t>t^*$ that
	\begin{equation*}
		\mathcal{N}_s^*(x,t)-\frac{nA}{2}t
		\leq
		\int_M\Big(\mathcal{N}_s^*(\cdot,t^*)-\frac{nA}{2}t^*\Big)d\nu_{x,t;t^*}.
	\end{equation*}
	The other inequality is similar, by using $\square_{x,t}(\mathcal{N}_s^*(x,t)+\frac{n}{2}\log(t-s))\geq 0$.
\end{proof}

\begin{Corollary}\label{cor:oscillation-of-pointed-Nash-entropy}
	For any $s<t^*\leq t_1,t_2$ and $x_1,x_2\in M$,
	\begin{equation*}
		\mathcal{N}_s^*(x_1,t_1)-\mathcal{N}_s^*(x_2,t_2)
		\leq
		\frac{C}{(t^*-s)^{1/2}}d_{W_1}^{g(t^*)}(\nu_{x_1,t_1;t^*},\nu_{x_2,t_2;t^*})
		+\frac{nA}{2}(t_1-t^*)+\frac{n}{2}\log\frac{t_2-s}{t^*-s},
	\end{equation*}
	where $C(n,R_{\mathrm{min}}(t^*-s),|E|_{\mathrm{max}}(t^*-s))$ is the constant obtained in \hyperref[thm:gradient-estimate-Nash-entropy]{Theorem~\ref*{thm:gradient-estimate-Nash-entropy}}.
\end{Corollary}
\begin{proof}
	Applying (\ref{ineq:averaged-Nash-entropy}) to $(x_1,t_1)$ and $(x_2,t_2)$, we get
	\begin{align*}
		&\mathcal{N}_s^*(x_1,t_1)-\mathcal{N}_s^*(x_2,t_2)\\
		\leq&
		\left\{\int_M\mathcal{N}_s^*(\cdot,t^*)d\nu_{x_1,t_1;t^*}+\frac{nA}{2}(t_1-t^*)\right\}
		-
		\left\{\int_M\mathcal{N}_s^*(\cdot,t^*)d\nu_{x_2,t_2;t^*}-\frac{n}{2}\log\frac{t_2-s}{t^*-s}\right\}\\
		=&\left\{\int_M\mathcal{N}_s^*(\cdot,t^*)d\nu_{x_1,t_1;t^*}-\int_M\mathcal{N}_s^*(\cdot,t^*)d\nu_{x_2,t_2;t^*}\right\}+\frac{nA}{2}(t_1-t^*)+\frac{n}{2}\log\frac{t_2-s}{t^*-s}\\
		\leq\,&\Big(\sup_M|\nabla\mathcal{N}_s^*(\cdot,t^*)|\Big)
		d_{W_1}^{g(t^*)}(\nu_{x_1,t_1;t^*},\nu_{x_2,t_2;t^*})
		+\frac{nA}{2}(t_1-t^*)+\frac{n}{2}\log\frac{t_2-s}{t^*-s}.
	\end{align*}
	Here we use (\ref{eq:definition-of-W1}) in the last inequality. The result then follows from \hyperref[thm:gradient-estimate-Nash-entropy]{Theorem~\ref*{thm:gradient-estimate-Nash-entropy}}.
\end{proof}

As in \cite{bamler2020entropy}, we have the following no local collapsing estimate.
\begin{Theorem}[No local collapsing]\label{thm:no-local-collapsing}
	If $R\leq r^{-2}$ on $B(x,t,r)$ and $[t-r^2,t]\subset I$, then
	\begin{equation}\label{ineq:no-local-collapsing}
		|B(x,t,r)|_{g(t)}\geq c(n,R_\mathrm{min}r^2,|E|_{\mathrm{max}}r^2)\exp(\mathcal{N}_{x,t}(r^2))r^n.
	\end{equation}
\end{Theorem}
\begin{proof}
	Suppose the function $\rho\mapsto|B(x,t,\rho)|_{g(t)}/\rho^n$ over $[0,r]$ takes its minimum at $r_0$. If $r_0=0$, then $|B(x,t,r)|_{g(t)}\geq\omega_nr^n$ and (\ref{ineq:no-local-collapsing}) follows from (\ref{cor:upper-bound-on-N}). Here $\omega_n$ denotes the volume of the unit ball in $\R^n$. So we may assume $r_0>0$. By \hyperref[prop:time-derivative-of-pointed-Nash-entropy]{Proposition~\ref*{prop:time-derivative-of-pointed-Nash-entropy}},
	\begin{equation*}
		\mathcal{N}_{x,t}(r_0^2)\geq\mathcal{N}_{x,t}(r^2)-C(|E|_{\mathrm{max}}r^2).
	\end{equation*}
	Therefore it suffices to prove
	\begin{equation}\label{ineq:temp-reduced-version}
		|B(x,t,r_0)|_{g(t)}/r_0^n
		\geq
		c(n,R_\mathrm{min}r^2,|E|_{\mathrm{max}}r^2)
		\exp(\mathcal{N}_{x,t}(r_0^2)).
	\end{equation}
	Note by our choice of $r_0$, we have $|B(x,t,r_0)|_{g(t)}\leq 2^n|B(x,t,\frac{1}{2}r_0)|_{g(t)}$.
	
	By parabolic rescaling and time-shifting, we may assume $r_0=1$ and $t=0$, which then implies $r\geq 1$. Take a cut-off function $\eta$ over $\R$ with $\eta|_{(-\infty,\frac{1}{2}]}=1$, $\eta|_{[1,\infty)}=0$ and $|\eta'|\leq 10$. Define
	\begin{equation*}
		\varphi:=\frac{1}{\sqrt{L}}\eta(d_0(x,\cdot)),
	\end{equation*}
	where $L$ is a normalization constant such that $\int_M\varphi^2dV_{g(0)}=1$. Note
	\begin{equation}\label{ineq:estimate-on-L}
		2^{-n}
		\leq
		\frac{|B(x,0,\frac{1}{2})|_{g(0)}}{|B(x,0,1)|_{g(0)}}
		\leq
		\frac{L}{|B(x,0,1)|_{g(0)}}
		\leq 1.
	\end{equation}
	
	Set $\tau(t):=1-t$. Let $u=(4\pi\tau)^{-n/2}e^{-f}\in C^\infty(M\times[-1,0))$ solve the conjugate heat equation $\square^*u=0$ with initial condition $u(\cdot,0)=\varphi^2$. Denote $d\nu:=(4\pi\tau)^{-n/2}e^{-f}dV_g$. Then $d\nu$ is a probability measure at each time and
	\begin{equation*}
		u(\cdot,-1)
		=
		\int_M   K(y,0;\cdot,-1)   \,d\nu_0(y).
	\end{equation*}
	By Jensen's inequality,
	\begin{equation*}
		-u(\cdot,-1)\log u(\cdot,-1)
		\geq
		-\int_M   K(y,0;\cdot,-1)\log K(y,0;\cdot,-1)   \,d\nu_0(y).
	\end{equation*}
	Integrating this over $M$ with respect to $dV_{g(-1)}$, using the definition of (pointed) Nash entropy and \hyperref[thm:gradient-estimate-Nash-entropy]{Theorem~\ref*{thm:gradient-estimate-Nash-entropy}}, we get
	\begin{align}
		\mathcal{N}(g(-1),f(-1),2)
		&\geq
		\int_M   \mathcal{N}_{y,0}(1)   \,d\nu_0(y)+\frac{n}{2}\log(4\pi)-\frac{n}{2}\log(8\pi)  \notag  \\
		&\geq
		\mathcal{N}_{x,0}(1)-C(n,R_\mathrm{min}r^2,|E|_{\mathrm{max}}r^2).\label{ineq:lower-bound-by-pointed-Nash-entropy}
	\end{align}
	
	By (\ref{eq:derivative-of-W}), for $t\in[-1,0]$ we have
	\begin{equation*}
		\mathcal{W}(g(0),f(0),\tau(0))
		-
		\mathcal{W}(g(t),f(t),\tau(t))
		\geq
		-\int_t^0\int_M\frac{\tau}{8}|E|^2d\nu d\tilde{t}
		\geq -C(|E|_{\mathrm{max}}r^2).
	\end{equation*}
	Integrating this over $t\in[-1,0]$, we get
	\begin{equation*}
		\mathcal{W}(g(0),f(0),\tau(0))
		\geq
		\int_{-1}^0\mathcal{W}(g(t),f(t),\tau(t))\,dt
		-C(|E|_{\mathrm{max}}r^2).
	\end{equation*}
	By (\ref{eq:derivative-of-pointed-Nash-entropy}),
	\begin{align*}
		\int_{-1}^0\mathcal{W}(g(t),f(t),\tau(t))\,dt
		&=
		\int_{-1}^0
		\left(-\frac{d}{dt}(\tau\mathcal{N}(g,f,\tau))
		+\frac{\tau}{2}\int_M \mathrm{tr}\,E\,d\nu\right)dt\\
		&\geq
		2\mathcal{N}(g(-1),f(-1),2)
		-\mathcal{N}(g(0),f(0),1)
		-C(|E|_{\mathrm{max}}r^2).
	\end{align*}
	Combining the above two inequalities and (\ref{ineq:lower-bound-by-pointed-Nash-entropy}), we obtain
	\begin{align*}
		\mathcal{W}(g(0),f(0),1)
		&\geq
		2\mathcal{N}(g(-1),f(-1),2)
		-\mathcal{N}(g(0),f(0),1)
		-C(|E|_{\mathrm{max}}r^2)\\
		&\geq
		2\mathcal{N}_{x,0}(1)
		-\mathcal{N}(g(0),f(0),1)
		-C(n,R_\mathrm{min}r^2,|E|_{\mathrm{max}}r^2),
	\end{align*}
	which can be written equivalently as
	\begin{equation*}
		\int_M
		\big(
		4|\nabla\varphi|^2+R\varphi^2
		-2\varphi^2\log\varphi^2
		\big)
		dV_g
		\geq
		2\mathcal{N}_{x,0}(1)
		-C(n,R_\mathrm{min}r^2,|E|_{\mathrm{max}}r^2).
	\end{equation*}
	Then (\ref{ineq:temp-reduced-version}) follows by substituting the followings to the left hand side.
	\begin{itemize}
		\item Using $|\nabla\varphi|^2\leq CL^{-1}$ and (\ref{ineq:estimate-on-L}), we get $\int_M  4|\nabla\varphi|^2 \leq  C|B(x,0,1)|/L  \leq  C(n)$.
		
		\item Since $R\leq r^{-2}\leq 1$, we have $\int_M R\varphi^2 dV_g \leq 1$.
		
		\item Since $-\varphi^2\log\varphi^2
		=-\frac{1}{L}\eta^2
		\log(\eta^2)
		+\varphi^2\log L
		\leq e^{-1}L^{-1}\mathds{1}_{B(x,0,1)}+\varphi^2\log L$, we obtain by (\ref{ineq:estimate-on-L}) that
		\begin{equation*}
			\int_M  -2\varphi^2\log\varphi^2  dV_g
			\leq
			2e^{-1}L^{-1}|B(x,0,1)|+2\log L
			\leq C(n)+2\log|B(x,0,1)|.
		\end{equation*}
	\end{itemize}
\end{proof}

In the absence of a scalar curvature upper bound, one can still find a similar volume bound around an $(H_n,A)$-center of a conjugate heat kernel.
\begin{Theorem}\label{thm:no-local-collapsing-around-center}
	Suppose $E(g)\geq-Ag$ for some $A\geq 0$ (one may take $A=|E|_{\mathrm{max}}$). Suppose $[t-r^2,t]\subset I$ and $(z,t-r^2)\in M\times I$ is an $(H_n,A)$-center of some point $(x,t)\in M\times I$. Then
	\begin{equation}
		|B(z,t-r^2,\sqrt{2H_n\sigma_A(r^2)})|_{g(t-r^2)}
		\geq
		c(n,R_\mathrm{min}r^2,|E|_{\mathrm{max}}r^2)\exp(\mathcal{N}_{x,t}(r^2))r^n.
	\end{equation} 
\end{Theorem}
\begin{proof}
	Note $\sigma_A(r^2)=\sigma_{Ar^2}(1)r^2$. Thus the result is invariant under parabolic rescaling, and we may assume $r=1$ and $t=1$. Let $\tau(t)=1-t$, $d\nu=(4\pi\tau)^{-n/2}e^{-f}dV_g$ be the conjugate heat kernel measure based at $(x,1)$ and $B=B(z,0,\sqrt{2H_n\sigma_A(1)})$. By \hyperref[prop:basic-properties-of-center]{Proposition~\ref*{prop:basic-properties-of-center}}(3), $\nu_0(B)\geq\frac{1}{2}$.
	
	By \hyperref[prop:energy-estimate]{Proposition~\ref*{prop:energy-estimate}},
	\begin{equation*}
		\int_M  \Big|f-\mathcal{N}_{x,1}(1)-\frac{n}{2}\Big|
		\,d\nu_{t_0-\tau}
		\leq
		\bigg(
		\int_M  \Big(f-\mathcal{N}_{x,1}(1)-\frac{n}{2}\Big)^2
		d\nu_{t_0-\tau}
		\bigg)^{1/2}
		\leq
		C(n,R_\mathrm{min}r^2,|E|_{\mathrm{max}}r^2).
	\end{equation*}
	Combining this and the fact $\nu_0(B)\geq\frac{1}{2}$, we obtain
	\begin{align*}
		\frac{1}{\nu_0(B)}\int_B  f  \,d\nu_0
		&\geq
		\mathcal{N}_{x,1}(1)+\frac{n}{2}
		-\frac{1}{\nu_0(B)}
		\int_B  \Big|f-\mathcal{N}_{x,1}(1)-\frac{n}{2}\Big|  \,d\nu_0\\
		&\geq
		\mathcal{N}_{x,1}(1)-C(n,R_\mathrm{min}r^2,|E|_{\mathrm{max}}r^2).
	\end{align*}
	Let $u=(4\pi)^{-n/2}e^{-f}/\nu_0(B)$. The above inequality together with $\nu_0(B)\geq\frac{1}{2}$ implies
	\begin{align*}
		\int_B u\log u \,dV_{g(0)}
		&=
		-\frac{1}{\nu_0(B)}\int_B f \,d\nu_0
		-\log(\nu_0(B))-\frac{n}{2}\log(4\pi)\\
		&\leq
		-\mathcal{N}_{x,1}(1)+C(n,R_\mathrm{min}r^2,|E|_{\mathrm{max}}r^2).
	\end{align*}
	By Jensen's inequality and the fact $\int_B  u  \,dV_{g(0)}=1$,
	\begin{equation*}
		\frac{1}{|B|_{0}}\int_B  u\log u  \,dV_{g(0)}
		\geq
		\left(\frac{1}{|B|_{0}}\int_B  u  \,dV_{g(0)}\right)
		\log\left(\frac{1}{|B|_{0}}\int_B  u  \,dV_{g(0)}\right)
		=
		\frac{1}{|B|_{0}}\log\left(\frac{1}{|B|_{0}}\right),
	\end{equation*}
	which implies
	\begin{equation*}
		-\log|B|_0
		\leq
		\int_B  u\log u  \,dV_{g(0)}
		\leq
		-\mathcal{N}_{x,1}(1)+C(n,R_\mathrm{min}r^2,|E|_{\mathrm{max}}r^2).
	\end{equation*}
\end{proof}

\section{Heat Kernel Estimates}
This section corresponds to $\S 7$ of \cite{bamler2020entropy}.

Let's start from the following simple lemma bounding the integral of heat kernel.
\begin{Lemma}
	Let $u$ be a solution to the heat equation $\square u=0$. Then for any $s<t$,
	\begin{equation*}
		\int_M u\,dV_g(t)
		\leq
		\exp\big(-(R-\tfrac{1}{2}\mathrm{tr}\,E)_{\mathrm{min}}(t-s)\big)
		\int_M u\,dV_g(s).
	\end{equation*}
	In particular, for any $s<t$ and $y\in M$,
	\begin{equation}\label{ineq:heat-kernel-integral-bound}
		\int_M   K(\cdot,t;y,s)   \,dV_{g(t)}
		\leq
		\exp\big(-(R-\tfrac{1}{2}\mathrm{tr}\,E)_{\mathrm{min}}(t-s)\big).
	\end{equation}
\end{Lemma}
\begin{proof}
	This follows from
	\begin{equation*}
		\frac{d}{dt}   \int_M   u   \,dV_g
		=
		-\int_M   u(R-\tfrac{1}{2}\mathrm{tr}\,E)   \,dV_g
		\leq
		-(R-\tfrac{1}{2}\mathrm{tr}\,E)_{\mathrm{min}}
		\int_M u\,dV_g.
	\end{equation*}
\end{proof}

Next we prove an $L^\infty$-bound on $K(x,t;y,s)$ in terms of the pointed Nash entropy $\mathcal{N}_{x,t}(t-s)$.
\begin{Theorem}[Heat kernel upper bound]\label{thm:heat-kernel-upper-bound}
	There exists a constant $C$ depending only on $n$, $R_{\mathrm{min}}(t-s)$ and $|E|_{\mathrm{max}}(t-s)$ such that
	\begin{equation}\label{ineq:heat-kernel-upper-bound}
		K(x,t;y,s)
		\leq
		\frac{C}{(t-s)^{n/2}}\exp(-\mathcal{N}_{x,t}(t-s)).
	\end{equation}
\end{Theorem}
\begin{proof}
	After parabolic rescaling and time-shifting, we may assume $s=0$, $t=1$.
	Set $A=|E|_{\mathrm{max}}$.
	Then $E(g)\geq -Ag$. Fix $y\in M$ and set $u=K(\cdot,\cdot;y,0)$. Then we need to prove
	\begin{equation}\label{ineq:rough-upper-bound}
		u(x,t)\leq\frac{Z}{t^{n/2}}\exp(-\mathcal{N}_0^*(x,t)),
		\qquad
		\forall\,(x,t)\in M\times(0,1]
	\end{equation}
	for some $Z$ depending only on $n,R_{\mathrm{min}},|E|_{\mathrm{max}}$. By \hyperref[prop:time-derivative-of-pointed-Nash-entropy]{Proposition~\ref*{prop:time-derivative-of-pointed-Nash-entropy}}, $\mathcal{N}_0^*(x,t)$ is bounded above over $(x,t)\in M\times(0,1]$. Thus (\ref{ineq:rough-upper-bound}) holds for some large $Z$, possibly depending on the underlying flow; see for instance \cite[Theorem 26.25]{chow2010ricci}. To prove our result, it suffices to show that if (\ref{ineq:rough-upper-bound}) holds for some $Z$, then it also holds for $Z$ replaced with $Z/2$ if $Z\geq\underline{Z}(n,R_{\mathrm{min}},|E|_{\mathrm{max}})$. By parabolic rescaling, it is enough to show (\ref{ineq:rough-upper-bound}) holds at time $t=1$ for $Z$ replaced with $Z/2$ if $Z\geq\underline{Z}$.
	
	\emph{Step 1.} We first derive a bound on $|\nabla u|$.
	Set $v:=(t-\frac{1}{2})|\nabla u|^2+(1+A)u^2$. By \hyperref[prop:evolution-of-various-quantities]{Proposition~\ref*{prop:evolution-of-various-quantities}}, we have for $t\in[\frac{1}{2},1]$ that
	\begin{equation*}
		\square v
		=
		|\nabla u|^2
		-(t-\tfrac{1}{2})(2|\nabla^2u|^2+E_{ij}u_iu_j)
		-2(1+A)|\nabla u|^2
		\leq 0.
	\end{equation*}
	Therefore for any $(x,t)\in M\times(\tfrac{1}{2},1]$ we have by (\ref{ineq:rough-upper-bound}) that
	\begin{align}
		(t-\tfrac{1}{2})|\nabla u|^2(x,t)
		\leq
		v(x,t)
		\leq
		\int_M   K(x,t;\cdot,\tfrac{1}{2})v(\cdot,\tfrac{1}{2})
		\,dV_g(\tfrac{1}{2})
		=
		(1+A)\int_M   u^2(\cdot,\tfrac{1}{2})
		\,d\nu_{x,t;\frac{1}{2}}  \notag \\
		\leq
		(1+A)2^nZ^2
		\int_M   \exp(-2\mathcal{N}^*_0(\cdot,\tfrac{1}{2}))
		\,d\nu_{x,t;\frac{1}{2}}.\label{ineq:integral-gradient-bound}
	\end{align}
	Let $(z,\frac{1}{2})$ be an $(H_n,A)$-center of $(x,t)$. By \hyperref[thm:gradient-estimate-Nash-entropy]{Theorem~\ref*{thm:gradient-estimate-Nash-entropy}} and \hyperref[cor:oscillation-of-pointed-Nash-entropy]{Corollary~\ref*{cor:oscillation-of-pointed-Nash-entropy}},
	\begin{equation}\label{ineq:change-base-point}
		-\mathcal{N}^*_0(y',\tfrac{1}{2})
		\leq
		-\mathcal{N}^*_0(z,\tfrac{1}{2})
		+Cd_{\frac{1}{2}}(z,y')
		\leq
		-\mathcal{N}^*_0(x,t)
		+C+Cd_{\frac{1}{2}}(z,y'),
		\qquad\forall\,y'\in M
	\end{equation}
	for some $C(n,R_{\mathrm{min}},|E|_{\mathrm{max}})$. Set $f(r)=\nu_{x,t;\frac{1}{2}}(M\setminus B(z,\frac{1}{2},r))$. Then
	\begin{equation}\label{eq:integration-by-parts}
		\int_M   \exp(Cd_{\frac{1}{2}}(z,\cdot))   \,d\nu_{x,t;\frac{1}{2}}
		=
		-\int_0^\infty   e^{Cr}   f'(r)\,dr
		=
		-e^{Cr}f(r)\big|_0^\infty
		+\int_0^\infty Ce^{Cr}f(r)\,dr.
	\end{equation}
	Note \hyperref[prop:basic-properties-of-center]{Proposition~\ref*{prop:basic-properties-of-center}}(4) guarantees that $-e^{Cr}f(r)\big|_0^\infty=1$ and
	\begin{align*}
		\int_0^\infty   e^{Cr}f(r)\,dr
		&\leq
		\int_0^\infty   e^{Cr}\cdot
		2\exp   \bigg(
		-\frac{\big(r-\sqrt{2H_n\sigma_A(t-\frac{1}{2})}\big)_+^2}{8\sigma_A(t-\frac{1}{2})}   \bigg)dr\\
		&=
		\sqrt{\sigma_A(t-\tfrac{1}{2})}
		\int_0^\infty   e^{C\sqrt{\sigma_A(t-\frac{1}{2})}\,r}\cdot
		2\exp\bigg(
		-\frac{(r-\sqrt{2H_n})_+^2}{8}   \bigg)dr
		\leq C.
	\end{align*}
	Combining (\ref{ineq:change-base-point}), (\ref{eq:integration-by-parts}) and the above estimate, we get
	\begin{align*}
		\int_M   \exp(-2\mathcal{N}^*_0(\cdot,\tfrac{1}{2}))
		\,d\nu_{x,t;\frac{1}{2}}
		&\leq
		C\exp(-2\mathcal{N}^*_0(x,t))
		\int_M   \exp(Cd_{\frac{1}{2}}(z,\cdot))   \,d\nu_{x,t;\frac{1}{2}}\\
		&\leq
		C\exp(-2\mathcal{N}^*_0(x,t)).
	\end{align*}
	Hence (\ref{ineq:integral-gradient-bound}) implies that for all $(x,t)\in M\times [\frac{3}{4},1]$,
	\begin{equation}\label{ineq:gradient-bound-by-N}
		|\nabla u|(x,t)
		\leq CZ\exp(-\mathcal{N}^*_0(x,t)).
	\end{equation}
	
	\emph{Step 2.} We next show that (\ref{ineq:rough-upper-bound}) holds at time $t=1$ for $Z$ replaced with $Z/2$ if $Z\geq\underline{Z}$. Fix $(x,1)\in M\times I$. Let $\rho\in(0,\frac{1}{2}]$ be a small constant whose value will be determined later. Set $t_1:=1-\rho^2\in[\frac{3}{4},1]$ and let $(z_1,t_1)$ be an $(H_n,A)$-center of $(x,1)$. Then \hyperref[cor:oscillation-of-pointed-Nash-entropy]{Corollary~\ref*{cor:oscillation-of-pointed-Nash-entropy}} gives
	\begin{equation*}
		-\mathcal{N}^*_0(z_1,t_1)
		\leq
		-\mathcal{N}^*_0(x,1)
		+C\sqrt{H_n\sigma_A(\rho^2)}
		+\frac{nA}{2}\rho^2
		\leq
		-\mathcal{N}^*_0(x,1)+\log 2,
	\end{equation*}
	provided $\rho\leq\overline{\rho}(n,R_{\mathrm{min}},|E|_{\mathrm{max}})$. Combining this with \hyperref[thm:gradient-estimate-Nash-entropy]{Theorem~\ref*{thm:gradient-estimate-Nash-entropy}}, we get
	\begin{equation}\label{ineq:oscillation-of-N}
		-\mathcal{N}^*_0(\cdot,t_1)
		\leq
		-\mathcal{N}^*_0(x,1)+\log 2
		+Cd_{t_1}(z_1,\cdot).
	\end{equation}
	Set $B:=B(z_1,t_1,\sqrt{100H_n\sigma_A(\rho^2)})$. To estimate $u(x,1)=\int_M u(\cdot,t_1)\,d\nu_{x,1;t_1}$, we split the integral into two parts as
	\begin{equation}\label{eq:split-into-two-parts}
		u(x,1)=\int_B u\,d\nu_{x,1;t_1}+\int_{M\setminus B} u\,d\nu_{x,1;t_1}.
	\end{equation}
	
	The first integral can be estimated as follows. By \hyperref[thm:no-local-collapsing-around-center]{Theorem~\ref*{thm:no-local-collapsing-around-center}} and \hyperref[prop:time-derivative-of-pointed-Nash-entropy]{Proposition~\ref*{prop:time-derivative-of-pointed-Nash-entropy}},
	\begin{equation}\label{ineq:volume-lower-bound}
		|B|_{t_1}
		\geq
		c\exp(\mathcal{N}_{x,1}(\rho^2))\rho^n
		\geq
		c\exp(\mathcal{N}^*_0(x,1))\rho^n.
	\end{equation}
	Our gradient bound (\ref{ineq:gradient-bound-by-N}) together with (\ref{ineq:oscillation-of-N}) implies
	\begin{equation*}
		|\nabla u|(\cdot,t_1)
		\leq CZ\exp(-\mathcal{N}^*_0(x,1))
		\quad\text{over }B.
	\end{equation*}
	Thus for any $x',x''\in B$, (note $\sigma_A(\rho^2)\leq C\rho^2$)
	\begin{equation*}
		u(x',t_1)
		\leq
		u(x'',t_1)+CZ\exp(-\mathcal{N}^*_0(x,1))\rho.
	\end{equation*}
	Integrating this over $B$ with respect to $x''$, we get by (\ref{ineq:heat-kernel-integral-bound}) and (\ref{ineq:volume-lower-bound}) that
	\begin{equation*}
		u(x',t_1)
		\leq
		\frac{1}{|B|_{t_1}}\int_B  u\,dV_{g}(t_1)
		+CZ\exp(-\mathcal{N}^*_0(x,1))\rho
		\leq
		(C\rho^{-n}+CZ\rho)\exp(-\mathcal{N}^*_0(x,1)).
	\end{equation*}
	In particular, the first integral in (\ref{eq:split-into-two-parts}) satisfies
	\begin{equation}\label{ineq:integral-bound-I}
		\int_B u\,d\nu_{x,1;t_1}
		\leq
		(C\rho^{-n}+CZ\rho)\exp(-\mathcal{N}^*_0(x,1)).
	\end{equation}
	
	Next we estimate the second integral in (\ref{eq:split-into-two-parts}). By \hyperref[prop:basic-properties-of-center]{Proposition~\ref*{prop:basic-properties-of-center}}(3) and our choice of $B$, we have $\nu_{x,1;t_1}(M\setminus B)\leq\frac{1}{100}$. By (\ref{ineq:rough-upper-bound}), (\ref{ineq:oscillation-of-N}) and the fact $e^s\leq 1+\rho e^{s/\rho}$, we get
	\begin{align*}
		\int_{M\setminus B} u\,d\nu_{x,1;t_1}
		&\leq
		2Z\int_{M\setminus B}   \exp(-\mathcal{N}^*_0(\cdot,t_1))   \,d\nu_{x,1;t_1}\\
		&\leq
		4Z\exp(-\mathcal{N}^*_0(x,1))
		\int_{M\setminus B}   \exp(Cd_{t_1}(z_1,\cdot))
		\,d\nu_{x,1;t_1}\\
		&\leq
		4Z\exp(-\mathcal{N}^*_0(x,1))
		\int_{M\setminus B}
		\big(1+\rho\exp(C\rho^{-1}d_{t_1}(z_1,\cdot))\big)
		d\nu_{x,1;t_1}\\
		&\leq
		\bigg(\frac{4}{100}+4\rho\int_{M}
		\exp(C\rho^{-1}d_{t_1}(z_1,\cdot))\,d\nu_{x,1;t_1}\bigg)
		Z\exp(-\mathcal{N}^*_0(x,1)).
	\end{align*}
	Repeating the argument in Step 1 (namely (\ref{eq:integration-by-parts}) and what follows it), we have
	\begin{align*}
		\int_{M}
		\exp(C\rho^{-1}d_{t_1}(z_1,\cdot))\,d\nu_{x,1;t_1}
		&=
		1+C\rho^{-1}\int_0^\infty e^{C\rho^{-1}r}
		\nu_{x,1;t_1}(M\setminus B(z_1,t_1,r))\,dr\\
		&\leq
		1+C\rho^{-1}\int_0^\infty e^{C\rho^{-1}r}
		\cdot 2\exp\bigg(
		-\frac{\big(r-\sqrt{2H_n\sigma_A(\rho^2)}\big)_+^2}{8\sigma_A(\rho^2)}
		\bigg)dr\\
		&=
		1+C\frac{\sqrt{\sigma_A(\rho^2)}}{\rho}
		\int_0^\infty e^{C\rho^{-1}\sqrt{\sigma_A(\rho^2)}r}
		\cdot 2\exp\big(
		-\tfrac{1}{8}{\big(r-\sqrt{2H_n}\big)_+^2}
		\big)dr\\
		&\leq C.
	\end{align*}
	Therefore
	\begin{equation}\label{ineq:integral-bound-II}
		\int_{M\setminus B} u\,d\nu_{x,1;t_1}
		\leq
		\Big(\frac{4}{100}+C\rho\Big)
		Z\exp(-\mathcal{N}^*_0(x,1)).
	\end{equation}
	
	Combining (\ref{eq:split-into-two-parts}), (\ref{ineq:integral-bound-I}) and (\ref{ineq:integral-bound-II}), we know
	\begin{equation*}
		u(x,1)
		\leq
		\Big(C\rho^{-n}+\frac{4}{100}Z+C\rho Z\Big)\exp(-\mathcal{N}^*_0(x,1)),
	\end{equation*}
	provided $\rho\leq\overline{\rho}$. So if we fix some small $\rho$ and assume $Z\geq\underline{Z}$, then $C\rho^{-n}+\frac{4}{100}Z+C\rho Z\leq Z/2$.
\end{proof}

The above theorem will imply the following stronger pointwise Gaussian bounds.
\begin{Theorem}[Stronger upper Gaussian bounds]
	Assume $[s,t]\subset I$ and $E(g)\geq-Ag$ for some $A\geq 0$. Let $(z,s)\in M\times I$ be an $(H_n,A)$-center of a point $(x,t)\in M\times I$. Then for any $\epsilon>0$ and $y,y'\in M$, there exists a constant $C$ depending only on $n$, $R_{\mathrm{min}}(t-s)$, $|E|_{\mathrm{max}}(t-s)$ and $\epsilon$ such that
	\begin{align*}
		K(x,t;y,s)
		&\leq
		\frac{C\exp(-\mathcal{N}_{x,t}(t-s))}{(t-s)^{n/2}}
		\exp\left(-\frac{d_s^2(z,y)}{(8+\epsilon)(t-s)}\right),\\
		K(x,t;y,s)K(x,t;y',s)
		&\leq
		\frac{C\exp(-2\mathcal{N}_{x,t}(t-s))}{(t-s)^{n}}
		\exp\left(-\frac{d_s^2(y,y')}{(8+\epsilon)(t-s)}\right).
	\end{align*}
	Moreover, if $[s-r^2,t]\subset I$ for some $r^2\geq\epsilon(t-s)$, then we may replace $\mathcal{N}_{x,t}(t-s)$ by $\mathcal{N}_{y,s}(r^2)$ in the above two inequalities.
\end{Theorem}
This can be proved by slightly modifying the proof of \cite[Theorem 7.2]{bamler2020entropy}.
Since it will not be used in this paper, we omit the proof.

Then we turn to a gradient bound on the heat kernel.
\begin{Theorem}[Gradient estimate]\label{thm:heat-kernel-gradient-estimate}
	Assume $[s,t]\subset I$. There are constants $C_0,C$ depending only on $n$, $R_{\mathrm{min}}(t-s)$ and $|E|_{\mathrm{max}}(t-s)$ such that
	\begin{equation}\label{ineq:heat-kernel-gradient-estimate}
		\frac{|\nabla_xK|(x,t;y,s)}{K(x,t;y,s)}
		\leq
		\frac{C}{(t-s)^{1/2}}\sqrt{\log\left(
		\frac{C_0\exp(-\mathcal{N}_{x,t}(t-s))}{(t-s)^{n/2}K(x,t;y,s)}\right)}.
	\end{equation}
\end{Theorem}
\begin{proof}
	After parabolic rescaling and time-shifting, we may assume $s=0$, $t=2$. Set $A=|E|_{\mathrm{max}}$. Then $E\geq -Ag$. In the following, all generic constants may depend on $n,R_{\mathrm{min}},|E|_{\mathrm{max}}$, without further specification.
	
	Set $u:=K(\cdot,1;y,0)$ and $d\nu:=d\nu_{x,2;1}=K(x,2;\cdot,1)dV_{g(1)}$. Then
	\begin{align}
		K(x,2;y,0)
		&=
		\int_M K(x,2;\cdot,1)K(\cdot,1;y,0)\,dV_{g(1)}
		=\int_M u\,d\nu,  \notag  \\
		|\nabla_xK|(x,2;y,0)
		&\leq
		\int_M\frac{|\nabla_xK|(x,2;\cdot,1)}{K(x,2;\cdot,1)}u\,d\nu.\label{ineq:norm-of-gradient}
	\end{align}
	Let $(z,1)$ be an $(H_n,A)$-center of $(x,2)$. Then by \hyperref[thm:gradient-estimate-Nash-entropy]{Theorem~\ref*{thm:gradient-estimate-Nash-entropy}} and \hyperref[cor:oscillation-of-pointed-Nash-entropy]{Corollary~\ref*{cor:oscillation-of-pointed-Nash-entropy}},
	\begin{equation*}
		-\mathcal{N}^*_0(\cdot,1)
		\leq
		-\mathcal{N}^*_0(z,1)
		+Cd_{1}(z,\cdot)
		\leq
		-\mathcal{N}^*_0(x,2)
		+C+Cd_{1}(z,\cdot).
	\end{equation*}
	Combining this with \hyperref[thm:heat-kernel-upper-bound]{Theorem~\ref*{thm:heat-kernel-upper-bound}}, we get
	\begin{equation*}
		u\leq C\exp(-\mathcal{N}^*_0(\cdot,1))
		\leq C\exp(-\mathcal{N}^*_0(x,2))\exp(Cd_{1}(z,\cdot)).
	\end{equation*}
	Using \hyperref[prop:basic-properties-of-center]{Proposition~\ref*{prop:basic-properties-of-center}}(4),
	\begin{align}
		\int_M u^2\,d\nu
		&\leq
		C\exp(-2\mathcal{N}^*_0(x,2))
		\int_M   \exp(Cd_{1}(z,\cdot))   \,d\nu  \notag  \\
		&\leq
		C\exp(-2\mathcal{N}^*_0(x,2))
		\left(1+C\int_0^\infty e^{Cr}\nu(M\setminus B(z,1,r))\,dr\right)  \notag  \\
		&\leq
		C\exp(-2\mathcal{N}^*_0(x,2))
		\left(1+C\int_0^\infty e^{Cr}\cdot 2\exp\bigg(
		-\frac{(r-\sqrt{2H_n\sigma_A(1)})_+^2}{8\sigma_A(1)}   \bigg)   dr\right)  \notag  \\
		&\leq
		C\exp(-2\mathcal{N}^*_0(x,2)) \label{ineq:integral-of-u2}.
	\end{align}
	
	By \hyperref[thm:heat-kernel-upper-bound]{Theorem~\ref*{thm:heat-kernel-upper-bound}} we can fix a constant $C_0$ such that
	\begin{equation*}
		K(x,2;y,0)
		\leq
		\frac{1}{2}C_0\exp(-\mathcal{N}^*_0(x,2)).
	\end{equation*}
	Set
	\begin{equation*}
		a:=
		\left(\frac{K(x,2;y,0)}{C_0\exp(-\mathcal{N}^*_0(x,2))}\right)^2
		\leq\frac{1}{4}.
	\end{equation*}
	Choose $h\geq 0$ such that we can find a subset $X_h\subset M$ with
	\begin{equation*}
		\left\{\frac{|\nabla_xK|(x,2;\cdot,1)}{K(x,2;\cdot,1)}>h\right\}
		\subset
		X_h
		\subset
		\left\{\frac{|\nabla_xK|(x,2;\cdot,1)}{K(x,2;\cdot,1)}\geq h\right\}
	\end{equation*}
	and $\nu(X_h)=a$. By \hyperref[prop:integral-gradient-bound]{Proposition~\ref*{prop:integral-gradient-bound}} and the fact $-\log(a/2)\leq 2(-\log a)$, we have
	\begin{equation*}
		ah
		\leq
		\int_{X_h}   \frac{|\nabla_xK|(x,2;\cdot,1)}{K(x,2;\cdot,1)}\,d\nu
		\leq
		C\nu(X_h)(-\log\nu(X_h))^{1/2}
		=Ca(-\log a)^{1/2},
	\end{equation*}
	i.e., $h\leq C(-\log a)^{1/2}$. By (\ref{ineq:norm-of-gradient}), \hyperref[prop:integral-gradient-bound]{Proposition~\ref*{prop:integral-gradient-bound}}, (\ref{ineq:integral-of-u2}) and the fact $h\leq C(-\log a)^{1/2}$,
	\begin{align*}
		|\nabla_xK|(x,2;y,0)
		&\leq
		\int_M\frac{|\nabla_xK|(x,2;\cdot,1)}{K(x,2;\cdot,1)}u\,d\nu\\
		&=
		\int_{X_h}\frac{|\nabla_xK|(x,2;\cdot,1)}{K(x,2;\cdot,1)}u\,d\nu
		+\int_{M\setminus X_h}\frac{|\nabla_xK|(x,2;\cdot,1)}{K(x,2;\cdot,1)}u\,d\nu\\
		&\leq
		\bigg(   \int_{X_h}\left(\frac{|\nabla_xK|(x,2;\cdot,1)}{K(x,2;\cdot,1)}\right)^2   d\nu\bigg)^{1/2}
		\bigg(\int_{X_h}   u^2   d\nu\bigg)^{1/2}
		+h\int_M u\,d\nu\\
		&\leq
		C(-a\log a)^{1/2}\exp(-\mathcal{N}^*_0(x,2))
		+C(-\log a)^{1/2}K(x,2;y,0)\\
		&=
		C(-\log a)^{1/2}K(x,2;y,0).
	\end{align*}
	This is precisely (\ref{ineq:heat-kernel-gradient-estimate}).
\end{proof}

In the presence of a global bound on Nash entropy, (\ref{ineq:heat-kernel-gradient-estimate}) will imply
\begin{equation*}
	\frac{|\nabla_xK|(x,t;y,s)}{K(x,t;y,s)}
	\leq
	\frac{C}{(t-s)^{1/2}}\sqrt{\log\left(
		\frac{C}{(t-s)^{n/2}K(x,t;y,s)}\right)},
\end{equation*}
which is equivalent to
\begin{equation*}
	\left|\nabla_x\sqrt{\log\left(
		\frac{C}{(t-s)^{n/2}K(x,t;y,s)}\right)}\right|_{g(t)}
	\leq
	\frac{C}{(t-s)^{1/2}}.
\end{equation*}
Similarly, we can obtain the following corollary.

\begin{Corollary}\label{cor:heat-kernel-gradient-estiamte}
	Assume $[s,t]\subset I$. There are constants $C_0,C$ depending only on $n$, $R_{\mathrm{min}}(t-s)$ and $|E|_{\mathrm{max}}(t-s)$ such that
	\begin{equation*}
		\left|\nabla_x\sqrt{\log\left(
		\frac{C_0\exp(-\mathcal{N}_{x,t}(t-s))}{(t-s)^{n/2}K(x,t;y,s)}\right)}\right|_{g(t)}
		\leq
		\frac{C}{(t-s)^{1/2}}.
	\end{equation*}
\end{Corollary}
\begin{proof}
	Fix $t$ and $(y,s)$. Let $C_0$ be the constant obtained in \hyperref[thm:heat-kernel-gradient-estimate]{Theorem~\ref*{thm:heat-kernel-gradient-estimate}}. Set
	\begin{equation}\label{eq:auxiliary-function-v}
		v(x)
		:=\sqrt{\log\left(
		\frac{C_0\exp(-\mathcal{N}_{x,t}(t-s))}{(t-s)^{n/2}K(x,t;y,s)}\right)}.
	\end{equation}
	Then (\ref{ineq:heat-kernel-gradient-estimate}) is equivalent to
	\begin{equation*}
		|2v\nabla v+\nabla\mathcal{N}^*_{s}(\cdot,t)|
		\leq
		\frac{C}{(t-s)^{1/2}}v.
	\end{equation*}
	This implies
	\begin{align*}
		|\nabla v|
		&\leq
		\frac{C}{(t-s)^{1/2}}
		+\bigg|\frac{\nabla\mathcal{N}^*_{s}(\cdot,t)}{2v}\bigg|.
	\end{align*}
	Since $v$ is bounded below by \hyperref[thm:heat-kernel-upper-bound]{Theorem~\ref*{thm:heat-kernel-upper-bound}} and $|\nabla\mathcal{N}^*_{s}(\cdot,t)|$ is bounded above by \hyperref[thm:gradient-estimate-Nash-entropy]{Theorem~\ref*{thm:gradient-estimate-Nash-entropy}}, the result follows.
\end{proof}

\begin{Corollary}\label{cor:heat-kernel-lower-bound}
	Let $[s,t]\subset I$, $x_0\in M$, $\alpha>0$. Assume
	\begin{equation*}
		K(x_0,t;y,s)
		\geq
		\alpha\cdot\frac{\exp(-\mathcal{N}_{x_0,t}(t-s))}{(t-s)^{n/2}}.
	\end{equation*}
	\begin{enumerate}
		\item For any $\beta\in(0,\alpha)$, there exists $\delta>0$ depending only on $n$, $R_{\mathrm{min}}(t-s)$, $|E|_{\mathrm{max}}(t-s)$, $\alpha$ and $\beta$ such that
		\begin{equation*}
			K(x,t;y,s)
			\geq
			\beta\cdot\frac{\exp(-\mathcal{N}_{x,t}(t-s))}{(t-s)^{n/2}},
			\qquad\forall\,x\in B(x_0,t,\delta\sqrt{t-s}).
		\end{equation*}
		
		\item For any $\lambda>0$, there exists $c>0$ depending only on $n$, $R_{\mathrm{min}}(t-s)$, $|E|_{\mathrm{max}}(t-s)$, $\alpha$ and $\lambda$ such that
		\begin{equation*}
			K(x,t;y,s)
			\geq
			c\cdot\frac{\exp(-\mathcal{N}_{x,t}(t-s))}{(t-s)^{n/2}},
			\qquad\forall\,x\in B(x_0,t,\lambda\sqrt{t-s}).
		\end{equation*}
	\end{enumerate}
\end{Corollary}
\begin{proof}
	We prove (1) first. Define $v$ as in (\ref{eq:auxiliary-function-v}). Let $x$ be any point such that
	\begin{equation*}
		K(x,t;y,s)
		\leq
		\beta\cdot\frac{\exp(-\mathcal{N}_{x,t}(t-s))}{(t-s)^{n/2}}.
	\end{equation*}
	Then the distance between $x$ and $x_0$ at time $t$ is bounded from below by \hyperref[cor:heat-kernel-gradient-estiamte]{Corollary~\ref*{cor:heat-kernel-gradient-estiamte}} as
	\begin{equation*}
		\frac{Cd_t(x,x_0)}{(t-s)^{1/2}}
		\geq
		{v(x)}-{v(x_0)}
		\geq
		\sqrt{\log(\beta^{-1}C_0)}-\sqrt{\log(\alpha^{-1}C_0)}.
	\end{equation*}
	Assertion (2) is proved similarly. For any $x\in B(x_0,t,\lambda\sqrt{t-s})$, \hyperref[cor:heat-kernel-gradient-estiamte]{Corollary~\ref*{cor:heat-kernel-gradient-estiamte}} gives
	\begin{equation*}
		v(x)
		\leq
		v(x_0)+\frac{Cd_t(x,x_0)}{(t-s)^{1/2}}
		\leq
		v(x_0)+\lambda C
		\leq
		\sqrt{\log(\alpha^{-1}C_0)}+\lambda C.
	\end{equation*}
\end{proof}

The above \hyperref[cor:heat-kernel-lower-bound]{Corollary~\ref*{cor:heat-kernel-lower-bound}} will imply the following partial lower bound on heat kernel.
\begin{Theorem}[Partial lower bound]\label{thm:partial-lower-bound}
	Given $(x_0,t_0)\in M\times I$ and $\lambda>0$. For any $s<t_0$ with $[s,t_0]\subset I$, there exists a point $y\in M$ such that
	\begin{equation*}
		K(x,t_0;y,s)
		\geq
		c\cdot
		\frac{\exp(-\mathcal{N}_{x,t_0}(t_0-s))}{(t_0-s)^{n/2}},
		\qquad\forall\,x\in B(x_0,t_0,\lambda\sqrt{t_0-s})
	\end{equation*}
	for some $c>0$ depending only on $n$, $R_{\mathrm{min}}(t_0-s)$, $|E|_{\mathrm{max}}(t_0-s)$, $\lambda$.
\end{Theorem}
\begin{proof}
	By the definition of pointed Nash entropy,
	\begin{equation*}
		\mathcal{N}_{x_0,t_0}(t_0-s)
		=
		-\int_M   \log K(x_0,t_0;\cdot,s)   \,d\nu_{x_0,t_0;s}
		-\frac{n}{2}\log(4\pi(t_0-s))-\frac{n}{2}.
	\end{equation*}
	Therefore we can find a point $y\in M$ with
	\begin{equation*}
		\log K(x_0,t_0;y,s)
		\geq
		-\mathcal{N}_{x_0,t_0}(t_0-s)
		-\frac{n}{2}\log(4\pi(t_0-s))-\frac{n}{2},
	\end{equation*}
	or equivalently,
	\begin{equation*}
		K(x_0,t_0;y,s)
		\geq
		(4\pi)^{-n/2}e^{-n/2}
		\frac{\exp(-\mathcal{N}_{x_0,t_0}(t_0-s))}{(t_0-s)^{n/2}}.
	\end{equation*}
	Our results then follows from \hyperref[cor:heat-kernel-lower-bound]{Corollary~\ref*{cor:heat-kernel-lower-bound}}(2).
\end{proof}

The above partial lower bound on heat kernel implies the following no-local-inflating estimate.
\begin{Theorem}[No local inflating]\label{thm:no-local-inflating}
	Given $(x_0,t_0)\in M\times I$, $\lambda>0$ and a scale $r>0$ with $[t_0-r^2,t_0]\subset I$. Then
	\begin{equation*}
		\frac{|B(x_0,t_0,\lambda r)|_{g(t_0)}}{r^n}
		\leq
		C\exp(\mathcal{N}_{x_0,t_0}(r^2))
	\end{equation*}
	for some $C$ depending only on $n$, $R_{\mathrm{min}}r^2$, $|E|_{\mathrm{max}}r^2$ and $\lambda$.
\end{Theorem}
\begin{proof}
	Let $y\in M$ and $c$ be as in \hyperref[thm:partial-lower-bound]{Theorem~\ref*{thm:partial-lower-bound}}. By (\ref{ineq:heat-kernel-integral-bound}), \hyperref[thm:partial-lower-bound]{Theorem~\ref*{thm:partial-lower-bound}} and \hyperref[thm:gradient-estimate-Nash-entropy]{Theorem~\ref*{thm:gradient-estimate-Nash-entropy}}, we know for $s<t_0$ that
	\begin{align*}
		\exp\big(-(R-\tfrac{1}{2}\mathrm{tr}\,E)_{\mathrm{min}}(t_0-s)\big)
		&\geq
		\int_MK(\cdot,t_0;y,s)\,dV_{g(t_0)}\\
		&\geq
		\int_{B(x_0,t_0,\lambda\sqrt{t_0-s})}
		c\cdot\frac{\exp(-\mathcal{N}_{\cdot,t_0}(t_0-s))}{(t_0-s)^{n/2}}\,dV_{g(t_0)}\\
		&\geq
		|B(x_0,t_0,\lambda\sqrt{t_0-s})|_{g(t_0)}
		c\cdot\frac{\exp(-\mathcal{N}_{x_0,t_0}(t_0-s)-C\lambda)}{(t_0-s)^{n/2}}.
	\end{align*}
	The result follows by taking $s=t_0-r^2$.
\end{proof}

\section{\texorpdfstring{$P^*$}{P*}-Parabolic Neighborhoods}
The notion of $P^*$-parabolic neighborhood also carries over to the more general case $E(g)\geq -Ag$, where $A\geq 0$.

For $(x,t)\in M\times I$ and $r,T^+,T^-\geq 0$, the conventional parabolic neighborhood is defined as
\begin{equation}\label{eq:conventional-parabolic-neighborhood}
	P(x,t;r,-T^{-},T^{+})
	:=B(x,t,r)\times([t-T^{-},t+T^{+}]\cap I),
\end{equation}
where we may omit $-T^{-}$ or $T^{+}$ if it is zero.  We use the abbreviations
\begin{align*}
	P(x,t;r)&:=P(x,t;r,-r^2,r^2),  \\
	P^+(x,t;r)&:=P(x,t;r,0,r^2),  \\
	P^-(x,t;r)&:=P(x,t;r,-r^2,0).
\end{align*}
\begin{Definition}[$P^*$-parabolic neighborhood]
	Suppose $(x,t)\in M\times I$, $r,T^+,T^-\geq 0$ are such that $t-T^-\in I$. The $P^*$-parabolic neighborhood $P^*(x,t;r,-T^-,T^+)\subset M\times I$ is defined to be the set of points $(y,s)\in M\times I$ satisfying $s\in[t-T^-,t+T^+]$ and
	\begin{equation*}
		d_{W_1}^{g(t-T^-)}
		(\nu_{x,t;t-T^-},\nu_{y,s;t-T^-})<r.
	\end{equation*}
	For any $A\geq 0$, we define
	\begin{align*}
		&P^*_A(x,t;r):=P^*(x,t;re^{Ar^2},-r^2,r^2),\\
		&P^{*+}_{A}(x,t;r):=P^*(x,t;re^{Ar^2},0,r^2),\\
		&P^{*-}_{A}(x,t;r):=P^*(x,t;re^{Ar^2},-r^2,0).
	\end{align*}
	For $A=0$, we will drop the parameter $A$ and simply write $P^*,P^{*\pm}$. When $A>0$, the additional factor $e^{Ar^2}$ is chosen to make the containment relationships admitting the same coefficients (as the $A=0$ case). See \hyperref[prop:containment-relationship]{Proposition~\ref*{prop:containment-relationship}} (3)(5)(6).
\end{Definition}

The following proposition shows that $P^*$-parabolic neighborhoods satisfy similar containment relationships as conventional ones.
\begin{Proposition}[Containment relationships]\label{prop:containment-relationship}
	Suppose $E(g)\geq-Ag$ for some $A\geq 0$.
	The following holds as long as the corresponding $P^*$-parabolic neighborhoods are defined.
	\begin{enumerate}
		\item If $0\leq T_1^\pm\leq T_2^\pm$ and $0\leq e^{-\frac{A}{2}T_1^-}r_1\leq e^{-\frac{A}{2}T_2^-}r_2$, then
		\begin{equation*}
			P^*(x,t;r_1,-T_1^-,T_1^+)
			\subset
			P^*(x,t;r_2,-T_2^-,T_2^+).
		\end{equation*}
		In particular, for any $0\leq r_1\leq r_2$, we have $P^*_A(x,t;r_1)\subset P^*_A(x,t;r_2)$.
		
		\item If $r,T^\pm\geq 0$ and $(y,s)\in P^*(x,t;r,-T^-,T^+)$, then
		\begin{equation*}
			(x,t)\in P^*(y,s;r,-(s-t+T^-),T^-),
		\end{equation*}
		\begin{equation*}
			P^*(x,t;r,-T^-,T^+)
			\subset
			P^*(y,s;2r,-(s-t+T^-),t-s+T^+).
		\end{equation*}
		
		\item If $r\geq 0$ and $(y,s)\in P^*_A(x,t;r)$, then
		\begin{equation*}
			(x,t)\in P^*_A(y,s;\sqrt{2}r)
			\quad\text{and}\quad
			P^*_A(x,t;r)\subset P^*_A(y,s;2r).
		\end{equation*}
	
		\item If $r_i,T_i^\pm\geq 0$ for $i=1,2$ and $(x_1,t_1)\in P^*(x_2,t_2;r_2,-T_2^-,T_2^+)$, then
		\begin{equation*}
			P^*(x_1,t_1;r_1,-T_1^-,T_1^+)
			\subset
			P^*(x_2,t_2;e^{\frac{A}{2}(t_1-t_2+T_2^-)}r_1+e^{\frac{A}{2}T_1^-}r_2,-(T_1^-+T_2^-),T_1^++T_2^+).
		\end{equation*}
		
		\item If $r_1,r_2\geq 0$ and $(x_1,t_1)\in P^*_A(x_2,t_2;r_2)$, then
		\begin{equation*}
			P^*_A(x_1,t_1;r_1)\subset P^*_A(x_2,t_2;r_1+r_2).
		\end{equation*}
		
		\item If $r_1,r_2\geq 0$ and $P^*_A(x_1,t_1;r_1)\cap P^*_A(x_2,t_2;r_2)\neq\varnothing$, then
		\begin{equation*}
			P^*_A(x_1,t_1;r_1)\subset P^*_A(x_2,t_2;2r_1+r_2).
		\end{equation*}
	\end{enumerate}
\end{Proposition}
\begin{proof}
	Assertion (1) follows from (\ref{ineq:monotonicity-of-W1}), since for any $(y,s)\in P^*(x,t;r_1,-T_1^-,T_1^+)$ we have
	\begin{equation*}
		d_{W_1}^g(\nu_{x,t},\nu_{y,s})(t-T_2^-)
		\leq
		e^{\frac{A}{2}(T_2^--T_1^-)}
		d_{W_1}^g(\nu_{x,t},\nu_{y,s})(t-T_1^-)
		<
		e^{\frac{A}{2}(T_2^--T_1^-)}r_1
		\leq r_2.
	\end{equation*}
	Assertion (2) follows directly from the definition and triangle inequality. Assertion (3) is a consequence of (1) and (2), as
	\begin{equation*}
		(x,t)\in P^*(y,s;re^{Ar^2},-(s-t+r^2),r^2)
		\subset
		P^*(y,s;re^{Ar^2}e^{\frac{A}{2}2r^2},-2r^2,2r^2)
		\subset
		P^*_A(y,s;\sqrt{2}r)
	\end{equation*}
	and
	\begin{align*}
		P^*(x,t;re^{Ar^2},-r^2,r^2)
		&\subset
		P^*(y,s;2re^{Ar^2},-(s-t+r^2),t-s+r^2)\\
		&\subset
		P^*(y,s;2re^{Ar^2}e^{\frac{A}{2}4r^2},-4r^2,4r^2)
		\subset
		P^*_A(y,s;2r).
	\end{align*}
	The remaining assertions can be verified in the same way.
\end{proof}

Next we want to compare conventional and $P^*$-parabolic neighborhoods under a local Ricci curvature bound. The result is stated in \hyperref[cor:comparison-of-parabolic-neighborhoods]{Corollary~\ref*{cor:comparison-of-parabolic-neighborhoods}}. Let's start from the following Hamilton-Perelman type distance distortion estimate, which will also be used in our proof of pseudolocality.
\begin{Lemma}[Distance distortion estimate]\label{lem:distance-distortion}
	Let $g(t)$ be a smooth family of complete Riemannian metrics over $M^n$ with $E(g)\geq-Ag$ for some $A\geq 0$.
	\begin{enumerate}
		\item Assume $\Ric(\cdot,t_0)\leq(n-1)K$ over $B(x_0,t_0,r_0)$. Then the distance function $d(x,t):=d_{t}(x,x_0)$ satisfies
		\begin{equation*}
			\square d
			\geq
			-(n-1)\Big(\frac{2}{3}Kr_0+r_0^{-1}\Big)-\frac{A}{2}d
		\end{equation*}
		at time $t=t_0$, outside $B(x_0,t_0,r_0)$. Here the inequality is understood in the barrier sense when necessary.
		
		\item Assume $d_{t_0}(x_0,y_0)\geq 2r_0$ and $\Ric(\cdot,t_0)\leq(n-1)K$ over $B(x_0,t_0,r_0)\cup B(y_0,t_0,r_0)$. Then
		\begin{equation*}
			\frac{d}{dt}\Big|_{t=t_0}d_t(x_0,y_0)
			\geq
			-2(n-1)\Big(\frac{2}{3}Kr_0+r_0^{-1}\Big)-\frac{A}{2}d_{t_0}(x_0,y_0).
		\end{equation*}
	\end{enumerate}
\end{Lemma}
\begin{proof}
	The proof is a straightforward modification of that in Lemma 8.3 of \cite{perelman2002entropy}.
\end{proof}

\begin{Lemma}\label{lem:consequence-of-distance-distortion}
	Let $(g(t))_{t\in[0,T]}$ be a smooth family of complete Riemannian metrics over $M^n$ with $E(g)\geq-Ag$ for some $A\geq 0$. Assume
	\begin{equation*}
		\Ric(x,t)\leq\frac{A'}{\beta^2\bar{\sigma}_{A}(t)},
		\quad
		\forall\,x\in B\big(x_0,t,\beta\sqrt{\bar{\sigma}_{A}(t)}\big),
		\ t\in(0,T]
	\end{equation*}
	for some $x_0\in M$, $0\leq A'<\infty$ and $\beta>0$. Then the distance function $d(x,t):=d_t(x,x_0)$ satisfies
	\begin{equation}\label{ineq:deformed-square-d}
		(\partial_t^{A/2}-\Delta)
		(  d+2A''\sqrt{\bar{\sigma}_{A}(t)}  )
		\geq 0
	\end{equation}
	outside $B\big(x_0,t,\beta\sqrt{\bar{\sigma}_{A}(t)}\big)$, where $A'':=(\frac{2}{3}A'+n-1)\beta^{-1}$. Recall $\partial_t^{A/2}:=e^{-\frac{A}{2}t}  \partial_t  (e^{\frac{A}{2}t}\cdot)=\partial_t+\frac{A}{2}$.
\end{Lemma}
\begin{proof}
	This follows from \hyperref[lem:distance-distortion]{Lemma~\ref*{lem:distance-distortion}}, as in \cite[Lemma 5.6]{wang2018local}.
	At time $t_0\in(0,T]$, taking $r_0=\beta\sqrt{\bar{\sigma}_{A}(t_0)}$ and $K=\frac{A'}{(n-1)\beta^2\bar{\sigma}_{A}(t_0)}$ in \hyperref[lem:distance-distortion]{Lemma~\ref*{lem:distance-distortion}}, we have
	\begin{align*}
		\square d
		&\geq
		-(n-1)\left(
		\frac{2}{3}\cdot\frac{A'}{n-1}
		+1
		\right)\beta^{-1}\frac{1}{\sqrt{\bar{\sigma}_{A}(t_0)}}
		-\frac{A}{2}d\\
		&=
		-2A''\partial_t^{A/2}\sqrt{\bar{\sigma}_{A}(t_0)}
		-\frac{A}{2}d,
	\end{align*}
	which is equivalent to (\ref{ineq:deformed-square-d}).
\end{proof}

The following theorem gives a lower bound on the amount of probability that $\nu_{x_0,t_0;t}$ assigns to some neighborhood of $x_0$, under certain Ricci curvature bound.

\begin{Theorem}[c.f. \cite{wang2018local} Proof of Theorem 5.4]\label{thm:concentration-under-Ricci-bound}
	Let $(g(t))_{t\in[0,T]}$ be a smooth family of complete Riemannian metrics over $M^n$ with $E(g)\geq-Ag$ for some $A\geq 0$. Suppose
	\begin{equation*}
		\Ric(x,t)\leq\frac{A'}{\beta^2\bar{\sigma}_{A}(t)},
		\quad
		\forall\,x\in B\big(x_0,t,\beta\sqrt{\bar{\sigma}_{A}(t)}\big),
		\ t\in(0,T]
	\end{equation*}
	for some $x_0\in M$, $A'\in(0,\infty)$ and $\beta\in(0,1]$. Set $A'':=(\frac{2}{3}A'+n-1)\beta^{-1}$,
	\begin{align*}
		&\Omega_T':=B\big(x_0,T,8A''\sqrt{\bar{\sigma}_{A}(T)}\big),\\
		&\Omega_0:=B\big(x_0,0,20A''\sqrt{\sigma_A(T)}\big).
	\end{align*}
	Then for any solution $u$ to the backward heat equation $\square^*u=0$ with $\int_M u \,dV_g=1$ and $u(\cdot,T)$ supported in $\Omega_T'$, it holds
	\begin{equation*}
		\int_{\Omega_0}   u   \,dV_g   \bigg|_{t=0}
		\geq
		\exp\Big(-\frac{1}{10(A'')^2}\Big).
	\end{equation*}
\end{Theorem}
\begin{proof}
	Choose $\psi$ as in Theorem 5.4 of \cite{wang2018local} (and \cite{kleiner2008notes} $\S 34$). Namely, $\psi$ is a cut-off function such that $\psi\equiv 1$ on $(-\infty,1)$, $\psi\equiv 0$ on $(2,\infty)$, $-10\leq\psi'\leq 0$, $\psi''\geq-10\psi$, $(\psi')^2\leq 10\psi$. Define
	\begin{equation*}
		\eta(x,t)=\frac{1}{C_0}
		e^{\frac{A}{2}t}\big(d(x,t)+2A''\sqrt{\bar{\sigma}_A(t)}\big),
	\end{equation*}
	where $A''$ is defined as in \hyperref[lem:consequence-of-distance-distortion]{Lemma~\ref*{lem:consequence-of-distance-distortion}} and $C_0:=e^{\frac{A}{2}T}10A''\sqrt{\bar{\sigma}_A(T)}=10A''\sqrt{\sigma_A(T)}$. By \hyperref[lem:consequence-of-distance-distortion]{Lemma~\ref*{lem:consequence-of-distance-distortion}}, $\square\eta(x,t)\geq 0$ outside $B\big(x_0,t,\beta\sqrt{\bar{\sigma}_{A}(t)}\big)$. Over $B\big(x_0,t,\beta\sqrt{\bar{\sigma}_{A}(t)}\big)$, we have
	\begin{equation*}
		\eta(x,t)
		\leq
		\frac{1}{C_0}e^{\frac{A}{2}t}(\beta+2A'')\sqrt{\bar{\sigma}_{A}(t)}
		=
		\frac{(\beta+2A'')\sqrt{{\sigma}_{A}(t)}}{10A''\sqrt{{\sigma}_{A}(T)}}
		\leq 1
	\end{equation*}
		and therefore $\psi(\eta)=1$ over there. These two facts together imply
	\begin{equation*}
		\square(\psi(\eta))
		=
		\psi'(\eta)\square\eta-\psi''(\eta)|\nabla\eta|^2
		\leq
		\frac{10}{C_0^2}e^{At}\psi(\eta).
	\end{equation*}
	Take
	\begin{equation*}
		h(x,t)=\exp\Big(-\frac{10}{C_0^2}\sigma_A(t)\Big)\psi(\eta(x,t)).
	\end{equation*}
	Then $\square h\leq 0$ over $M\times[0,T]$. At time $T$, $\mathrm{supp}(h)\subset B\big(x_0,T,18A''\sqrt{\bar{\sigma}_{A}(T)}\big)$ and $\psi(\eta)\equiv 1$ over $B\big(x_0,T,8A''\sqrt{\bar{\sigma}_{A}(t)}\big)$. Hence
	\begin{align*}
		\int_M   hu   \,dV_g   \bigg|_{t=0}
		\geq
		\int_M   hu   \,dV_g   \bigg|_{t=T}
		&=
		\int_{\Omega_T'}   hu   \,dV_g   \bigg|_{t=T}\\
		&=
		\exp\Big(-\frac{10}{C_0^2}\sigma_A(T)\Big)
		\int_{\Omega_T'}   u   \,dV_g   \bigg|_{t=T}
		=
		\exp\Big(-\frac{10}{C_0^2}\sigma_A(T)\Big).
	\end{align*}
	Note $h=\psi(\eta)$ at time $t=0$. Thus $\mathrm{supp}(h(\cdot,0))\subset B(x_0,0,2C_0)=\Omega_0$. It follows that
	\begin{equation*}
		\int_{\Omega_0}   u   \,dV_g   \bigg|_{t=0}
		\geq
		\int_{M}   hu   \,dV_g   \bigg|_{t=0}
		\geq
		\exp\Big(-\frac{10}{C_0^2}\sigma_A(T)\Big)
		=
		\exp\Big(-\frac{1}{10(A'')^2}\Big).
	\end{equation*}
\end{proof}

\begin{Remark}\label{rmk:gradient-bound-on-cut-off-function}
	The cut-off function $h$ constructed in the proof of \hyperref[thm:concentration-under-Ricci-bound]{Theorem~\ref*{thm:concentration-under-Ricci-bound}} satisfies
	\begin{equation*}
		|\nabla h|(\cdot,0)
		\leq
		\frac{1}{A''\sqrt{\sigma_A(T)}},
	\end{equation*}
	\begin{equation*}
		\big|\nabla\sqrt{h}\big|^2(\cdot,0)
		\leq
		\frac{1}{40(A'')^2\sigma_A(T)}.
	\end{equation*}
\end{Remark}

\begin{Remark}\label{rmk:concentration}
	Set $\Omega_0'=B\big(x_0,0,10A''\sqrt{\sigma_A(T)}\big)$. By a slight modification of the cut-off function (namely replace $C_0$ by $C_0'=C_0/2$), one can show
	\begin{equation*}
		\int_{\Omega_0'}   u   \,dV_g   \bigg|_{t=0}
		\geq
		\exp\Big(-\frac{2}{5(A'')^2}\Big),
	\end{equation*}
	provided we require $u(\cdot,T)$ being supported in $B\big(x_0,T,3A''\sqrt{\bar{\sigma}_{A}(T)}\big)$.
\end{Remark}

\begin{Corollary}\label{cor:concentration-under-Ricci-bound}
	In the same setting as \hyperref[thm:concentration-under-Ricci-bound]{Theorem~\ref*{thm:concentration-under-Ricci-bound}}, we have
	\begin{equation*}
		d_{W_1}^{g(0)}(\nu_{x_0,T;0},\delta_{x_0})
		\leq C(n,A',\beta)\sqrt{\sigma_A(T)}.
	\end{equation*}
\end{Corollary}
\begin{proof}
	By \hyperref[thm:concentration-under-Ricci-bound]{Theorem~\ref*{thm:concentration-under-Ricci-bound}}, there exist $c_0,c_1>0$ depending only on $A',\beta$ such that
	\begin{equation*}
		\nu_{x_0,T;0}\big(B(x_0,0,c_0\sqrt{\sigma_A(T)})\big)\geq c_1.
	\end{equation*}
	Let $(z,0)$ be an $(H_n,A)$-center of $(x_0,T)$. Taking $\delta=2/c_1$ in \hyperref[prop:basic-properties-of-center]{Proposition~\ref*{prop:basic-properties-of-center}}(3), we get
	\begin{equation*}
		B(x_0,0,c_0\sqrt{\sigma_A(T)})
		\cap B\big(z,0,\sqrt{\delta H_n\sigma_A(T)}\big)
		\neq\varnothing.
	\end{equation*}
	Therefore
	\begin{align*}
		d_{W_1}^{g(0)}(\nu_{x_0,T;0},\delta_{x_0})
		&\leq
		d_{W_1}^{g(0)}(\nu_{x_0,T;0},\delta_z)
		+d_{W_1}^{g(0)}(\delta_z,\delta_{x_0})\\
		&\leq
		\sqrt{H_n\sigma_A(T)}
		+c_0\sqrt{\sigma_A(T)}
		+\sqrt{\delta H_n\sigma_A(T)}.
	\end{align*}
\end{proof}

Using the above \hyperref[cor:concentration-under-Ricci-bound]{Corollary~\ref*{cor:concentration-under-Ricci-bound}}, we can obtain the following containment relationships between conventional and $P^*$-parabolic neighborhoods.
\begin{Corollary}\label{cor:comparison-of-parabolic-neighborhoods}
	Suppose $E\geq -Ag$ for some $A\geq 0$. Given $\lambda>0$, $K,T^\pm\geq 0$, $(x,t)\in M\times I$ and a scale $r>0$ with $[t-T^{-}r^2,t]\subset I$.
	\begin{enumerate}
		\item If $|\Ric|\leq Kr^{-2}$ on $P(x,t;Lr,-T^{-}r^2,T^{+}r^2)$ for some $L\geq\underline{L}(Ar^2,\lambda,K,T^{\pm})$, then
		\begin{equation*}
			P(x,t;\lambda r,-T^{-}r^2,T^{+}r^2)
			\subset
			P^*(x,t;Lr,-T^{-}r^2,T^{+}r^2).
		\end{equation*}
	
		\item If $|\Ric|\leq Kr^{-2}$ on $P^*(x,t;Lr,-T^{-}r^2,T^{+}r^2)$ for some $L\geq\underline{L}(Ar^2,\lambda,K,T^{\pm})$, then
		\begin{equation*}
			P(x,t;\lambda r,-T^{-}r^2,T^{+}r^2)
			\subset
			P^*(x,t;Lr,-T^{-}r^2,T^{+}r^2).
		\end{equation*}
		
		\item If $|\Ric|\leq Kr^{-2}$ on $P(x,t;Lr,-T^{-}r^2,T^{+}r^2)$ for some $L\geq\underline{L}(Ar^2,\lambda,K,T^{\pm})$, then
		\begin{equation*}
			P^*(x,t;\lambda r,-T^{-}r^2,T^{+}r^2)
			\subset
			P(x,t;Lr,-T^{-}r^2,T^{+}r^2).
		\end{equation*}
	\end{enumerate}
\end{Corollary}
\begin{proof}
	Assertion (1) follows from \hyperref[cor:concentration-under-Ricci-bound]{Corollary~\ref*{cor:concentration-under-Ricci-bound}} and a standard distance distortion estimate.
	
	Next we prove (2). We temporarily denote the critical constant $\underline{L}$ in assertion (1) by $L_1$. Choose $r^*\in(0,r]$ maximal with the property that $|\Ric|\leq Kr^*{}^{-2}$ on $P(x,t;L_1r^*,-T^{-}r^*{}^2,T^{+}r^*{}^2)$. If $r^*=r$, then we are done by our previous conclusion (1). If $r^*<r$, then another application of assertion (1) with $\lambda$ replaced by $L_1$ yields
	\begin{equation*}
		P(x,t;L_1r^*,-T^{-}r^*{}^2,T^{+}r^*{}^2)
		\subset
		P^*(x,t;L'r^*,-T^{-}r^*{}^2,T^{+}r^*{}^2),
	\end{equation*}
	provided $L'\geq\underline{L}(Ar^2,L_1,K,T^\pm)=\underline{L}(Ar^2,\lambda,K,T^{\pm})$. Set $L=e^{\frac{A}{2}r^2T^{-}}L'$. By \hyperref[prop:containment-relationship]{Proposition~\ref*{prop:containment-relationship}}(1),
	\begin{equation*}
		P^*(x,t;L'r^*,-T^{-}r^*{}^2,T^{+}r^*{}^2)
		\subset
		P^*(x,t;Lr,-T^{-}r{}^2,T^{+}r{}^2),
	\end{equation*}
	The above two relations imply $|\Ric|\leq Kr^{-2}\leq Kr^*{}^{-2}$ on $P(x,t;L_1r^*,-T^{-}r^*{}^2,T^{+}r^*{}^2)$, which contradicts the maximal choice of $r^*$.
	
	Now we prove (3). By parabolic rescaling, we may assume $r=1$. We first show that we can assume without loss of generality that $T^{-}=0$. For any $(y,s)\in P^*(x,t;\lambda,-T^{-},T^{+})$, \hyperref[cor:concentration-under-Ricci-bound]{Corollary~\ref*{cor:concentration-under-Ricci-bound}} gives
	\begin{align*}
		d_{W_1}^{g(t-T^{-})}(\delta_x,\nu_{y,s;t-T^{-}})
		&\leq
		d_{W_1}^{g(t-T^{-})}(\delta_x,\nu_{x,t;t-T^{-}})
		+d_{W_1}^{g(t-T^{-})}(\nu_{x,t;t-T^{-}},\nu_{y,s;t-T^{-}})\\
		&\leq
		C(A,K,T^{-})+\lambda.
	\end{align*}
	Therefore
	\begin{equation*}
		P^*(x,t;\lambda,-T^{-},T^{+})
		\subset
		P^*(x,t-T^{-};C(A,K,T^{-})+\lambda,0,T^{-}+T^{+}).
	\end{equation*}
	By standard distance distortion estimates,
	\begin{equation*}
		P(x,t;L,-T^{-},T^{+})
		\supset
		P(x,t-T^{-};c(A,K,T^{-})L,0,T^{-}+T^{+}).
	\end{equation*}
	Hence it suffices to prove the case $T^-=0$. We may also assume $t=0$. In sum, we only need to show that if $L\geq\underline{L}(A,\lambda,K,T^+)$ and $|\Ric|\leq K$ on $P(x,0;L,T^+)$, then
	\begin{equation*}
		P^*(x,0;\lambda,T^+)
		\subset
		P(x,0;L,T^+).
	\end{equation*}
	
	Choose $T^*\in[0,T^+]$ maximal with the property that for all $T'\in[0,T^*)$ we have
	\begin{equation*}
		P^*(x,0;\lambda,T')
		\subset
		P(x,0;\tfrac{1}{2}L,T').
	\end{equation*}
	Assuming $L>2\lambda$, we must have $T^*>0$. Consider $(y,s)\in P^*(x,0;\lambda,T')$. Then for $L\geq\underline{L}(A,K,T^+)$, the distance distortion estimate implies
	\begin{equation*}
		P(y,s;1,-s)
		\subset
		P(x,0;L,T').
	\end{equation*}
	Thus by \hyperref[cor:concentration-under-Ricci-bound]{Corollary~\ref*{cor:concentration-under-Ricci-bound}},
	\begin{equation*}
		d_0(x,y)
		=d_{W_1}^{g(0)}(\delta_x,\delta_y)
		\leq
		d_{W_1}^{g(0)}(\delta_x,\nu_{y,s;0})
		+d_{W_1}^{g(0)}(\nu_{y,s;0},\delta_y)
		\leq
		\lambda+C(A,K,T^+)
		\leq \tfrac{1}{4}L.
	\end{equation*}
	This implies $(y,s)\in P(x,0;\frac{1}{4}L,T')$. By continuity and the maximal property of $T^*$, we must have $T^*=T^+$.
\end{proof}

Next we discuss a volume upper bound on the time-slices of a $P^*$-parabolic neighborhood.
\begin{Theorem}[Volume of time-slices]\label{thm:volume-of-time-slices}
	Let $\alpha>0$, $\lambda,T^\pm\geq 0$, $(x_0,t_0)\in M\times I$ and $r>0$ be a scale with $[t_0-(T^{-}+\alpha)r^2,t_0]\subset I$. Set
	\begin{equation*}
		S_t:=P^*(x_0,t_0;\lambda r,-T^{-}r^2,T^{+}r^2)\cap(M\times\{t\}).
	\end{equation*}
	Then for any $t\in[t_0-T^{-}r^2,t_0+T^{+}r^2]$, the volume of the time-slice $S_t$ is bounded by 
	\begin{equation}\label{ineq:volume-of-time-slice}
		|S_t|_{g(t)}
		\leq
		C(n,\alpha,\lambda,T^{\pm},R_{\mathrm{min}}r^2,|E|_{\mathrm{max}}r^2)\exp(\mathcal{N}_{x_0,t_0}(T^{-}r^2))r^n.
	\end{equation}
	Moreover, for any $L>0$, the volume of the $Lr$-neighborhood of $S_t$ is bounded by
	\begin{equation}\label{ineq:volume-of-neighborhood-of-time-slice}
		|B(S_t,t,Lr)|_{g(t)}
		\leq
		C(n,\alpha,\lambda,L,T^{\pm},R_{\mathrm{min}}r^2,|E|_{\mathrm{max}}r^2)\exp(\mathcal{N}_{x_0,t_0}(T^{-}r^2))r^n.
	\end{equation}
\end{Theorem}
\begin{proof}
	By parabolic rescaling and time shifting we may assume $r=1$ and $t_0-T^{-}r^2=0$. Set $A=|E|_{\mathrm{max}}$. Then $E\geq -Ag$. Consider a point $x$ in $S_t$. Let $(z_0,0)$ (resp. $(z,0)$) be an $(H_n,A)$-center of $(x_0,t_0)$ (resp. $(x,t)$). Then
	\begin{align*}
		d_0(z_0,z)
		=d_{W_1}^{g(0)}(\delta_{z_0},\delta_z)
		&\leq
		d_{W_1}^{g(0)}(\delta_{z_0},\nu_{x_0,t_0;0})
		+d_{W_1}^{g(0)}(\nu_{x_0,t_0;0},\nu_{x,t;0})
		+d_{W_1}^{g(0)}(\nu_{x,t;0},\delta_z)\\
		&\leq
		\sqrt{H_n\sigma_A(t_0)}+\lambda+\sqrt{H_n\sigma_A(t)}
		\leq C(n,A,\lambda,T^{\pm}).
	\end{align*}
	By \hyperref[prop:basic-properties-of-center]{Proposition~\ref*{prop:basic-properties-of-center}}(3), there is a constant $D(n,A,\lambda,T^{\pm})$ such that for $B:=B(z_0,0,D)$ we have
	\begin{equation}\label{ineq:nu(B)}
		\nu_{x,t;0}(B)
		\geq
		\nu_{x,t;0}\big(B(z,0,\sqrt{2H_n\sigma_A(t)})\big)
		\geq
		\frac{1}{2}.
	\end{equation}
	Let $u\in C^\infty(M\times(0,t])$ solve the heat equation $\square u=0$ with initial condition $u(\cdot,0)=\mathds{1}_B$. Then (\ref{ineq:nu(B)}) implies $u\geq\frac{1}{2}$ on $S_t$. Now
	\begin{equation*}
		\frac{d}{ds}\int_M   u\,dV_g(s)
		=
		-\int_M   u(R-\tfrac{1}{2}\mathrm{tr}\,E)   \,dV_g(s)
		\leq
		-(R-\tfrac{1}{2}\mathrm{tr}\,E)_{\mathrm{min}}
		\int_M   u\,dV_g(s).
	\end{equation*}
	Hence
	\begin{equation}\label{ineq:volume-of-time-slice-temp}
		\frac{1}{2}|S_t|_t
		\leq
		\int_M   u\,dV_g(t)
		\leq
		\exp\big(-(R-\tfrac{1}{2}\mathrm{tr}\,E)_{\mathrm{min}}t\big)
		|B|_0.
	\end{equation}
	By \hyperref[cor:oscillation-of-pointed-Nash-entropy]{Corollary~\ref*{cor:oscillation-of-pointed-Nash-entropy}} and \hyperref[prop:time-derivative-of-pointed-Nash-entropy]{Proposition~\ref*{prop:time-derivative-of-pointed-Nash-entropy}},
	\begin{align*}
		\mathcal{N}^*_{-\alpha}(z_0,0)
		&\leq
		\mathcal{N}^*_{-\alpha}(x_0,t_0)
		+C(n,\alpha,T^{-},R_{\mathrm{min}},|E|_{\mathrm{max}})\\
		&\leq
		\mathcal{N}_{x_0,t_0}(\alpha)
		+C(n,\alpha,T^{-},R_{\mathrm{min}},|E|_{\mathrm{max}}).
	\end{align*}
	By \hyperref[thm:no-local-inflating]{Theorem~\ref*{thm:no-local-inflating}},
	\begin{equation*}
		|B|_0
		\leq
		C\exp(\mathcal{N}_{z_0,0}(\alpha))
		\leq
		C\exp(\mathcal{N}_{x_0,t_0}(\alpha))
	\end{equation*}
	for some $C(n,\alpha,\lambda,T^{\pm},R_{\mathrm{min}},|E|_{\mathrm{max}})$. Combining this with (\ref{ineq:volume-of-time-slice-temp}) gives (\ref{ineq:volume-of-time-slice}).
	
	Next we show (\ref{ineq:volume-of-neighborhood-of-time-slice}) can be reduced to (\ref{ineq:volume-of-time-slice}). For any $y\in B(S_t,t,Lr)$, there exists $(x,t)\in S_t$ with $d_t(x,y)<Lr$. By \hyperref[cor:monotonicity-of-W1]{Corollary~\ref*{cor:monotonicity-of-W1}},
	\begin{equation*}
		d_{W_1}^{g(t_0-T^{-}r^2)}
		(\nu_{x,t;t_0-T^{-}r^2},\nu_{y,t;t_0-T^{-}r^2})
		<
		\exp(\tfrac{A}{2}(T^{-}+T^{+})r^2)
		Lr.
	\end{equation*}
	Since $(x,t)\in S_t\subset P^*(x_0,t_0;\lambda r,-T^{-}r^2,T^{+}r^2)$, we also have
	\begin{equation*}
		d_{W_1}^{g(t_0-T^{-}r^2)}
		(\nu_{x,t;t_0-T^{-}r^2},\nu_{x_0,t_0;t_0-T^{-}r^2})
		<
		\lambda r.
	\end{equation*}
	Combining the above two inequalities, we get
	\begin{equation*}
		B(S_t,t,Lr)
		\subset
		P^*(x_0,t_0;(\exp(\tfrac{1}{2}(T^{-}+T^{+})Ar^2)L+\lambda)r,-T^{-}r^2,T^{+}r^2).
	\end{equation*}
\end{proof}

This volume upper bound can be used to establish the following covering result.
\begin{Theorem}
	Suppose $E(g)\geq -Ag$ for some $A\geq 0$. Let $\alpha_0>0$, $\lambda,T^\pm\geq 0$, $(x_0,t_0)\in M\times I$ and $r>0$ be a scale with $[t_0-(T^{-}+\alpha_0^2)r^2,t_0]\subset I$.
	
	Then for any subset $X\subset P^*(x_0,t_0;\lambda r,-T^{-}r^2,T^{+}r^2)$ and $\alpha\in(0,\alpha_0]$, we can find points $(y_1,s_1),\cdots,(y_N,s_N)\in X$ with the property that
	\begin{equation}\label{rel:covering}
		X\subset\bigcup_{i=1}^N   P^*_A(y_i,s_i;\alpha r),
		\qquad
		N\leq C(n,\alpha_0,\lambda,T^{\pm},R_{\mathrm{min}}r^2,|E|_{\mathrm{max}}r^2)
		\alpha^{-n-2}.
	\end{equation}
\end{Theorem}
\begin{proof}
	We may assume $\alpha\leq\frac{1}{3}\alpha_0$ and it suffices to prove (\ref{rel:covering}) with $\alpha$ replaced by $3\alpha$. After parabolic rescaling, we may also assume $r=1$. Choose a maximal collection of points
	\begin{equation*}
		(y_1,s_1),\cdots,(y_N,s_N)\in X
	\end{equation*}
	with
	\begin{equation*}
		P^*_A(y_i,s_i;\alpha)
		\cap
		P^*_A(y_j,s_j;\alpha)
		=\varnothing
		\quad\text{for all }i\neq j.
	\end{equation*}
	Then for any $(y',s')\in X$, we have $P^*_A(y',s';\alpha)\cap P^*_A(y_i,s_i;\alpha)\neq\varnothing$ for some $i$ by the maximality. By \hyperref[prop:containment-relationship]{Proposition~\ref*{prop:containment-relationship}}(6), $P^*_A(y',s';\alpha)\subset P^*_A(y_i,s_i;3\alpha)$. Therefore
	\begin{equation*}
		X\subset\bigcup_{i=1}^N   P^*_A(y_i,s_i;3\alpha).
	\end{equation*}

	It remains to derive an upper bound on $N$. Let $\beta\in(0,\frac{1}{2}]$ be a constant to be determined later. Consider the function over $[t_0-T^{-}-2\beta\alpha^2,t_0+T^{+}-\beta\alpha^2]$ defined by
	\begin{equation*}
		t\mapsto
		\#\{i:t\in[s_i-2\beta\alpha^2,s_i-\beta\alpha^2]\}.
	\end{equation*}
	Let $t^*$ be a time at which the above function takes its maximum $k$. Then
	\begin{equation}\label{ineq:temp-estimate-on-number}
		\beta\alpha^2   \cdot   N
		\leq
		k   \cdot   (T^{-}+T^{+}+\beta\alpha^2).
	\end{equation}
	By relabeling the points, we may assume $t^*\in [s_i-2\beta\alpha^2,s_i-\beta\alpha^2]$ for $1\leq i\leq k$. In what follows, we only consider these first $k$ points. Let $(z_i,t^*)$ be an $(H_n,A)$-center of $(y_i,s_i)$. Then for all $i\neq j$, we have
	\begin{align*}
		d_{t^*}(z_i,z_j)
		&\geq
		d_{W_1}^{g(t^*)}(\nu_{y_i,s_i;t^*},\nu_{y_j,s_j;t^*})
		-d_{W_1}^{g(t^*)}(\delta_{z_i},\nu_{y_i,s_i;t^*})
		-d_{W_1}^{g(t^*)}(\delta_{z_j},\nu_{y_j,s_j;t^*})\\
		&\geq
		d_{W_1}^{g(s_i-\alpha^2)}(\nu_{y_i,s_i;s_i-\alpha^2},\nu_{y_j,s_j;s_i-\alpha^2})
		-2\sqrt{H_n\sigma_A(2\beta\alpha^2)}\\
		&\geq
		\alpha e^{A\alpha^2}
		-2\sqrt{H_n\sigma_A(2\beta\alpha^2)}
		\geq \tfrac{1}{2}\alpha,
	\end{align*}
	provided $\beta\leq\overline{\beta}(n,A,\alpha_0)$. So the balls $B_i:=B(z_i,t^*,\frac{1}{4}\alpha)$ are pairwise disjoint. By \hyperref[thm:no-local-collapsing-around-center]{Theorem~\ref*{thm:no-local-collapsing-around-center}}, \hyperref[prop:time-derivative-of-pointed-Nash-entropy]{Proposition~\ref*{prop:time-derivative-of-pointed-Nash-entropy}} and \hyperref[cor:oscillation-of-pointed-Nash-entropy]{Corollary~\ref*{cor:oscillation-of-pointed-Nash-entropy}},
	\begin{align}
		|B_i|_{t^*}
		&\geq
		c(n,\alpha_0,R_{\mathrm{min}},|E|_{\mathrm{max}})
		\exp(\mathcal{N}_{y_i,s_i}(\alpha^2))\alpha^n  \notag  \\
		&\geq
		c(n,\alpha_0,R_{\mathrm{min}},|E|_{\mathrm{max}},\lambda,T^{\pm})
		\exp(\mathcal{N}_{x_0,t_0}((\tfrac{\alpha_0}{3})^2))\alpha^n\label{ineq:lower-volume-bound-on-balls}
	\end{align}
	for sufficiently small $\beta$. On the other hand, \hyperref[thm:volume-of-time-slices]{Theorem~\ref*{thm:volume-of-time-slices}} implies that for
	\begin{equation*}
		S_{t^*}
		:=
		P^*(x_0,t_0;\lambda,-T^{-},T^{+})\cap(M\times\{t^*\})
	\end{equation*}
	we have
	\begin{equation}\label{ineq:temp-volume-upper-bound-neighborhood}
		|B(S_{t^*},t^*,\tfrac{1}{4}\alpha_0)|_{t^*}
		\leq
		C(n,\alpha_0,\lambda,T^{\pm},R_{\mathrm{min}},|E|_{\mathrm{max}})
		\exp(\mathcal{N}_{x_0,t_0}((\tfrac{\alpha_0}{3})^2)).
	\end{equation}
	Combining (\ref{ineq:temp-estimate-on-number}), (\ref{ineq:lower-volume-bound-on-balls}) and (\ref{ineq:temp-volume-upper-bound-neighborhood}), we obtain
	\begin{equation*}
		N
		\leq
		\frac{T^{-}+T^{+}+\beta\alpha^2}{\beta}\cdot \alpha^{-2}k
		\leq
		C(n,\alpha_0,\lambda,T^{\pm},R_{\mathrm{min}},|E|_{\mathrm{max}})
		\alpha^{-n-2}.
	\end{equation*}
\end{proof}

\newpage
\bibliography{pseudolocality}
\end{document}